\documentclass[12pt,leqno]{amsart}
\makeatletter
\AtBeginDocument{%
  \let\original@@tocwrite=\@tocwrite
  \newif\ifAHVflag
  \def\jlreq@uniqtoken{\jlreq@uniqtoken}
  \def\jlreq@endmark{\jlreq@endmark}
  \long\def\jlreq@getfirsttoken#1#{\jlreq@getfirsttoken@#1\bgroup\jlreq@endmark}
  \long\def\jlreq@getfirsttoken@#1#2\jlreq@endmark#3\jlreq@endmark{#1}
  \renewcommand{\@tocwrite}[2]{%
    \begingroup
      \AHVflagfalse 
      \@ifempty{#2}{}{%
        \expandafter\expandafter\expandafter\ifx\jlreq@getfirsttoken#2\jlreq@uniqtoken{}\jlreq@endmark\Sectionformat\expandafter\@firstoftwo\else\expandafter\@secondoftwo\fi
        {%
          \def\Sectionformat##1##2{\@ifempty{##1}{}{\AHVflagtrue}}%
          #2
        }{\AHVflagtrue}%
       }%
      \def\@tempa{}%
      \ifAHVflag\def\@tempa{\original@@tocwrite{#1}{#2}}\fi
    \expandafter\endgroup
    \@tempa
  }%
}
\makeatother

\usepackage{a4wide}
\usepackage{amssymb}
\usepackage{amsmath}
\usepackage{amsthm}
\usepackage[all]{xy}
\usepackage[usenames,dvipsnames]{xcolor}
\usepackage{lmodern}
\usepackage[T1]{fontenc}

\newcommand{\im}{\operatorname{Im}}

\newcommand{\charf}{\operatorname{char}}

\newcommand{\Hom}{\operatorname{Hom}}

\newcommand{\Ind}{\operatorname{Ind}}

\newcommand{\ind}{\operatorname{ind}}

\newcommand{\Ker}{\operatorname{Ker}}
\newcommand{\End}{\operatorname{End}}

\newcommand{\Gal}{\operatorname{Gal}}

\newcommand{\Lie}{\operatorname{Lie}}

\newcommand{\val}{\operatorname{val}}

\newcommand{\St}{\operatorname{St}}

\DeclareMathOperator{\SL}{SL}

\DeclareMathOperator{\diag}{diag}

 \DeclareMathOperator{\Ad}{Ad}
\DeclareMathOperator{\tr}{tr}
\DeclareMathOperator{\vol}{vol}

\DeclareMathOperator{\Fr}{Fr}
 
 \DeclareMathOperator{\Irr}{Irr}

\DeclareMathOperator{\red}{red}

    \DeclareMathOperator{\Gr}{Gr}

     \DeclareMathOperator{\st}{st}

\theoremstyle{plain} 
\newtheorem{theorem}{Theorem}[section]
\newtheorem{corollary}[theorem]{Corollary}
\newtheorem{lemma}[theorem]{Lemma}
\newtheorem{proposition}[theorem]{Proposition}

\newtheorem{definition-proposition}[theorem]{Definition-Proposition}

\theoremstyle{definition}
\newtheorem{definition}[theorem]{Definition}

\theoremstyle{remark}
\newtheorem{remark}[theorem]{Remark}

\newtheorem{notation}[theorem]{Notation}

\numberwithin{equation}{section}

\title{Representations of $SL_2(F)$}

\author{Guy Henniart and  Marie-France Vign\'eras  
}
\address
[Guy Henniart]{Facult\'e des Sciences d'Orsay, Universit\'e Paris-Saclay 
F-91405 Orsay Cedex}
\email{guy.henniart@math.u-psud.fr}
\address[M.-F. Vign\'eras]{Institut de Math\'ematiques de Jussieu,  place Jussieu, Paris 75005 France}
\email{marie-france.vigneras@imj-prg.fr}

\date{\today}

\begin{document} 

 \begin{abstract}Let $p$ be a prime number, $F $ a  non-archimedean local field with residue field $k_F$ of characteristic $p$, and $R$  an algebraically closed field of characteristic  different from $ p$. We thoroughly investigate the irreducible smooth $R$-representations of   $SL_2(F)$. The components of an  irreducible smooth $R$-representation $\Pi$ of   $GL_2(F)$  restricted to  $SL_2(F)$  form an $L$-packet $L(\Pi)$. We use the classification  of such $\Pi$ to determine  the cardinality of $L(\Pi)$, which is $1,2$ or $4$.  When $p=2$ we have to use the Langlands correspondence for $GL_2(F)$. When $\ell$  is a prime number distinct from $p$ and $R=\mathbb Q_\ell^{ac}$, we establish the behaviour of an integral $L$-packet under reduction modulo $\ell$. We  prove a  Langlands correspondence for $SL_2(F)$, and even an enhanced one when the characteristic of $R$ is not $2$. Finally, pursuing a theme of \cite{HV23}, which studied the case of inner forms of $GL_n(F)$, we show that near identity a non-trivial irreducible smooth R-representation $\pi$
 of $SL_2(F)$ is, up to  a finite dimensional representation, isomorphic to a sum of $1,2$ or $4$ representations in an $L$-packet of size $4$ (when $p$ is odd there is only one such $L$-packet).  We show that for  $\pi$ in an $L$-packet of size $r_\pi$ and a  sufficiently large integer $j$,
 the dimension of the invariants of $\pi$   by the $j$-th congruence subgroup of an Iwahori or a pro-$p$ Iwahori subgroup of $SL_2(F)$ is  equal to $a_\pi+2 r_\pi^{-1} |k_F|^j$, with $a_\pi=-1/2$ if  $p$ is odd and $r_\pi=4$, otherwise  $a_\pi$  is an integer. We also study the fixed points by the $j$-th congruence subgroups of the maximal compact subgroups of $SL_2(F)$ where the answer depends on the parity of $j$.

  \end{abstract}
  \maketitle
 
\setcounter{tocdepth}{2}  
\tableofcontents
  \section{Introduction}
\subsection{}  Let  $F $ be a locally compact non-archimedean field  with residue characteristic $p$ and $R$ an algebraically closed field of characteristic $\charf_R$ different from $ p$.  We thoroughly investigate the irreducible smooth $R$-representations of   $SL_2(F)$.  
Although when $R=\mathbb C$ and $p$ odd the  first investigations appeared in the 1960's, in work of Gelfand-Graev and Shalika, the study of the modular case (i.e. when $\charf_R>0$) started only recently \cite{Cui23}, \cite{CLL23} when $\charf_F\neq 2$ and $ \charf_R\neq 2$. Here we give a complete treatment and we make no assumption on $p,
\charf_F, \charf_R$, apart from $\charf_R \neq  p$. 

As Labesse and Langlands did in the 1970's when $R=\mathbb C$ and $\charf_F=0$, we use the restriction of
smooth $R$-representations from $G=GL_2(F)$ to $G'=SL_2(F)$.  We prove that an irreducible smooth $R$-representation
of $G'$ extends to a smooth representation of an open subgroup $H$ of $G$ containing $ZG'$ where $Z$ is the
centre of $G$, and appears in the restriction to $G'$ of a smooth irreducible $R$-representation of $G$,
unique up to isomorphim and twist by smooth $R$-characters of $G/G'$.   
 When $\charf_F\neq 2$ we can simply take  $H=ZG'$, but not    when  $\charf_F=2$  because the compact quotient $G/ZG'$ is infinite.
  Those results follow from general facts about $R$-representations, which appear in
Section \ref{S1}. They apply to more general reductive groups over $F$, as we show in Section \ref{S:2.3}.

In Section \ref{SL2},  using Whittaker models, we show that the restriction to $G'$ of an irreducible smooth $R$-representation
$\Pi$  of $G$  is semi-simple and has finite length and multiplicity one. Its irreducible components form an $L$-Lacket $L(\Pi)$. An $L$-packet $L(\Pi)$ is called cuspidal when $\Pi$  is cuspidal, supercuspidal when $\Pi$  is supercuspidal, of level $0$  if $\Pi$  can be chosen to have level $0$ (that is, having
non-zero fixed vectors under $1+M_2(P_F)$), and of positive level otherwise.  
 
\begin{theorem}   \label{lpi} The size of an $L$-packet is $1,2$ or $4$.
 \end{theorem}
 When $p$ is odd that follows rather easily
from $|G/ZG'|=4$, but it is also true when $p=2$, in which case the proof is completed only in  Proposition \ref{cor:bq1}, and  
uses  the Langlands $R$-correspondence  for $G$, which we recall in \S \ref{pas3}.

 \begin{proposition}\label{cor:bq} (Corollary  \ref{cor:tout}, Proposition \ref{cor:bq1})  The $L$-packets of size $4$  are  cuspidal  and in bijection the  biquadratic separable extensions of $F$.    \end{proposition}

The bijection is described in the proof.
 When $p\neq 2$  there is just one  $L$-packet of size $4$ and it has level $0$. When $p=2$ the $L$-packets of size $4$ have positive level,  their 
number is finite if $\charf_F=0$, but there are infinitely many 
if $\charf_F=2$. 
     \begin{proposition} \label{prop:ty} (Proposition \ref{prop:ty1})       When  $p$ is odd, the   cuspidal $L$-packets  are not singletons.
When $p=2$, the cuspidal $L$-packets  of level $0$ have size $2$.   \end{proposition}

\begin{proposition} \label{prop:ncl} (Proposition \ref{pro:cnsc}) There is a cuspidal non-supercuspidal $L$-packet if and only if  $q+1=0$ in $R$. It is unique  of level $0$, and  size  $4$ when  $\charf_R=2$, and  size $2$ when $\charf_R\neq 2$.
 \end{proposition}

\bigskip From the Langlands $R$-correspondence for $GL(2,F)$, we get a bijection  from the set of $L$-packets to the set of conjugacy classes of Deligne morphisms of  $W_F$ into $PGL_2(R)$, the dual group of $ SL_2$ over $R$.  When $\charf_R\neq 2$, we even get an enhanced Langlands correspondence, in that we parametrize 
the elements in an $L$-packet  $L(\Pi)$  by the characters of the group $S_\Pi$ of connected components of
the centralizer $C_\Pi$ of  the image of the corresponding  Deligne morphism in $PGL_2(R)$.
When $\charf_R= 2$,  $C_\Pi$  is always connected and  the  supercuspidal $L$-packets  are not singletons.
 We will determine explicitely $C_\Pi$ for each $\Pi$.

\begin{theorem}\label{thm:LLSL2} (Theorem \ref{prop:LLSL2})\footnote{When $R=\mathbb C$  this was already
established by Gelbart and Knapp \cite[\S 4]{GeKn82} assuming that it
could be done for $GL_n(F)$}. Let $\Pi$ be  an irreducible smooth $R$-representation of $GL_2(F)$.

 When $\charf_R\neq 2$,  the component group $S_\Pi $ of  $C_\Pi$  is isomorphic to $\{1\}, \mathbb Z/ 2 \mathbb Z$ or $ \mathbb Z/ 2 \mathbb Z\times \mathbb Z/ 2 \mathbb Z$.
 
  When $\charf_R=2$, $C_\Pi$  is connected for each $\Pi$,  but the cardinality of the $L$-packet $L(\Pi)$ is 

$1$ if $\Pi$ is not cuspidal,

$2$ if $\Pi$ is supercuspidal,

$4$ if $\Pi$ is   cuspidal not supercuspidal. \end{theorem}

When $L(\Pi)$ is not a singleton, we take as a
base point the element having a non-zero Whittaker model with respect to a non-trivial
smooth $R$-character of $F$.  When $\charf_R\neq 2$, the theorem gives a bijection $$\iota:L(\Pi)\to \Irr_R(S_\Pi)$$ 
respecting the base points (the trivial representation in $ \Irr_R(S_\Pi)$).
It is unique when $|L(\Pi)|=2$. There are  partial results   on the unicity of $\iota$  when $|L(\Pi)|=4$. Under the restriction $p=2, \charf_F=0$, for  the complex $L$-packet of size $4$ (unique, of level $0$), there is a unique bijection  compatible with the endoscopic character identities \cite{AP24}.

\medskip When $\charf_R=2$, Treuman and Venkatesh introduced a  ``linkage'' between irreducible smooth $R$-representations of $G$ and $G'$. In \S  \ref{S:464} we interpret this notion in terms of dual groups, thus proving their conjectures in a special case.

\medskip   Let $\ell\neq p$ be a prime number, and $\mathbb Q_\ell^{ac}$  an algebraic closure of $\mathbb Q_\ell$ with residue field  $\mathbb F_\ell^{ac}$.  
Each irreducible  smooth  $\mathbb F_\ell ^{ac}$-representation  of $GL_2(F)$ lifts to a smooth $\mathbb Q_\ell^{ac}$-representation. We show that this remains true for $SL_2(F)$. 

    \begin{proposition}  \label{prop:nc} (Corollary \ref{prop:redu1},  Proposition \ref{prop:nc1})  
   
     Each irreducible  smooth  $\mathbb F_\ell ^{ac}$-representation $\pi$ of $SL_2(F)$  is the reduction modulo $\ell$ of  an integral  irreducible  smooth   $\mathbb Q_\ell^{ac}$-representation $\tilde \pi$ of $SL_2(F)$. 
     \end{proposition} 

An equivalent formulation is that each  irreducible  smooth  $\mathbb F_\ell ^{ac}$-representation $\Pi$ of $GL_2(F)$ is the reduction modulo $\ell$ of an integral irreducible smooth  $\mathbb Q_\ell ^{ac}$-representation $\tilde \Pi$ of $GL_2(F)$  such that 
$$|L(\Pi)|=  |L(\tilde \Pi)|.$$
     
The reduction modulo $\ell$ of each  integral supercuspidal    $\mathbb Q_\ell^{ac}$-representation  of $GL_2(F)$  is  irreducible, but this is not true for $SL_2(F)$.  Each supercuspidal $\mathbb Q_\ell^{ac}$-representation   $\tilde \pi$ of $SL_2(F)$ is integral and we determine all the cases of reducibility. We choose a supercuspidal    $\mathbb Q_\ell^{ac}$-representation  $\tilde \Pi$  of $GL_2(F)$  such that  $\tilde \pi\in L(\tilde \Pi)$ and  denote by  $\sigma_{\tilde \Pi}$ the  irreducible $2$-dimensional $\mathbb Q_\ell^{ac}$-representation   of $W_F$  image of $\tilde \Pi$  by the local Langlands correspondence.

\begin{proposition} \label{prop:redu}   (Corollary  \ref{prop:redu1})  The  reduction modulo $\ell$  of   $\tilde \pi$  has length $\leq 2$. The length is $2$ if and only if
$$p=2, \  \sigma_{\tilde \Pi}=\ind_{W_E}^{W_F} \tilde \xi,  \  \tilde \xi (b)\neq 1, \  \tilde \xi (b)^{\ell^{s}}=1, \  \ell^s  \   \text{divides $ q+1$, the order of
  $(\tilde \xi ^\tau/ \tilde \xi )|_{1+P_E}$ is $2$}, $$     where $b $ is 
  a root of unity  of order $q+1$ in a quadratic unramified extension $E/F$, $\tilde \xi $ is a smooth $\mathbb Q_\ell^{ac}$-character of  $E^*$ (of $W_E$  via class field theory), and  $\tau\in \Gal(E/F)$ is not trivial.\end{proposition}

\medskip 
Finally we study for $G'$ the problem that we treated in \cite{HV23}  for inner forms of $GL_n(F)$. An infinite dimensional  irreducible smooth $R$-representation $\Pi$ of $G=GL_2(F)$ is isomorphic near the identity to  $ a_\Pi 1+\ind_B^G1$ where $a_\Pi$ is an integer (its value is given in Proposition \ref{pro7.5}) and $\ind_B^G1$ is the usual principal series.
For an infinite dimensional irreducible smooth $R$-representation $\pi$   of $G'$, we show that up to finitely many trivial $R$-characters,
$\pi$    is isomorphic 
 near the identity to the sum of $1, 2$ or $4$ elements of an $L$-packet of size $4$.

 \begin{theorem} \label{th:anyp} (Theorem \ref{th:anyp1})  Let $\pi$ be an infinite dimensional irreducible smooth  $R$-representation of $G'$.
There   are  irreducible smooth  $R$-representations  $ \{\tau_1, \tau_2, \tau_3, \tau_4 \}$  of $G'$ forming an $L$-packet, and an   integer $a_0 $,  such that on a small enough compact open subgroup $K $ of $G'$ we have 
$$\pi \simeq a_0 1 + \sum_{ i=1} ^{4/r} \tau_i,$$
where $r$ is the  size of the $L$-packet containing $\pi$.
\end{theorem} 
  For $R=\mathbb C$ and $p$ odd, Monica Nevins has similar  results which are more precise in that the subgroup $K$ is large.  We show that her results carry over to any $R$ (\S \ref{S:Nevins}). 
  
  \medskip As in \cite{HV23} we first deal with the case where $R=\mathbb C$, using a germ expansion near the identity  \`a la Harish-Chandra,
in terms of nilpotent orbital integrals. However, when $\charf_F=2$, such an expansion is not available,
so we work instead with a complex representation $\pi$  of  an open subgroup $H$ of $G$ containing $ZG'$.
For such a group a germ expansion has been obtained by Lemaire \cite{L04}. Adapting \cite{MW87}
and \cite{Va14} (who assumed $\charf_F=0$) we compute the germ expansion in terms of the dimensions
of the different Whittaker models of $\pi$, and express it in terms of $L$-packets of size $4$.  Theorem \ref{th:anyp} easily transfers to any $R$ with $\charf_R=0$, in particular $R=\mathbb Q_\ell^{ac}$. From  our complete classification of irreducible smooth $R$-representations of $G'$, and in particular that the $\mathbb F_\ell^{ac}$-representations of $G'$ lift to characteristic $0$ when $\ell\neq p$  (Proposition \ref{prop:nc}),  we get Theorem \ref{th:anyp}  for $R=\mathbb F_\ell^{ac}$ and transfer it to any $R$ with $\charf_R=\ell$.

  We think that Theorem \ref{th:anyp}  will extend in the same way to inner forms of $SL_n$, using the work of \cite{HS12}. We expect that if $\charf_F=0$ and $R=\mathbb C$, a variant of  the theorem is true for any connected reductive $F$-group $\underline H$, because of the Harish-Chandra germ expansion and of the work of Moeglin-Waldspurger and Varma. But  when $\ell\neq p$, it is not known  in general if   virtual finite length  $\mathbb F_\ell^{ac}$-representations lift to characteristic $0$ and it is unlikely that cuspidal irreducible $\mathbb F_\ell^{ac}$-representations lift. The reason is that the  first  point has a positive answer when $G$ is a finite group and the answer to the second is negative in general for finite reductive groups. Moreover when  $ \charf_F= p$  and $R=\mathbb C$, we have to face the problem that a germ expansion in terms of nilpotent orbital integrals might not exist. It is not clear how to define such integrals for bad primes, and sometimes the number of unipotent orbits in $H$ and of nilpotent orbits in $\Lie(H)$ are not the same, even over an algebraic closure of $F$. Given our investigation of the case $SL_2(F)$, which uses $L$-indistinguishability, one may wonder about the role of endoscopy and stability in analogous results for a general $H$.

\bigskip  The dimension of the invariants by the $j$-th congruence subgroup of a Moy-Prasad group of  an infinite dimensional irreducible smooth $R$-representation  of $G $  for $j$  large,  is  the value at $q^j$ of a polynomial of degree  $1$ and   integral coefficients.   We will prove a similar result  for  $G'$  but the coefficients of the polynomial are  not always  integral and the polynomial may depend on the parity of $j$.

Let $\Pi$ be an infinite dimensional  irreducible smooth $R$-representation  of $G $ and $\pi$ be an  element of $L(\Pi)$.  Around the identity, $$ \Pi\simeq a_\Pi 1+\ind_B^G 1$$ for  an integer $a_\Pi$ and the usual principal series $\ind_B^G 1$. 
Let  $O_F$ denote the ring of integers of $F$, $K'=SL_2(O_F)$, $I'$ its Iwahori subgroup, $I'_{1/2}$ its pro-$p$ Iwahori, and $K'_j, I'_j, I'_{1/2+j}$ their $j$-th congruence subgroups. 
  
\begin{theorem} (Theorem \ref{th.pol}) For a sufficiently large $j$,
$$\dim_R \pi^{I'_j}=\dim_R \pi^{I'_1/2+j}=|L(\Pi)|^{-1} (a_\Pi + 2 q^j).$$
$$\dim_R \pi^{K'_j}= |L(\Pi)|^{-1} (a_\Pi + (q+1) q^{j-1}) \ \ \ \text{if} \ \Pi|_{ZKG'} \ \text{is irreducible}.$$
When $p$ is odd and $|L(\Pi)|=4$, we have $|L(\Pi)|^{-1} a_\Pi=-1/2$.
\end{theorem}
 
When $\Pi|_{ZKG'} $ is reducible, it has length $2$. The  two irreducible components $\Pi^+$ and $\Pi^-$ are distinguished by their Whittaker models.
\begin{theorem} (Corollary \ref{cor.pol})  If $\Pi|_{ZKG'} $ is reducible, for a sufficiently large $j$,
$$\dim_R \pi^{K'_j}=
 \begin{cases} |L(\Pi)|^{-1} (a_\Pi + 2 q^j)  &  \text{ for $j$ odd and $\pi \subset \Pi^+|_{G'}$, or $j$ even and $\pi \subset \Pi^-|_{G'}$}\\
L(\Pi)|^{-1} (a_\Pi + 2 q^{j-1}) &  \text{ otherwise}
 \end{cases}.$$
\end{theorem}

By $G$-conjugation, we have similar asymptotics for all Moy-Prasad subgroups de $G'$.

\medskip The study of $R$-representations of $G'$  has a long history, especially when $R=\mathbb C$. Even for odd $p$ and $R=\mathbb C$, there is current research on $GL_2$ and $SL_2$ \cite{Ngo24}. Inevitably some of our proofs
are adapted from previous papers. However, because we make only the assumption that $\charf_R\neq  p$, we have usually
preferred to give complete proofs in that general setting. We refer essentially only to papers that we are using.

\medskip We thank  Anne-Marie Aubert, Don Blasius, C\'edric Bonnaf\'e,  Jean-Francois Dat, Jean-Pierre Labesse, Bertrand Lemaire, Monica Nevins,  Dipendra Prasad, and Akshay Venkatesh,  for helpful communications.
Special thanks are due to Peiyi Cui, for  conversations and correspondences. Some arguments using types in \S \ref{S:331} are due to her,
 she called our attention to the study of the quotient in \eqref{eq:red'},  and shared her  results  on principal series. She also kindly 
gave us  references to her works \cite{Cui23}, \cite{CLL23}.

  \section{Generalities} \label{S1}

\subsection{}Let $R$ be a  field,  $G$  a group, $H$ a subgroup of   $G$,  $V$ an  $R$-representation of $G$. We denote $\charf_R$ the characteristic of $R$, and  $V|_H$ the restriction of $V$ to $H$.

\subsubsection{}  When $H$ has {\sl finite index} in $G$,  any irreducible $R$-representation  of $H$  is contained in the restriction  to $H$ of an irreducible $R$-representation  of $G$   \cite[Proposition 2.2]{H01}.

\subsubsection{}\label{2.1.2}  If $H$ is {\sl normal of finite index} in $G$ and  $V$ is irreducible, then  $V|_H$  is   semisimple   of finite length  (loc.cit. Proposition 2.1).

\subsubsection{}   If  $H$ is {\sl normal} in $G$,  $V$ is irreducible and      $V|_H$   contains  an irreducible subrepresentation, then $V|_H$ is semisimple and its isotypic components are $G$-conjugate with the same multiplicity.

  \begin{proof}   Let  $W$ be an irreducible subrepresentation of $V|_H$. Since $H$ is normal in $G$, 
  for  $g\in G$,  $H$ acts irreducibly on $gW$  by $(h,gw)\to hgh^{-1} hw$.  The  subspace $\sum_{g\in G}gW$  is a non-zero subrepresentation of $V$. Since $V$ is irreducible, it is equal to $V$.  Since a representation which is a sum of irreducible subrepresentations is semi-simple \cite[\S 4.1 Corollary 1]{BkiA8}, $V|_H$ is semisimple. The last assertion follows in the same way. \end{proof}

\subsubsection{} \label{2.1.4} Assume $H $ {\sl normal of finite index} in $G$ and let $\pi$ be an irreducible $R$-representation of $H$. We saw that there is an irreducible $R$-representation $\Pi$ of $G$ whose restriction to $H$ (which is semisimple of finite length) contains $\pi$. Clearly if $\chi$ is a $R$-character of $G$ trivial on $H$ then  the restriction of $\Pi \otimes \chi$ to $H$ contains $\pi$.

\begin{lemma}\label{le:2.1}Assume $R$ algebraically closed and  $G/H $  {\sl abelian}.  Any  irreducible $R$-representation $\Pi'$ of $G$ containing $\pi$ is isomorphic to $\Pi \otimes \chi$ for some $R$-character $\chi$  of $G$ trivial on $H$.
\end{lemma}
\begin{proof}\footnote{This proof was suggested by Peyi Cui \cite[Proposion 2.6]{Cui23}, and replaces a more complicated argument of ours.} We have $\Hom _H (\Pi'|_H, \Pi|_H)\neq 0$. The  right adjoint of the restriction  from $G$ to $H$ is the induction $\Ind _H^G$ from $H$ to $G$, therefore 
$ \Pi'$ is isomorphic to an irreducible subrepresentation of $ \Ind_H^G (\Pi|_H)$.
We have $\Ind_H^G (\Pi|_H))\simeq (\Ind_H^G 1) \otimes \Pi$ because $G/H$ is finite, and  the irreducible subquotients of $ \Ind_H^G 1$ are the characters $\chi$ of $G$ trivial on $H$ because $R$ is algebraically closed.  
Therefore, there exists $\chi$ such that $\Pi'\simeq \Pi \otimes \chi$.
\end{proof}

 \subsection{}\label{lc} We suppose    that  $H$ is a {\sl closed} subgroup of   a {\sl locally profinite} group $G$ and $V$ is  an $R$-representation of $G$. 
 
 If the index of $H$ in $G$ is finite, then $H$ is open. Conversely, if $H$ is   open  cocompact in $G$, then the  index of $H$ in $G$ is finite.
 If $V$ is smooth (i.e. the $G$-stabilizer of any vector is open) then $V|_H$ is  smooth.  Conversely, if $H$ is open in $G$ and  $V|_H$  is    smooth (resp. admissible i.e. smooth and the dimension of the space  $V^K$ of $K$-fixed vectors of $V$ is finite, for any open compact subgroup $K\subset H$), then $V$  is  smooth (resp. admissible).

\medskip    We suppose also  from now  on that  $H$ is  {\sl normal  in  $G$ with a   compact quotient $G/H$ and that $V$ is  smooth} (so $V|_H$ is smooth). 
   
   \subsubsection{} If  $V$ is {\sl finitely generated}  then $V|_H$  is finitely generated    \cite[Lemma 4.1]{H01}. 
   
 \subsubsection{}  \label{ss:2.2.4} If   $V$ is   {\sl  irreducible}, any irreducible  subrepresentation  of  $V|_H$  (when there exists one)  extends to a  (smooth and irreducible) representation of an open subgroup of $G$ of finite index which is admissible if $V$ is (loc.cit.Proposition 4.4).
 
 \subsubsection{} \label{ss:2.2.3} If   $V$ is  {\sl  irreducible and  $V|_H$ contains an irreducible subrepresentation or is  noetherian } (any subrepresentation is finitely generated), then $V|_H$ is semisimple of finite length  (loc.cit. Th\'eor\`eme 4.2). 
 
\bigskip  We introduce the two properties : 
 \begin{align} \label{i}  \text{Any finitely generated admissible $R$-representation of $G$ has finite length} \\
 \label{ii}  \text{ Any finitely generated smooth $R$-representation of $H$ is noetherian}
  \end{align} 
 

 \subsubsection{} \label{qu} Let  $W$ be  an admissible irreducible $R$-representation of $H$. 
 
 \medskip  
 1)  If {\sl  (\ref{i})  and  (\ref{ii})   are true},  then $W$ is contained in   some irreducible admissible $R$-representation   of $G$ restricted to $H$ (loc.cit. Corollaire  4.6).

\medskip 
 2) If {\sl  (\ref{i}) is true}, then $W$ is a quotient of  some irreducible admissible $R$-representation  of $G$ restricted to $H$ (loc.cit. Th\'eor\`eme 4.5). 
 
We give a simple proof of 2) adapted from \cite[Proposition 2.2]{T92}. The smooth induction $\Ind_H^G W$ of $W$ to $G$ is admissible as $W$ is  and $G/H$ is compact \cite[I.5.6]{V96}. A finitely generated subrepresentation of   $\Ind_H^G W$ is admissible, hence   of finite length by \eqref{i}. So  $\Ind_H^G W$ contains an irreducible admissible representation $U$. The restriction to $H$ is the left adjoint  of the induction $\Ind_H^G$ hence $W$ is a quotient of  $U|_H$.

 \subsubsection{} \label{2.2.9}   
We denote by  
$$\text{  $X_V$   the group of  $R$-characters $\chi$  of $G$  trivial on $H$ such that $ V\otimes \chi \simeq V$.}$$
The  characters in  $ X_V$ are smooth by the following lemma.
  
 \begin{lemma} \label{le}  $V\otimes \chi$ is smooth if and only if $\chi$ is smooth.
\end{lemma}
\begin{proof}  
Let $v\in V $ a non-zero element. An open subgroup $K\subset G$ fixing $v$ in $V$,  fixes $v$ in $V\otimes \chi$ if and only if $\chi$ is trivial on $K$.  The lemma follows because $V$ is smooth.  \end{proof}

 \subsubsection{}\label{S:226} Assume also  that $V$ is {\sl  irreducible and $V|_H$ has finite length} (semi-simple by \eqref{ss:2.2.3} and its isotypic components are $G$-conjugate)\footnote{This subsection generalises \cite[Corollary 3.8.3]{Cui23}, \cite[Corollary 2.5]{T92}, \cite[1.6(iii)]{BK94}}.
 
   Let  $W$ be an irreducible component of  $V|_H$,  $\pi$ its isomorphism class,  $G_\pi$ the $G$-stabilizer of $\pi$.  Let 
   $V_\pi$  be the   $\pi$-isotypic component of $V|_H$. The $G$-stabilizer  of $V_\pi$ is $G_\pi$. The $G$-stabilizer  of $W$   is open in $G$ (because it contains the $G$-stabilizer of $v\in W$ non $0$ and $V$ is smooth) and is contained in  $G_\pi$. Both  have  finite index in $G$ ($G/H$ is compact) and $$V=\Ind_{G_\pi}^G( V_\pi)$$ 
 by Clifford's theory.   
   The representation of $G_\pi$ on $ V_\pi$  is irreducible and the length of $V|_H$ is  
   $$\lg (V|_H)= [G:G_\pi] \, \lg (V_\pi|_H).$$
 
 \begin{lemma}\label{le:2.4} Assume  that $G/H$ is abelian. Then:
 
1) $G_\pi$ is normal in $G$ and does not depend on the choice of $\pi$ in $V_H$.  
  The  smooth $R$-characters of $G$ trivial on $G_\pi$ are in $ X_V$.

2) Assume $R$ algebraically closed.

a) Any irreducible subquotient of the smooth induction $\Ind_H^G 1$ is a smooth $R$-character $\chi$  of $G$ trivial on $H$.
  
  b) Any  irreducible $R$-representation   of $G$ containing $\pi$ is a twist  $V \otimes \chi $ of $V$  by   some smooth $R$-character $\chi$  of $G$ trivial on $H$.

3) When $V|_H $ has multiplicity $1$, then   $W = V_\pi$,  for a smooth  $R$-character $\chi$  of $G$ trivial on $H$,
$V \otimes \chi \simeq V$ if and only if $\chi$ is trivial on $G_\pi$, and $G_\pi$ is the largest subgroup $I$ of $G$ containing $H$ such that $\lg(V|_I) =\lg(V|_H)$.  

4) When $R$ is algebraically closed  and $V|_H$ has  multiplicity $1$, then
 $$|X_V|=\begin{cases} [G:G_\pi] \ &\text{ if  } \charf_R=0\cr
[G:G_{\pi,\ell}]\  &\text{ if  } \charf_R=\ell>0
\end{cases}
$$ 
where $G_{\pi,\ell}$ is the smallest subgroup of $G$ containing $G_\pi$ such that  
   $[G:G_{\pi,\ell}]$ is prime to $\ell$.
    \end{lemma} 
 
\begin{proof} 1)  The isotypic  components of $\Pi|_H$ are $G$-conjugate, their $G$-stabilizers are $G$-conjugate and  contain $H$  hence they are equal because $G/H$ is abelian.

The  smooth $R$-characters of $G$ trivial on $G(\pi)$ are in $ X_V$ because $V \otimes \chi  \simeq \Ind_{G_\pi}^G( \chi|_{G_\pi} \otimes V_\pi)$ for any smooth $R$-character $\chi$ of $G$.

2)  a) For any closed subgroup $Q$ of $G$ and a smooth $R$-representation $X$ of $Q$,  the representation $\Ind_Q^GX$ is  the space of functions $f:G\to X$ satisfying $f(qgk)=q f(g)$ for $q\in Q, g\in G, k\in K_f$ for some open subgroup $K_f$ of $G$ ,  with the action of $G$ by right translation, and  $\ind_Q^G 1$ is the  subrepresentation on the subspace  of  functions of compact support modulo $Q$.  When $G/Q$ is compact, $\Ind_Q^GX=\ind_Q^G X$. 
 
 Let $V\supset U$ be $G$-stable subspaces  with $V/U$ irreducible.  We can suppose $V$ generated by an element $f$ (indeed $V'/U'\simeq V/U$ for the  $G$-stable space $V'$ generated by a $f\in V\setminus U$  and the kernel $U'$ of $V'\to V/U$). There is an open subgroup $K$ of $G$ which fixes $f$. 
We have  $U\subset V\subset \ind_K^G 1$ and one is reduced to the  case where $G/H$ is finite.

b) The proof  of Lemma \ref{le:2.1} remains valid with the  smooth induction $\Ind_H^G$ which  is the smooth compact induction $\ind_H^G 1$ because $G/H$ is compact, so that $\ind_H^G(\Pi|_H)=\Pi \otimes \ind_H^G 1$.

 3)  Any smooth character  $\chi $ of $G$ trivial on $H$ such that 
  $\ind_{G_\pi}^G( V_\pi)\simeq \ind_{G_\pi}^G(V_\pi \otimes \chi|_{G_\pi} )$ is trivial on $G_\pi$. Indeed, restricting to $G_\pi$ we see that
  $V_\pi \otimes \chi|_{G_\pi} $ is  conjugate   to $V_\pi $  by some   $g\in G$. 
Restricting to $H$ gives that $\pi \simeq \pi^g$ so $g\in G_\pi$ hence $V_\pi \otimes \chi|_{G_\pi}  \simeq V_\pi $. As $\Ker(\chi)$ is open in $G$ and $G/H$ is compact, $J= \Ker(\chi) \cap G_\pi$ has finite index in $G_\pi$. 
 If $\chi$ is not trivial on $G_\pi $  then the  action of  $ J$  on $V_\pi$ is reducible.  Indeed, $ \ind_J^{G_\pi}(1)$ contains $1$ and $ \chi|_{G_\pi}$ as subrepresentations 
 and by Frobenius reciprocity
 $\End_{J}(V_\pi|_J)$ is equal to $\Hom_{G_\pi}(V_\pi, \ind_J^{G_\pi}(V_\pi|_J))= \Hom_{G_\pi}(V_\pi, V_\pi \otimes \ind_J^{G_\pi}(1))$. Hence $\dim( \End_{J_\pi}(V_\pi|_J))\geq 2$ and $V_\pi|_J$ is reducible. 
  But by hypothesis of multiplicity $1$, $V_\pi|_H$  is irreducible hence $V_\pi|_J$ is reducible as $H\subset J$.   So $\chi $ is trivial on $ G_\pi$.
  
  The group $G_\pi$   is a subgroup $I$ of $G$ containing $H$ with  $ \lg(V|_I)=\lg(V|_H)$.  
If $I$ has this property, the restriction to $H$ of any irreducible component on $V|_I$ is irreducible hence $I$ is contained in $G_\pi$.

 4) follows from 3).
\end{proof}
\begin{remark}\label{re:2.4} Assume that $V|_H$ has multiplicity $1$. The  $G$-stabilizer  of any irreducible component  of $V$ is $G_\pi$. Denote $G_\pi =G_V$.
Let $I$ be a subgroup  of $G$ containing $H$. The number of orbits of $I$ in the irreducible components of $V|_{G_V}$ is  $\lg(V|_I)$. This number is the same for $I$ and $IG_V$ hence $\lg(V|_I)=\lg(V|_{IG_V})$.  We deduce that $G_V\subset I$ if $V|_I$ is reducible  and  $|G/I|$ is a prime number.
\end{remark}
Let $\theta$ be  a  smooth  $R$-representation  of a closed subgroup   $U \subset H$. 
We consider the property:
\begin{equation}\label{eq:whi}\text{
The functor $\Hom_U(-,\theta)$ is exact on smooth $R$-representations of $H$}.
\end{equation}

 \begin{lemma} \label{le:U}   If \eqref{eq:whi} is true and  $\dim \Hom_U(V,\theta)=1$,  then $V|_H$ has multiplicity $1$.\end{lemma} 
 \begin{proof} We denote by $m_V(\pi)$ the multiplicity of any  irreducible smooth $R$-representation  $\pi$ of $H$ in  $V|_H$. By  \eqref{eq:whi},
  $$
\sum_{\pi}m_V(\pi)\dim \Hom_U(\pi, \theta) =\dim \Hom_U(V,\theta)=1.$$ 
There is a  single  $\pi $ with $m_V(\pi)= \dim \Hom_U(V, \theta)=1$.  
  \end{proof}

  \section{$p$-adic reductive group} \label{S:2.3}     We suppose now  that $G$ is a $p$-adic reductive group, that is, the group of rational points $\underline G(F)$ of reductive connected $F$-group  $\underline G$,   where $F$
is  a  local non archimedean field of residual characteristic $p$, of ring of integers $O_F$, uniformizer $p_F$, maximal ideal $P_F$,  residue field $k_F=O_F/P_F $ with $q$ elements, and absolute value  $| x|_F= q^{-val (x)}$, $| p_F|_F=q^{-1}$   (we do not  suppose that the characteristic of $F$ is $0$).

  For an algebraic group $\underline X$ over $F$, we denote by the corresponding lightfacee letter $X=\underline X(F)$ the group  its $F$-points. 

  Let $R$ be a field of characteristic $\charf_R\neq p$.
 Any irreducible smooth $R$-representation of $G$ is admissible \cite{HV19}, and   the properties (\ref{i})  and (\ref{ii}) hold for $G$. 
  For (\ref{i})  see 
 \cite[II.5.10]{V96}, \cite[\S 5]{V22},  and for (\ref{ii})  see \cite{D09}, 
 
 \noindent \cite{DHKM23}.   
 \begin{lemma}\label{le:clo} Let  $f :\underline H\to \underline G$ be an $F$-morphism of reductive connected $F$-groups. Then the subgroup $f(H)$ of $G$ is closed.
 \end{lemma}
    \begin{proof} The morphism $f$  induces  a constructible  action of $H$ on $G$
  \cite[6.15 Theorem A]{BZ77},   in particular the group $f(H)$, which is the  $H$-orbit of the unit of $G$,  is locally  closed (loc.cit. (6.8) Proposition), $f(H)$ is equal to its closure in $G$  (the closure of $f(H)$ in $G$ is a subgroup containing $f(H)$ as an open hence closed, subgroup).
   Note that $f(H)$ is  open  in $G$ when  $\charf_F=0$ \cite[\S 3.1 Corollary1]{PR91}. 
 \end{proof}
 
 \begin{theorem} \label{th:2.4} Let  $f :\underline H\to \underline G$ be an $F$-morphism of reductive connected $F$-groups such that $f(H)$ is a normal subgroup of $G$ of  compact quotient $G/f(H)$.
 Then, the restriction to $f(H)$ of any irreducible admissible $R$-representation of $G$ is semisimple of finite length. Any irreducible admissible $R$-representation of $f(H)$ is contained in some irreducible admissible $R$-representation of $G$ restricted to $f(H)$, and extends to an irreducible admissible representation of some open subgroup of $G$ of finite index.
 \end{theorem}
 
 \begin{proof} $G$ satisfies  (\ref{i}) and $f(H)$, satisfies  the property (\ref{ii}) because $H$ does.  Apply  the results of \S \ref{lc}.
  \end{proof}

 We now give two examples where we can apply Theorem \ref{th:2.4}.

 \begin{proposition}\label{pro:sc}  Let $f:\underline H\to \underline G$ be a surjective central $F$-morphism of connected reductive $F$-groups. Then, the subgroup $f(H)$  of $G$ is  normal  of abelian compact quotient $G/f(H)$.
  \end{proposition} 
  \begin{proof}   
   There is an $F$-morphism $\kappa:\underline G\times \underline G\to \underline H$ such that $\kappa(f(x), f(y))=xhx^{-1}y^{-1}$ for all $x,y\in \underline H$ \cite[2.2]{BT72}. So for all $u,v\in G$ we have $uvu^{-1}v^{-1}= f \circ \kappa (u,v)\in  f(H)$. The subgroup $f(H)$ of $H$  is closed (Lemma \ref{le:clo}), normal  with abelian quotient $G/f( H)$ (loc.cit. Proposition (2.7)). 
   
    The compacity $G/H$ is stated in  \cite{Sil79}  without proof and in  \cite[Proposition A.2.1]{L19} with indications for the proof. The idea is to reduce to  a  connected reductive $F$-anisotropic modulo the center $F$-group.

     Let $\underline S $ be  a maximal $F$-split subtorus  of  $\underline  G$,  and 
     $\underline B$   a parabolic $F$-subgroup of  $\underline  G$  containing  $\underline S $.
      The  $\underline  G$-centralizer  $\underline M$ of  $\underline S $ is compact modulo its center and is a Levi  component of   $\underline B $. Let  $\underline U$
     the unipotent  radical of  $\underline B$. By   \cite[22.6]{B91}. the inverse image $\underline S'$ of $\underline S$ in $\underline H$ is a maximal $F$-split torus in $\underline H$, and the inverse image $\underline B'$ of $ \underline B$ is a  parabolic $F$-subgroup of $\underline H$ Put $\underline M' $ for the $\underline H$-centralizer of $\underline S'$ and  $\underline U' $ for the unipotent radical of $\underline B'$. From loc.cit., $f $ induces a surjective central $F$-morphism $\underline M ' \to \underline  M$ and an $F$-isomorphism  $\underline U '  \to  \underline U$.   On the other hand, we have the Iwasawa decomposition $G= K B$  for an open compact subgroup $K$ of  $G$.
The product map $K \times  B\to G$ gives a surjective map
$K  \times B /f( B') \to  G /f(H)$.  We have
$  B /f( B') = M /f(M')$, 
so we just need to prove the compactness  of  $M /f(M')$.

   Let    $X^*(\underline S)$ denote the group of algebraic characters of  $\underline S$ and 
$\underline S  (p_F)=\Hom (X^*(\underline S),p_F^{\mathbb Z})$. The subgroup $\underline S  (p_F)$
 of $S$  is free abelian of finite rank with a compact quotient  $S/ \underline S  (p_F)$. 
 On the other hand, $f$ induces a surjective $F$-morphism $\underline S'\to \underline S$ sending  $\underline S'(p_F)$ onto a sub-lattice of $\underline S(p_F)$.    Hence $ S/ f( S')$ is finite.  So $M/  f(S')$ is compact as  $M/  S$ is compact, 
 a fortiori $M/ f (M')$ is compact.
 \end{proof}
 \begin{remark} The condition that $f$ is central in Proposition \ref{pro:sc} is necessary.   Indeed, assume   $\charf_F =2$ and  
$f:\underline {GL}_2\to \underline{ SL}_2, \ f(g)= \varphi(g)/\det (g)$ where 
 $\varphi(x)=x^2$ for $x\in F$  is the Frobenius\footnote{the map $f$ will also appear  in \S \ref{S:464}}.  The $F$-morphism  $f$  is surjective but not central.  Let $G=GL_2(F), G'=SL(2,F)$,    $T'$   the diagonal torus of $G'$, $U$ the group of unipotent  upper triangular matrices in $G'$. Then $f(G) = T' \varphi (G')$   is closed but not normal and not cocompact in $G'$  (as $\varphi(U)= U\cap T' \varphi (G')$  and  $U/\varphi(U)$   homeomorphic  to  $F/F^2$   is not compact).
 \end{remark}
\begin{corollary} \label{cor:2.9} Assume $R$ algebraically closed.
   Let $f:\underline H\to \underline G$ be an  $F$-morphism of connected reductive $F$-groups which induces a central $F$-isogeny  $\underline H^{der} \to \underline G^{der}$ between the derived groups.
Then the conclusions of Theorem \ref{th:2.4} apply to  $f(H)$.
\end{corollary}
\begin{proof} The  $F$-isogeny $\underline H^{der} \to \underline G^{der}$ is  surjective with finite kernel contained in the center of  $\underline H^{der}$ \cite[12.2.6]{Spr98}.  If $\underline Z$  is the connected centre of $\underline G$, the natural map 
$\underline Z \times \underline G^{der}\to \underline G $ is surjective \cite[8.1.6 Corollary]{Spr98} Hence
 the obvious map $\underline Z \times \underline H \to \underline G$  is surjective and central.
 Proposition \ref{pro:sc} applies to $Z f(H)$.  But $R$ being algebraically closed,  $Z$ acts by a character in any irreducible smooth $R$-representations of
$G$, and we get the corollary.
\end{proof}
\begin{remark}   There is a more elementary proof  that the restriction to $f(H)$ of any irreducible admissible $R$-representation of $G$ is semisimple of finite length  in \cite{Sil79}.  
\end{remark}

\section{Restriction to $SL_2(F)$ of  representations of $GL_2(F)$ }\label{SL2} 
 Let  $F$ be a  local non archimedean field  of  residue field  $k_F$ of characteristic $p$ as in \S\ref{S:2.3}, and  $R$  an algebraically closed field of characteristic different from $ p$.

   Let $G=GL_2(F)$, and let   $B $ (resp.  $B^-$)  denote the subgroup of upper (resp. lower) triangular matrices,       $T = $ the subgroup of diagonal matrices,  $U $ (resp.  $U^-$)  the   subgroup  of upper (resp. lower) triangular unipotent matrices, and $Z$  the center of   $G$.

 Let  $G'=SL_2(F)$. The  subgroup $H=ZG'  $ of $G$    is  closed normal of compact  abelian quotient $G/ZG'$ isomorphic via the determinant to
  $ F^*/(F^*)^2$, which is   a
  $\mathbb F_2$-vector space of dimension
\cite[Corollary 5.8]{Neu99}
 \begin{equation}\label{eq:car}\dim_{\mathbb F_2} F^*/(F^*)^2= \begin{cases} 2+e \ \text{if} \charf_F\neq 2\\
 \infty \ \text{if} \charf_F\neq 2  
 \end{cases}, \  \text{where } 2O_F=P_F^e.
 \end{equation}
Note that $ZG' $ is open in $G$ if and only if  $\charf_F\neq 2$.
 
  For a subset $X\subset G$, put $X'=X\cap G'$.   Write $x=(x_{i,j})$ a matrix in $G$ or $\Lie G=M_2(F)$.

We fix a separable closure $F^{sc}$ of $F$ and 
 will consider only extensions  of $F$  contained in $F^{sc}$. We write $W_F$ for the Weil group of $F^{sc}/F$ and $\Gal_F$ for the Galois group of $F^{sc}/F$. For a  field $k$, we denote by $k^{ac}$ an algebraic closure of $k$, and if $k\subset R$ we suppose   $k^{ac}\subset R$. 

  We fix an additive  $R$-character $\psi$ of $F$  trivial on $O_F$ but not on $P_F^{-1}$.

\subsection{Whittaker spaces} \label{ss:Wh}

 The smooth $R$-characters   of $U$  have the form \begin{equation}\label{eq:the}\theta_Y(u)=\psi\circ  \tr(Y(u-1))= \psi(Y_{2,1}u_{1,2}), \ \  u \in U,\end{equation} for  a lower triangular nilpotent matrix  $Y$ in $M_2(F)$.  The case $Y=0 $ gives the trivial character of $U$, the cases with $Y\neq 0$ give the {\sl non-degenerate} characters of $U$. 
 \begin{notation} \label{not:theta}When $Y_{2,1}=1$ we denote $\theta_Y=\theta $.
 \end{notation}
  The normalizer of $U$ in $G$ is $TU$. By conjugation,  $U$ acts trivially on $U$ and its characters, and a diagonal matrix $t=diag(t_1 ,t_2) $ acts on $u \in U$  by $(tut^{- 1})_{1,2} = (t_1/t_2) u_{1,2}$. Also, $t $ acts on a lower triangular nilpotent matrix $Y $ by  $(tYt^{-1})_{2,1}=(t_2/t_1)Y_{2,1}$. It follows that $T $ acts transitively on the non-degenerate characters of $U$, the quotient $T/Z$ acting simply transitively. By the same formulas, two non-trivial characters $\theta_Y$ and $\theta_{Y'}$ of $U$ are conjugate in $G'$  if and only if they are conjugate by an element of $T'$ if and only if $Y_{1,2}$ and $Y'_{1,2}$ differ by a square in $F^*$.

The   $T$-normalizer of $\theta_Y  $  is  equal to $Z$ if $Y\neq 0$ and to $T$ if $Y=0$.  The $\theta_Y $-coinvariants  functor $\tau \mapsto W_Y(\tau)$  from the smooth $R$-representations $\tau$ of $U$ to the smooth  $R$-representations  of the $T$-normalizer of $\theta_Y  $   is exact.
 A smooth $R$-representation $\tau$ of $U$  is called {\sl degenerate} when $W_Y(\tau)= 0$ for all $Y\neq 0$, and {\sl non-degenerate} otherwise. A smooth $R$-representation of $G$ or of $G'$ is called degenerate (or non-degenerate) if its restriction to $U$ is.
 
 The  finite  dimensional  irreducible smooth $R$-representations of $G$ are  of the form $\chi \circ \det$ for a smooth $R$-character $\chi$ of $F^*$ and are degenerate.  
 If $\Pi$ is  an  infinite  dimensional irreducible smooth $R$-representation of $G$, then 
 $\dim W_Y(\Pi)=1$  for all  $ Y\neq 0$ by the uniqueness of Whittaker models (\cite[III.5.10]{V96} when $\charf_R>0$)
 
 \subsection{$L$-packets}\label{S4.2}  We will classify  the irreducible smooth $R$-representations of  $G' $  by restricting to $G'$ the irreducible smooth $R$-representations  $\Pi$ of  $G $. The set $L(\Pi)$ of (isomorphism classes of) irreducible components of $\Pi|_{G'}$ is called an $L$-packet.
 A parametrization  along these lines was obtained when $\charf_F=0$ and $\charf_R=\mathbb C$  in \cite{LL79}.  When $\charf_F\neq 2 $ and $\charf_R\neq 2 $,  this question is  studied  for supercuspidal representations in the recent  work  \cite[\S 6.2 and \S 6.3]{CLL23}.
 
\medskip Applying  
      Lemma \ref{le:2.4}, Remark \ref{re:2.4}, Lemma \ref{le:U},  Theorem \ref{th:2.4},   Corollary \ref{cor:2.9}, we have: 
  \begin{align}   \text{Any  irreducible smooth $R$-representation of $G'$ belongs to a unique $L$-packet.}  \end{align} 
For two irreducible smooth $R$-representations $\Pi_1, \Pi_2$ of $G$,
\begin{align}  \label{re:scL}  L(\Pi_1)= L(\Pi_2)  \ \Leftrightarrow \Pi_1=(\chi \circ \det)\otimes \Pi _2\ \text{for some $R$-character $\chi\circ \det $ of $G$. } 
\end{align}

The  trivial character of $G'$ is the unique   finite  dimensional  irreducible smooth $R$-representation of $G'$, it  is degenerate and  forms an {\sl $L$-packet } $L(1)=L(\chi \circ \det)$ for any smooth $R$-character $\chi$ of $F^*$.

 If $\Pi$ is  an   irreducible smooth $R$-representation of $G$\footnote{  For cuspidal representations this is proved in \cite[Proposition 2.37, Corollary 2.38]{Cui23}.}, 
  \begin{align}
 \label{ssfl1} \text{The restriction   of $\Pi$ to $G'$
 is semi-simple of finite length and multiplicity $1$.  }
\end{align}
   The irreducible constituents   of  $\Pi |_{ G'}$ are $G$-conjugate  (even  $B$-conjugate  as  $G=BG'$), have a and  form an {\sl $L$-packet} $L(\Pi)$    of cardinality  the length  of $\Pi|_{ G'}$. 
   The    $G$-stabilizer of  $\pi\in L(\Pi)$ does  not depend on the choice of $\pi$ in $L(\Pi)$ and  is denoted $G_\Pi$. By \S \ref{S:226}, $G_\Pi$ is an open normal subgroup  of $G$  containing $G'Z$,  the subgroup $\det (G_\Pi)$ of $F^*$ is open and contains $(F^*)^2$.  The order of the quotient $G/G_\Pi\simeq F^*/\det(G_\Pi)$  is a power of $2$. When  $\charf_F\neq 2$,   $|G/G_\Pi|$ divides $ |F^*/(F^*)^2|=2^{2+e}$ with $e$ defined in \eqref{eq:car}.
    \begin{align}  G_\Pi \ \text{ is the largest subgroup $I$ of $G$   such that} \     lg(\Pi|_I ) = \lg(\Pi |_{G'} ), \\
       \Pi= \ind_{G_\Pi}^G V_\pi \ \text{where $V_\pi$ is the space of $\pi$,}  \ \ \ \ \ \ \ \ \ \ \ \ \ \ \ \ \ \ \ \ \ \ \ \ \ \  \ \ \\ 
\label{eq:LPi1}
  \text{$L(\Pi)$ is a principal homogeneous space for $G/G_\Pi$.}     \ \ \ \ \ \ \ \ \ \ \ \ \ \ \ \ \  \ \ \ \ \ \ \ \ \ 
 \\
   \text{ $|L(\Pi)|$ is a power of $2$, and $|L(\Pi)|$ divides $2^{2+e}$ when $\charf_F\neq 2$. \ \ \ \ \ \ \ \ \ }    
\end{align}
When $p$ is odd since  $ |F^*/(F^*)^2|=4 $ we deduce:
  \begin{proposition} \label{LPip}
 \text{When $p$ is odd, the cardinality of an $L$-packet is $1,2$ or $4$.  }
\end{proposition}
When $p=2$ we will prove that  this remains true  using the local Langlands 
correspondence.

\medskip
  By class field theory,
any  open subgroup of $F^*$ of index $2$ is equal to $N_{E/F}(E^*)$ for 
    a unique quadratic separable extension $E/F$ of relative  norm $N_{E/F}:E^*\to F^*$, and conversely. 
    Any  open subgroup of $F^*$ of index $4$  containing $(F^*)^2$ is equal to $N_{K/F}(K^*)$ for 
    a unique biquadratic separable extension $K/F$ of relative  norm $N_{K/F}:K^*\to F^*$, and conversely. 
      
      When $p $ is odd,  each quadratic extension of $F$ is separable and tamely ramified,  and there is a  unique biquadratic separable extension of $F$.

When $p=2$,  if   $ \charf_F=0$, there are finitely many quadratic  separable extensions of $F$  and 
   finitely many   biquadratic  separable extensions of $F$ (formula \eqref{eq:car}).
   If $\charf_F=2$, there are infinitely many quadratic, resp. biquadratic, separable extensions of $F$.

\begin{definition}\label{def43} When $\Pi$ is an irreducible smooth $R$-representation of $G$, we denote by $E_\Pi$ the separable extension of $F$ such that $N_{E_\Pi/F}(E_\Pi^*)=\det (G_\Pi)$.  
\end{definition}       
 
We denote by 
\begin{equation}\label{def:XPi} \text{$X_\Pi$ the group of characters $\chi \circ \det $ of $G$ such that $\Pi\otimes (\chi \circ \det) \simeq \Pi$. }
\end{equation} 
A character  of $X_\Pi$ is  smooth (Lemma \ref{le})  of trivial square. So $X_\Pi=\{1\}$ if $\charf_R=2$.
 \begin{notation}    \label{not:etaE}  When $\charf_R\neq 2$, the non-trivial smooth $R$-characters of $F^*$ of trivial square are      the  $R$-characters $\eta_E$ of $F^*$ of kernel $N_{E/F}(E^*)$, for   quadratic separable extensions $E/F$.
The modulus $ q^{\pm val}$ of $F^*$  is equal to $\eta_E$ if and only if  $E/F$ is unramified and  $q+1= 0$ in $R$.
 \end{notation}

By Lemma \ref{le:2.4} and the formula \eqref{eq:LPi1},
\begin{align} \label{eq:LPi}
 \text{ $X_\Pi $ is the group of $R$-characters of $G$ trivial on $G_\Pi$,} \ \ \ \ \ \ \\
\text{When $\charf_R\neq 2$, the cardinality of $L(\Pi) $  is   $|X_\Pi|$. \ \ \ \ \ \   \ \ }
  \end{align}
It is known that  $|X_\Pi| =1,2$ or $4$ 
 when 
 
 a) $R=\mathbb C$ and $\charf_F=0 $  \cite{LL79}\cite{Sh79},  
 
  b) $\charf_F\neq 2,  \charf_R\neq 2$  \cite[Proposition 6.6]{CLL23}. 
  
 \noindent   
  When $\charf_R\neq 2$ we will prove that   $|X_\Pi| =1,2$ or $4$  using the local Langlands correspondence, therefore   $|L_\Pi| =1,2$ or $4$ when $p=2$.

\medskip      For a lower triangular matrix $Y\neq 0$, we have 
 $$\sum_{\pi\in L(\Pi)}\dim_R W_Y(\pi)= \dim_R W_Y(\Pi) .$$  
 As $ \dim_R W_Y(\Pi) =1$, we have  
  $ \dim_R W_Y(\pi)=0$ or $1$, and 
there is a single $\pi\in L(\Pi)$ with $W_Y(\pi)\neq 0$.  
 
 \subsection{Representations} \label{ss:rep} 
  We denote by $\Gr_R^\infty(G)$  the Grothendieck group    of finite length  smooth $R$-representations of $G$ and by 
$[\tau]$    the image in $\Gr_R^\infty(G)$  of a finite length  smooth $R$-representation $\tau$ of $G$. Similarly for $G'$.

\subsubsection{Parabolic induction}\label{S:iBG}
\ 

The smooth {\sl parabolic induction}   $\ind_B^ G (\sigma)$ of    a smooth $R$-representation $(\sigma, V)$ of $T$   is  the space of  functions $f:G\to  V $ such that $f(tugk)=\sigma(t) f(g)$ for $t\in T, u\in U, g\in G$ and an open compact subgroup $K_f\subset G$, with the action of $G$ by right translation. 
The functor $\ind_B^G $  is exact with  the $U$-coinvariant functor $(-)_U$  as left adjoint, and   $(-)_{\overline U}\otimes  \delta$  as rignt adjoint where $\delta$
  is the homomorphism of $T$:
  $$\delta(\diag(a,d))=q^{-\val (a/ d)}: T\to q^{\mathbb Z}   \ \ \ (a,d\in F^*)$$ \cite[Corollary 1.3]{DHKM23}.    
 The    modulus $| \ |_F$ of $F^*$ is  $q^{-\val}$ and  the modulus of $B$ is the inflation of $\delta$. 
{\it We choose a square root} $q^{1/2}$ of $q$ in $R^*$ to define the square root of $\delta$, \begin{equation}\label{eq:nu}\nu(\diag(a,d))=(q^{1/2})^{-\val(a/d)}: T\to (q^{1/2})^{\mathbb Z}   \ \ \ (a,d\in F^*), 
   \end{equation}
  and  the {\sl normalized parabolic induction}
   $i_B^ G(\sigma)=\ind_B^G (\sigma \nu) $. For a smooth $R$-character $\chi \circ \det$ of $G$ we have
   $$(\ind_B^G\sigma )\otimes (\chi \circ \det) \simeq \ind_B^G( \sigma \otimes (\chi\circ \det)), \ \ \ 
( i_B^G\sigma)  \otimes  (\chi \circ \det)\simeq i_B^G( \sigma  \otimes(\chi \circ \det) ).$$
 Similarly   for $G'$,  we define the  parabolic induction $\ind_{B'}^{G'}$ from the    smooth $R$-representation $\sigma$  of $T'$ to those of $G'$ and the   normalized parabolic induction  $i_{B'}^{G'}$
 $$ i_{B'}^{G'}(\sigma) =\ind_{B'}^{G'}( \nu' \sigma),  \ \ \nu'(diag(a,a^{-1}))=q^{-\val(a)}:T'\to q^{\mathbb Z}    \ \  (a\in F^*). $$
        As $G=BG'$ and $G/B$ is compact, the  restriction map $f\mapsto f|_{G'}$ gives  isomorphisms  
\begin{equation}\label{eq:i}(\ind_B^ G (\sigma))|_{G'}\to\ind_{B'}^{G'} (\sigma|_{T'}), \ \   (i_B^ G (\sigma))|_{G'}\to  i_{B'}^{G'} (\sigma|_{T'}).
\end{equation} 

 \subsubsection{Cuspidal representations of $GL_2(F)$} \label{S:331} 
 \
 
 When $\chi$ is a smooth $R$-character of $T$,  $\ind_B^G(\chi)$ is called a {\sl principal series}  of $G$.      An irreducible   smooth $R$-representation  of $G$  which  is not a subquotient of a principal series, is called {\sl supercuspidal}.   It is called 
   {\sl cuspidal} when its $U$-coinvariants are $0$.  A supercuspidal representation is cuspidal (the converse is true only when $q+1\neq 0$  in $R$). 
      The principal series and the  cuspidal $R$-representations are infinite dimensional. Similarly for  $G'$.
  
  Let $\Pi$ be an irreducible smooth $R$-representation of $G$ and $\pi\in L(\Pi)$. Then,\begin{equation}\label{eq:ss}\text{$\Pi$ is cuspidal if and only if $\pi$ is cuspidal (similarly for supercuspidal).}
\end{equation}
 Indeed, $L(\Pi)$ is the $B$-orbit of $\pi$, the $U$-coinvariant functor  is exact and commutes with the restriction to $G'$. We say that $L(\Pi)$ is  cuspidal if $\Pi$ is.
 Similarly for supercuspidal using  the formula \eqref{eq:i}.

    Let $\Pi$ be a  cuspidal $R$-representation of $G$. 
 It  is the compact induction of an extended maximal simple type $(J,\Lambda)$    $$\Pi=\ind_J^G(\Lambda)$$
 (\cite{BK94}  \cite{BH02} when $R=\mathbb C$, \cite[III.3-4]{V96}  for general $R$).
The group $J$ contains $Z$ and a  unique maximal open compact subgroup  $J^0$.     Let  $J^1$ be the pro-$p$ radical of $J^0$. The representation  $\Lambda|_{J^0}$   is  irreducible, equal to $\lambda= \kappa \otimes \overline \sigma$ where   $\kappa|_{J^1}$ is irreducible   and $\overline \sigma$ is inflated from an irreducible $R$-representation $\sigma$ of $J^0/J^1$. The  type  $(J,\Lambda)$ is unique modulo $G$-conjugacy (\cite[15.5 Induction theorem]{BH06} when $R=\mathbb C$,   \cite[III.5.3]{V96} for general $R$ \footnote{It is proved only that $(J^0,\lambda)$ is unique modulo $G$-conjugacy, but $J$ is the normalizer of $(J^0,
\lambda)$ and $\Lambda$ is $\lambda$-isotypic part of $\Pi$}. 

The open normal subgroup $JG'$ of $G$ has index $| F^*/\det (J)|$, and by Mackey theory
 \begin{equation}\label{eq:JG'}\Pi|_{JG'} = \oplus _{g\in  G/JG'} \ind_{J^g}^{JG'}\lambda ^g .\end{equation}
 Denote $J', {J^0}', {J^1}'$ the intersections of $J, J^0, J^1$ with $G'$. We have $J'=(J^0)'$ and  the length of
$$( \ind_{J^g}^{JG'}\lambda ^g  )|_{G'}\simeq \ind_{J'^g}^{G'} (\lambda^g|_{J'^g})$$
 is independent of $g$. By transitivity of the restriction
$\Pi|_{G'} =\oplus_{g\in  G/JG'} \ind_{J'^g}^{G'} (\lambda^g|_{J'^g})  $, 
and 
     \begin{equation}\label{eq:cases} |L(\Pi)|= | F^*/\det (J)| \,  lg(\ind_{J'}^{G'} (\lambda|_{J'})).
\end{equation}
 It follows from Lemma \ref{le:2.4} 3), remark \ref{re:2.4} and the formula \eqref{eq:JG'} that:
\begin{lemma}\label{le:GL} If $| F^*/\det (J)| =2$ then  $ \det(G_\Pi)\subset  \det(J)$.
\end{lemma}
\begin{remark}We have $ \det(G_\Pi)= \det(J) \Leftrightarrow G_\Pi=JG'$. If $| F^*/\det (J)| =2$, the group $J$ determines a quadratic separable extension $E/F$ such that $\det (J)=N_{E/F}(E^*)$. The representation  $\ind_{J'}^{G'} (\lambda|_{J'})$ is irreducible
if and only if
$|L(\Pi)|=| F^*/\det (J)| $. 
\end{remark} 
 If there is a smooth $R$-character $\chi$ of $F^*$ such that $\Lambda \simeq \Lambda_0\otimes (\chi \circ \det)$ and   $(J,\Lambda_0)$ is of level $0$, we say that the $L$-packet $L(\Pi)$ and its elements are of  level $0$. Otherwise we say that  $L(\Pi)$ and its elements are of  positive level.

\medskip  {\bf Level $0$.}  $J =  Z GL_2(O_F)$,  $J^0= GL_2(O_F)$, $J^0/J^1\simeq 
GL_2(k_F)$, $\kappa=1$,   $\sigma$ is a  cuspidal $R$-representation of  $ GL_2(k_F)$, $\lambda = \overline \sigma$. 
We have $\det J=\val^{-1}(2\mathbb Z)$, and by \eqref{eq:cases}: \begin{equation} \label{eq:case1} |L(\Pi)|=  2\, lg(\lambda|_{J' })=2\,lg(\sigma|_{SL_2( k_F)}),
\end{equation}
because  
 $\lambda|_{J' } $ is semisimple with   length $ lg(\sigma|_{SL_2( \mathbb F_q)})$, 
and for any irreducible component $\lambda'\subset \lambda|_{J'}$, the compact induction $\ind_{J'}^{G'} (\lambda')$ is irreducible (\cite{HV22} Corollary 4.29).  

The cardinality of the cuspidal  $L$-packet  $L(\Pi)$ of level $0$ can be computed  via   \eqref{eq:cases}, \eqref{eq:case1}, and  Remark \ref{re:redf} b) given in the appendix on 
 the classification of the irreducible $R$-representations of $GL_2(k)$ and of $SL_2(k)$ for a finite field $k$ with $\charf_k\neq \charf_R$.  We have two cases:
 \begin{enumerate} 
\item  $|F^*/\det (G_\Pi)|=2$ and $E_\Pi/F$ is the unramified quadratic extension. 
 
\item $p$ is odd, $\det (G_\Pi)=(F^*)^2$  and $E_\Pi/F$ is the unique biquadratic extension.   This case occurs for a unique packet $L(\Pi)$.
  \end{enumerate}
We deduce:  \begin{proposition} \label{prop:ty1}  When $p=2$,  each  level $0$ cuspidal $L$-packet  has size $2$.

  When  $p$ is odd, there is a unique  level $0$ cuspidal $L$-packet  of size $4$,  the other  level $0$ cuspidal $L$-packets have size $2$.
 \end{proposition}
 
 These results can be deduced from \cite[\S 2]{KP91} and the size 4 depth zero $L$-packet has been obtained in \cite[Method 2 Example 3.11]{Cui23}.
 
\bigskip   {\bf Positive Level.}   $J=E^*J^0$ for a quadratic separable\footnote{When $\charf_F=2$  the quadratic extension appearing in the construction \cite{BH06}  is not necessarily
separable. It is generated by an element $x\in G$, determined up to some open subgroup of $G$, so that modifying  $x$ slightly  yields a separable extension.} extension $E/F$,  $J^0=O_E^*J^1$,  
$J^0/J^1 \simeq  k_E^*$, $\sigma$ is an $R$-character of $k_E^*$, $\lambda= \kappa\otimes \sigma$ and  $\lambda|_{J'}$ is irreducible. The representation $\lambda_1=\lambda|_{J^1}$ is irreducible of $G$-intertwining 
 equal to $J$, because $J$ normalizes $\lambda_1$ and the $G$-intertwining of $\sigma$ is already $J$ \cite[15.1]{BH06}.  
 We have $N_{E/F}(E^*)\subset \det(J)$. 
  If the quadratic extension $E/F$ is  tamely ramified, then  $ \det(J) =N_{E/F}(E^*) $, because $J=E^*J^1$, $J^1=(1+P_F)J'^1$ and $ 1+P_F \subset \det (E^*)=N_{E/F}(E^*)$.
  
 If $p=2$ a tamely ramified  quadratic extension of $F$ is unramified, and  $E/F$ is unramified
if and only if $\det(J)= Ker((-1)^{val})$.

  If $p$ is odd, each quadratic extension of $F$  is tamely ramified.

\begin{proposition}\label{prop:3.12} If $p$ is odd,  each positive level cuspidal $L$-packet $L(\Pi) $ has size $2$ and  $E=E_\Pi$ (Definition \ref{def43}).     \end{proposition}

 \begin{proof}\footnote{This can also be obtained using  \cite{Cui23}}  The central subgroup $1+P_F$ of $J^1=(1+P_F)J'^1$ acts by scalars, the representation  $ \lambda'_1=\lambda|_{J'^1}$ is still irreducible of $G$-intertwining $J$, so its $G'$-intertwining is $J'$.
The isotypic component of $\Pi|_{J^1}$ of type $\lambda_1$ is the space of $\lambda$, so the  isotypic component of $\Pi|_{J'^1}$ of type $\lambda'_1$  is still the space of $\lambda$. As in the proof of  \cite[Corollary 4.29]{HV22}, we deduce that $\ind_{J'}^{G'}(\lambda|_{J'})$ is irreducible.  Apply Lemma \ref{le:GL}.
\end{proof}

  \begin{remark}\label{prop:3.13}      When   $p=2$ and $E/F$ is ramified,  then $J^0\cap G'$ is a pro-$2$-group.
     Indeed, the determinant induces a morphism $J^0/J^1\to k_F^*$ equal  via the natural isomorphism $J^0/J^1\to k_E= k_F^*$ to the automorphism  $x\to x^2$ on $k_F^*$. Hence $(J^0)'= (J^1)' $ is a pro-$2$-group.  Note also that $\Lambda$ is a character \cite[\S 15]{BH02}.
 \end{remark}

\begin{corollary}\label{cor:2} (Propositions \ref{prop:ty1}, \ref{prop:3.12}) When $p$ is odd, there is a unique cuspidal $L$-packet  of size $4$, and  it is of level $0$. The other  cuspidal $L$-packets  have size $2$.
\end{corollary}
 
 \subsubsection{Principal series of $GL_2(F)$}  \label{spG} We recall the description of the normalized principal series  $i_{B}^{G}(\chi)$ of $G$ for   a  smooth $R$-character $\chi$ of $T$.

Denote by  $\chi_1,\chi_2$   the smooth $R$-characters of $F^*$ such that 
\begin{equation}\label{eq:chi}
\chi(\diag(a,d)) = \chi_1 (a) \chi_2(d),   \ \ \ (a,d\in F^*),\end{equation}
 and by   $\chi^w$ the character  
 $ \chi^w(\diag(a,d)) = \chi(\diag(d,a))$ of $T$. 
In particular in \eqref{eq:nu}, $\nu^w= \nu^{-1}$ and $\nu/\nu^w=\delta$.

\begin{proposition} \label{prop:spG} 
\begin{enumerate}\item  For 
  two smooth $R$-characters  $\chi, \chi'$   of $T$,
 $[i_{B}^{G}(\chi )] $  and  $[i_{B}^{G}(\chi ')] $  are disjoint or equal, with equality if and only if $ \chi '=\chi$ or $\chi^w$.

\item  The smooth dual of   $i_{B'}^{G'}(\chi)$  is   $i_{B'}^{G'}(\chi^{-1})$.

\item    $(i_{B}^{G}(\chi ))_U$ has dimension $2$, contains $\chi^w$ and has quotient $\chi$.
 
\item    $\dim W_Y(i_{B}^{G}(\chi ))=1$  when  $ Y\neq 0$ (\cite{V96} III.5.10).

\item    $i_{B}^{G}(\chi )$  is reducible if and only if $\chi_1 \chi_2^{-1}= q^{\pm \val}$.   
\item    $\ind_B^G( 1)= i_B^G(\nu^{-1})$ contains the trivial representation $1$ and 

$\bullet$  if $q+1\neq 0$ in $R$,  $\lg(\ind_B^G( 1))= 2$, in particular  $\St=(\ind_B^G 1)/1$ is   irreducible (the Steinberg $R$-representation). The representation 
$\ind_B^G 1$ is semi-simple if and only if $q=1 $ in $R$  (and $\charf_R \neq 2$).

$\bullet$  if $q+1= 0$ in $R$,  $\lg(\ind_B^G( 1))= 3$,   $\ind_B^G( 1)$ is indecomposable of quotient $(-1)^{\val}\circ \det$, and   $\ind_B^G( 1)/1$   contains a cuspidal   representation  
$$\Pi_0=\ind_{Z GL_2(O_F)}^G \tilde \sigma_0$$ 
 where $\tilde \sigma_0$ is the inflation  to $Z GL(2,O_F)$ of the  cuspidal  subquotient  $\sigma_0$ of $\ind_{B(k_F)}^{GL(2,k_F)} 1$ (appendix).
\end{enumerate}
\end{proposition}
This is   \cite[Theorem3]{V89} but the proof of (i) is incomplete. What is missing is the proof that $\Pi_0$ occurs only in $i_B^G(\nu)$ and $i_B^G(\nu^{-1})$ when $q+1= 0$ in $R$. This is equivalent to 
 $X_{\Pi_0} =\{1, (-1)^{\val} \circ \det\}$ with the notation  \eqref{def:XPi}. This follows from  Remark \ref{re:redf} a) given in the appendix.

\begin{remark}\label{re:etaE}  

1) The   Steinberg representation $\St$ is infinite dimensional and not cuspidal.

2) When $\charf_R\neq 2$,  the principal series $[i_B^G(\chi)]$ are multiplicity free.

When  $\charf_R=2$, then  $q$ is  odd,     $\ind_B^G (1) $ has length $3$, of subquotients $\Pi_0$ and the trivial representation $1$  as  a subrepresentation and a quotient. 
 \end{remark}

 \begin{corollary}\label{cor:spG} The non-supercuspidal irreducible smooth $R$-representations of $G$ are:

- the  characters $\chi \circ \det$   for the  smooth $R$-characters $\chi$ of $F^*$,

- the principal series $i_B^G(\chi)$ for the  smooth $R$-characters $\chi$ of $T$ with $\chi_1 \chi_2^{-1}\neq q^{\pm \val}$.

- if $q+1\neq 0$ in $R$, the twists $(\chi \circ \det)\otimes \St$ of the Steinberg representation    by  the smooth  $R$-characters $\chi$ of $F^*$.

- if  $q+1=0$ in $R$, the  twists $(\chi \circ \det)\otimes \Pi_0$ of the cuspidal non-supercuspidal representation $\Pi_0$ by the smooth $R$-characters $\chi$ of $F^*$.

The only isomorphisms between those representations are $i_B^G(\chi) \simeq i_B^G(\chi^w)$ for the irreducible principal series and $(\chi \circ \det)\otimes \Pi_0 \simeq ((-1)^{val}\chi \circ \det)\otimes \Pi_0$.
\end{corollary}

  \subsubsection{}  Let $\ell$ be a prime number different from $p$. 
 An  irreducible  smooth $\mathbb Q_\ell ^{ac}$-representation $\tau$ of $G$ or $G'$   is integral if it preserves a lattice. It then gives by reduction modulo $\ell$ and semi-simplification a  finite length semi-simple smooth $\mathbb F_\ell ^{ac}$-representation, of isomorphism class (not depending of the lattice) which we write $r_\ell (\tau)$.
  The restriction from $G$ to $G'$ from  irreducible  smooth $\mathbb Q_\ell ^{ac}$-representations $\tilde \Pi$ of $G$ to  finite length semi-simple smooth $\mathbb Q_\ell ^{ac}$-representations of $G'$ respects integrality and commutes with the reduction modulo $\ell$.  When  $\tilde \Pi$  is integral, then any irreducible representation $\tilde \pi \subset  \tilde \Pi|_{G'}$ is integral,  the length of the reduction $r_\ell (\tilde \pi)$ modulo $\ell$ of  $\tilde \pi $ does not depend on the choice of $\tilde \pi$. If  $\Pi=r_\ell(\tilde \Pi)$   is irreducible, we have
   \begin{equation}\label{eq:case2} |L( \Pi) | =| L((\tilde \Pi))| \, \lg (r_\ell (\tilde \pi)) 
   \end{equation} 
   and by formula \eqref{eq:LPi}:
  \begin{equation} \label{eq:red'} 
 \lg (r_\ell(\tilde \pi ))= |X_\Pi/ X_{\tilde \Pi}|  \ \text{when}  \ \charf_R\neq 2.
   \end{equation}

  \begin{proposition}\label{pr:ilft}Each  irreducible smooth $\mathbb F_\ell ^{ac}$-representation $\Pi$ of $G$ is  the reduction modulo $\ell$ of some integral  irreducible  smooth $\mathbb Q_\ell ^{ac}$-representation $\tilde \Pi$ of $G$.
    \end{proposition} 
 \begin{proof}  Corollary \ref{cor:spG} when $\Pi$ is not-cuspidal, \cite{V01} when $\Pi$ is cuspidal. \end{proof}

A   supercuspidal $\mathbb Q_\ell ^{ac}$-representation $\tilde \Pi=\ind_J^G \tilde \Lambda $  of $G$ is integral if and only if $\tilde \Lambda$ is integral. Then, 
  its reduction modulo $\ell$   is  irreducible \cite{V89}, equal to $\Pi=\ind_J^G   \Lambda $  where $\Lambda =r_\ell(\tilde \Lambda )$. The reduction modulo $\ell$ of the   $L$-packet $ L(\tilde \Pi)$ is  $L(\Pi)$. The reduction modulo $\ell$ respects level $0$ and positive level.
 Conversely,    any   cuspidal $\mathbb F_\ell ^{ac}$-representation $\Pi=\ind_J^G   \Lambda $  of $G$  is the reduction modulo $\ell$ of   an integral   cuspidal     $\mathbb Q_\ell^{ac}$-representation $\tilde \Pi=\ind_J^G \tilde \Lambda$  of $G$  where $\Lambda= r_\ell (\tilde \Lambda)$
    \cite{V01}. 
    By the unicity of the extended maximal simple type $(J,\Lambda)$ modulo $G$ (see \S \ref{S:331}), two   supercuspidal integral $\mathbb Q_\ell ^{ac}$-representations have isomorphic reduction modulo $\ell$ if and  only if the reduction modulo $\ell$ of their  extended maximal simple types are $G$-conjugate.
    
     Any supercuspidal  $\mathbb Q_\ell^{ac}$-representation $\tilde \pi
 $ of $G'$ is integral, as $\tilde \pi \in L(\tilde \Pi)$  where $\tilde \Pi$  is a supercuspidal  $\mathbb Q_\ell^{ac}$-representation of $G$,  and some twist of $\tilde \Pi$ by a character is integral. Suppose  that $\tilde \Pi$ has level $0$. With the notations of the formula  \eqref{eq:case1}, the  formula \eqref{eq:red'} implies:
  \begin{equation} \label{eq:case2}  
 \lg (r_\ell(\tilde \pi ))=  lg(\sigma|_{SL_2( k_F)})/lg(\tilde \sigma|_{SL_2( k_F)}).
 \end{equation}

 \begin{proposition}  \label{prop:liftc}  
When $\tilde \pi$ is   supercuspidal  of  level $0$, the  length of  $r_\ell(\tilde \pi)$ is $\leq 2$.

 When  $\tilde \pi$ is  supercuspidal and $p$ is odd,  then  $r_\ell(\tilde \pi)$  is irreducible if  $\tilde \pi$ is minimal of positive level or if $\ell=2$.

Any  cuspidal   $\mathbb F_\ell^{ac}$-representation  $\pi$ of $G'$ is the reduction modulo $\ell$ of a  supercuspidal   $\mathbb Q_\ell^{ac}$-representation of $G'$, except may be  when $p=2$ and   $\pi$ is  in an $L$-packet $L(\Pi)$  with $\Pi$ minimal of  positive level with $E_\Pi/F$  unramified  (Definition  \ref{def43}).
  \end{proposition}
        
\begin{proof} 
For $\tilde \Pi$ of level $0$,   one computes  in the appendix the integer $\lg(\sigma|_{SL_2(k_F))}/\lg(\tilde \sigma|_{SL_2(k_F)})$, and one sees  that it is    equal to  $1$ or $2$ and  that there exists $\tilde \sigma$ such that  it is $1$.

For  $p$ odd,  if the level of $\tilde \pi$   is positive we have $\lg (\Pi|_{G'})=\lg (\tilde \Pi|_{G'})$  by Proposition \ref{prop:3.12}, hence  $r_\ell(\tilde \pi)$  is irreducible.

 For $\ell =2$ (so $p$ is odd), if the level of $\tilde \pi$   is $0$, then $r_\ell(\tilde \pi)$  is also irreducible by the formula \eqref{eq:case2} and   Lemma \ref{le:53}  in the appendix.

 For $p=2$ (so $\ell $ is odd),   $\pi$ is in a  cuspidal $L$-packet $L( \Pi)$ with $\Pi$ minimal   of positive level with $E_\Pi/F$ ramified. Let $\tilde \Pi$ a  $\mathbb Q_\ell^{ac}$-lift of $\Pi$. The reduction modulo $\ell$  from $X_{\tilde \Pi}$ onto $X_\Pi$ is injective.
 The proposition follows from the next lemma.\end{proof}

\begin{lemma} The reduction modulo $\ell$  from $X_{\tilde \Pi}$ onto $X_\Pi$ is  a bijection.
\end{lemma}
\begin{proof}  
 Let $\chi\in X_\Pi$, $\chi\neq 1$, and $\tilde \chi $ the unique $\mathbb Q_\ell^{ac}$ lift of $\chi$ of
order $2$. We have $\tilde \Pi =\ind_J^G \tilde \Lambda$ where $\tilde \Lambda$ is a character (Remark \ref{prop:3.13}). We have $ \Pi =\ind_J^G \Lambda$ where $\Lambda=r_\ell (\tilde \Lambda)$ and $(J,\chi \Lambda)=(J,{}^g\Lambda)$ for $g\in G$ normalizing $J$. So $\tilde \chi \tilde \Lambda= \epsilon \,{}^g \tilde \Lambda $ for a $\mathbb Q_\ell^{ac}$-character $\epsilon$ of $J$ of order  a power of $\ell$. So, $\epsilon|_{J_1}=1$ and $\epsilon|_Z=1$. As $E_\Pi/F$ is ramified, the index of  $ZJ^1$ in $J$ is $2$ hence $\epsilon =1$ and $\tilde \chi \in  X_{\tilde \Pi}$.
\end{proof}
When $\charf_F\neq 2$ and $\charf_R\neq 2$, compare with   \cite[Proposition 6.7]{CLL23}.   When $p=2$,  we shall complete the proposition in Corollary \ref{prop:redu1}: if  $\tilde \pi$ has positive level then  $r_\ell(\tilde \pi)$ has length   $\leq 2$,  if  $\pi$ is  in an $L$-packet $L(\Pi)$ of positive level with $E_\Pi/F$  unramified then  $\pi$  lifts to $\mathbb Q_\ell^{ac}$.

\subsection{Local Langlands $R$-correspondence for $GL_2(F)$}\label{pas3}
  
\subsubsection{}\label{S:2d}By local class field theory,  the  smooth $R$-characters  $\chi$ of $F^*$ identify with the smooth $R$-characters $\chi \circ \alpha_F$ of $W_F$ where  $\alpha_F: W_F\to F^*$  is the Artin reciprocity map   sending a arithmetic Frobenius  $\Fr$ to $p_F^{-1}$ (\cite{BH02} \S 29). 
 This is the  local Langlands $R$-correspondence for $ GL_1(F)$.
 
A two-dimensional  Deligne $R$-representation  of the Weil group $W_F$ is a pair $(\sigma, N)$  where $\sigma$ is a two dimensional  semi-simple smooth $R$-representation of the Weil group $W_F$ and $N$ a nilpotent $R$-endomorphism of the space  of $\sigma$ with the usual requirement: 
\begin{equation}\label{requi}\sigma(w) \, N =N |\alpha_F(w)|_F   \, \sigma(w)  \ \text{for} \ w\in W_F.
\end{equation}
 Two two-dimensional  Deligne $R$-representations $(\sigma,N)$  and  $(\sigma', N')$  of $W_F$ are  isomorphic  if there exists a linear isomorphism $f:V\to V'$ from the space $V$ of $\sigma$ to the space $V'$ of $\sigma'$ such that $\sigma'(w)\circ f= f \circ \sigma(w)$ for $w\in W_F$ and $N'\circ f=f\circ N$. 
 
   For  a smooth $R$-character $\chi$ of $F^*$, the twist $(\sigma, N)\otimes (\chi\circ \alpha_F)$ of   $(\sigma, N)$  by   $\chi\circ \alpha_F$  is  $( \sigma\otimes (\chi \circ \alpha_F),  N)$. 

When $R=\mathbb Q_\ell^{ac}$, $(\sigma, N)$   is called integral if $\sigma$ is integral.  

\begin{remark}\label{re:requi} When $\sigma$ is irreducible we have $N=0$.  

When  $\sigma=(\chi_1\oplus \chi_2)\circ \alpha_F$, if   $\chi_1\chi_2^{-1}\neq q^{\pm \val}$  then $N=0$, When $N\neq 0$,  we have $\{\chi_1,\chi_2 \}= \{\chi_i, q^{-\val} \chi_i \}$ for some $i$
  and $N$ sends the $\chi_i\circ \alpha_F$-eigenspace to the $q^{-\val} \chi_i \circ \alpha_F$-eigenspace or $0$.
 Therefore when  $\chi_1\chi_2^{-1}=q^{\val}$,

 If $q-1\neq 0$  and $q+1\neq 0$ in $R$, then $N=0$ or the kernel of $N$   is the $(\chi_2\circ \alpha_F)$-eigenline. 
 
 If $q-1\neq 0$  and $q+1=0 $ in $R$,  then $N=0$, or the kernel of $N$   is the $(\chi_2\circ \alpha_F)$-eigenline,  or the kernel of $N$   is the $(\chi_1\circ \alpha_F)$-eigenline.
 
 If $q-1=0 $,  then $N$ is any nilpotent.
 \end{remark}

 The  local Langlands  $R$-correspondence for $G=GL_2(F)$ is a  canonical bijection  \begin{equation}\label{eq:LL} LL_R:  \Pi \mapsto (\sigma_\Pi) , N_\Pi)
\end{equation}  from the isomorphism classes of the irreducible smooth $R$-representations $\Pi$ of $G$  onto the equivalence classes  of the  
 two-dimensional  Weil-Deligne $R$-representations of $W_F$\footnote{$(\sigma_\Pi , N_\Pi)$  is called the $L$-parameter of $\Pi$}. It  identifies   supercuspidal  $R$-representations  of $G$  and   irreducible 
 two-dimensional    $R$-representations of $W_F$,
 commutes with the automorphisms of $R$ respecting a chosen square root of $ q$,   with  the twist by   smooth $R$-characters  $\chi$ of $F^*$:
 \begin{equation}\label{eq:LLR} LL_{R}(\Pi \otimes (\chi \circ \det))=LL_R(\Pi )\otimes  (\chi \circ \alpha_F).  \end{equation}
   
 The  local Langlands  complex correspondence was proved  by  Kutzko \cite[\S 33]{BH02}.  An  isomorphism  $\mathbb C\simeq \mathbb Q_\ell^{ac}$  and the choice of a square root of $q$ in $ \mathbb Q_\ell^{ac}$  transfers $LL_{\mathbb C}$ to  a local Langlands $\mathbb Q_\ell^{ac}$-correspondence $ LL_{\mathbb Q_\ell^{ac}} $ respecting integrality. Any
irreducible smooth $\mathbb F_\ell^{ac}$-representation $\Pi$ of $G$ lifts to a $\mathbb Q_\ell^{ac}$-representation $\tilde \Pi$ of $G$ (Proposition \ref{pr:ilft}) and 
$ LL_{\mathbb Q_\ell^{ac}} $  descends to a  local Langlands  $\mathbb F_\ell^{ac}$-correspondence  $ LL_{\mathbb F_\ell^{ac}} $  compatible with reduction  modulo $\ell$ in the sense  of \cite[\S 1.8.5]{V01}.   The nilpotent part $N_\Pi$ is  subtle but the semi-simple part $\sigma_\Pi $ is simply the reduction modulo $\ell$ of $\sigma _{\tilde \Pi}$,
 \begin{equation}\label{eq:LLrell} \sigma _{\Pi}= r_\ell (\sigma _{\tilde \Pi}).
 \end{equation}
 The  local Langlands correspondence $LL_R$ of $G$ over $R$ is  deduced from 
 $LL_{\mathbb Q_\ell^{ac}}$ when $\charf_R=0$ and from $LL_{\mathbb F_\ell^{ac}}$ when $\charf_R=\ell$     \cite[\S 3.3]{V97}, \cite[\S I.7-8]{V01}.
  We recall from loc.cit.  a representative $(\sigma_\Pi , N_\Pi)$ of  $LL_R(\Pi)$ for an irreducible smooth $R$-representation $\Pi$ of $G$.
  
   \begin{proposition} \label{ex:1} A) Let $\Pi$ be an irreducible    subquotient    of 
 the un-normalized  $R$-principal series   $\ind_B^G(1)$ of $G $.  
  Then, $\sigma_\Pi =((q^{1/2})^{-\val } \oplus (q^{1/2})^{\val} )\circ \alpha_F$. 
 We have $N_\Pi =0$  if  
 
 $\Pi = 1$   the trivial character  if $q+1\neq 0$ in $R$, $\Pi = \Pi_0$  cuspidal   if $q+1=0$ in $R$.

\noindent Otherwise $N_\Pi \neq 0$.  When $q-1 \neq 0$ in $R$, the kernel of $N_\Pi$ is the 

 $(q^{1/2})^{-\val }\circ \alpha_F$-eigenline if $q+1= 0$ in $R$ and  $\Pi=1$,

$(q^{1/2})^{\val }\circ \alpha_F$-eigenline if $q+1= 0$ in $R$ and  $\Pi=q^{val} \circ \det$, 

$(q^{1/2})^{-\val }\circ \alpha_F$-eigenline if $q+1\neq 0$ in $R$ and $\Pi=\St$ the Steinberg representation.

\medskip B) Let $\Pi$ be the irreducible normalized principal series  $i_B^G(\eta)$, i.e.  $\eta\neq q^{\pm \val} $, with the notation of \eqref{eq:434}. Then   $ \sigma_\Pi =(\eta \oplus 1)\circ \alpha_F $ and $N_\Pi=0$. 
 
\medskip C)  Let  $\Pi$  be a  supercuspidal  $R$-representation of $G$.  Then  $\sigma_{\Pi}  $  is irreducible and $N_\Pi=0$.

\end{proposition}

\subsubsection{} For  a two dimensional semi-simple  smooth $R$-representation $\sigma$ of $W_F$, put 
$$X_{\sigma}=\{ \text{smooth $R$-characters $\chi$ of $F^*$ such that } \ 
 (\chi \circ \alpha_F )\otimes \sigma \simeq \sigma\}.$$  
The square of each  $\chi\in X_{\sigma}$ is trivial because $\dim_R\sigma=2$. We shall compute $X_\sigma$ when $\charf_R\neq 2$. 
When $\charf_R= 2$, $X_\sigma=\{1\}$.

\medskip  To a pair $(E, \xi)$ where $E$  is a quadratic separable extension of $F$ and $\xi$  is a smooth $R$-character  of $E^*$  different from its conjugate $\xi^\tau$ by a generator $\tau$ of  $\Gal(E/F)$
(i.e. $\xi$  is not trivial on   $\Ker N_{E/F} =\{x/x^\tau \ | \ x\in E^*\}$), is associated 
a $2$-dimensional irreducible smooth $R$-representation of $W_F$
$$\sigma(E,\xi)=\ind_{W_E}^{W_F}( \xi \circ \alpha_E).$$ 
 The character $\xi$ is unique modulo $\Gal(E/F)$-conjugation.

When $\charf_R\neq 2$, let $\sigma$ be a two dimensional irreducible  smooth $R$-representation of $W_F$ and $E/F$   a quadratic separable extension. By Clifford's theory  \cite[\S 10, 41.3 Lemma]{BH06}, with the notation \ref{not:etaE}) 
$$ \eta_E  \in X_\sigma \Leftrightarrow \sigma \simeq \sigma(E,\xi) \ \ \text{for some} \ \xi.$$

 \begin{proposition} \label{le:Hs} Assume  $\charf_R\neq 2$.  For a pair $(E, \xi)$ as above,
   $$X_{\sigma(E,\xi)}=\begin{cases} \{1, \eta_E\} \ & \text{ if } (\xi /\xi^\tau)^2 \neq 1\\ 
    \{1, \eta_E, \eta_{E'}, \eta_E \eta_{E'}\} \ & \text{ if } (\xi /\xi^\tau )^2=1, \ \xi /\xi^\tau =\eta_{E'}\circ N_{E/F}\\
    \end{cases}.$$
 For each  biquadratic  separable extension $K /F$, there exists 
a  two dimensional  irreducible smooth $R$-representation $\sigma$ of $W_F$, unique modulo twist by a character,  with 

\noindent $X_\sigma=  \{1, \eta_E, \eta_{E'}, \eta_{E''}\}$ for  the  three quadratic  extensions $E,E',E''$ of $F$ contained in $K$.
\end{proposition}
 \begin{proof}  $\chi \in X_{\sigma(E,\xi)}\Leftrightarrow   (\chi \circ \alpha_F )\otimes \ind_{W_E}^{W_F}( \xi \circ \alpha_E) \simeq \ind_{W_E}^{W_F}( \xi \circ \alpha_E)\Leftrightarrow  \xi (\chi\circ N_{E/F}) =\xi $ or $\xi^\tau$.
 
  $\xi (\chi\circ N_{E/F}) =\xi \Leftrightarrow  \chi$ is trivial on $N_{E/F}(E^*)$, so $\chi=1$ or $\eta_E$. 
  
    $\xi (\chi\circ N_{E/F}) =\xi ^\tau \Leftrightarrow \chi=\eta_{E'}$  for a quadratic separable extension $E'\neq E$  of $F$, as $\chi^2=1$.
    
    \noindent  If $\chi$ satisfies $\xi (\chi\circ N_{E/F}) =\xi ^\tau$, the order of  $ \xi ^\tau/\xi  $ is $2$, $\xi ^\tau/\xi  $ is fixed by $\tau$  and determines  $\chi$ up to multiplication by $\eta_E$. Let $K/F$ be the biquadratic extension  generated by $E$ and $E'$ and $E''/F$  the third quadratic extension  contained in  $K/F$.  We have $\eta_E \eta_{E'}=\eta_{E''}$. Hence the first assertion.
    
The unicity in  the second assertion follows from the fact that  for two smooth $R$-characters $\xi_1, \xi_2$ of $E^*$,
   $\xi_1 ^\tau/\xi _1 = \xi_2 ^\tau/\xi_2  \Leftrightarrow \xi_1=\xi _2(\chi \circ N_{E/F})$ for a smooth  $R$-character $\chi $ of $F^*$. 

The existence in the second assertion is as follows. When $p$ is odd, there is a unique biquadratic extension $K/F$ of $F$. Let $E/F$ be the  unramified quadratic extension.
We take $\sigma=\sigma(E, \xi)$  where    $\xi$ is  the character of $E^*$ trivial  on $1+p_FO_E$, $\xi(p_F)= -1$ and $\xi (x)=x^{(q+1)/2}$ if  $x^{q^2-1}=1$,  satisfies $\xi ^\tau/\xi \neq 1$ and $(\xi ^\tau/\xi)^2 =1$ hence  $\xi ^\tau/\xi =\eta_{E'}\circ  N_{E/F} =\eta_E \eta_{E'}\circ  N_{E/F}$ for $E'/F$ ramified.    When $p=2$, given two different quadratic separable extensions $ E'/F$  and $E/F$, there exists a smooth $R$-character $\xi$ of $E^*$ such that $\xi ^\tau/\xi =\eta_{E'}\circ  N_{E/F}=\eta_E \eta_{E'}\circ  N_{E/F}$, because  $\charf_R\neq 2$, and this is known when $R=\mathbb C$ (\cite[\S41]{BH06} when $ p\neq 2$, but the proof does not use $p\neq 2$)\footnote{We  gave a direct proof when $p$ is odd,  this was unnecessary}. \footnote{When $p$ is odd and $\charf_R=2$, there no $\xi$ such that $\sigma(E,\xi)$ is induced from a character of $W_{E'}$ for a quadratic  extension $E'/F$  distinct from $E/F$.}
   \end{proof}
 \begin{remark}   Let  $\Pi$ be a supercuspidal $R$-representation of $G$. Then $\Pi$ has level $0$ (resp.   $L(\Pi)$ has level $0$),  if and only if $\sigma_\Pi  = \ind_{W_E}^{W_F}(\xi \circ \alpha_E)$
where $E/F $ is quadratic unramified and $\xi$ is a tame character of $E^*$ (resp. $\xi ^\tau/\xi $ is a tame character of $E^*$ where $\tau$ is the non-trivial element of $\Gal(E/F)$).
\end{remark}

 \begin{remark} \label{Hs}  Assume $\charf_R\neq 2$. Let $\sigma=\chi_1\circ \alpha_F \oplus \chi_2\circ \alpha_F$  be a  reducible two dimensional  semi-simple smooth $R$-representation of $W_F$.
 Then   $\chi \circ \alpha_F\in X_\sigma \Leftrightarrow   \{ \chi \chi_1, \chi \chi_2\}=\{\chi_1, \chi_2 \} \Leftrightarrow  \chi=1$ or   $\chi \chi_1=\chi_2,  \chi \chi_2=\chi_1 \Leftrightarrow \chi=1$ or $\chi=\chi_2 \chi_1^{-1} , \chi^2=1$.
If  $\chi_1 \chi_2^{-1}= \eta_E $ for a  quadratic separable extension $ E/F$, then  $ X_\sigma  =\{1,\eta_E\}$. Otherwise, $ X_\sigma =\{1\}$.
   \end{remark}

\subsubsection{}Application to the cuspidal $L$-packets.

For  a two dimensional  Weil-Deligne $R$-representation $(\sigma,N) $ of $W_F$, put 
 $X_{(\sigma, N)}$ for the group of $\chi\in X_\sigma$ such that there exists an isomorphism of $\chi\otimes \sigma$ onto $\sigma$ preserving $N$.
 For any irreducible  $R$-representation $\Pi$ of $G$,   applying 
  the formulas \eqref{eq:LL}, \eqref{eq:LLR} 
  and   \eqref{eq:LPi}  we obtain :
   \begin{flalign}\label{eq:LLr}X_\Pi= \{ \chi \circ \det \ | \ \chi\in X_{(\sigma_\Pi , N_\Pi)}\}. \  \  \ \ \ \ \ \   \ \ \ \ \ \ \ \ \ \ \ \ \ \ \ \ \ \ \ \ \ \ \ \ \ \  \\
    \label{eq:LLr2} \text{When $\charf_R\neq 2$, the cardinality of the $L$-packet $L(\Pi)$ is $|X_{\sigma_\Pi }|$}.  \end{flalign}
     
\begin{proposition}   \label{cor:bq1} 1) When $\charf_R\neq 2$, we have:
\begin{itemize}
\item The cardinality of a cuspidal  $L$-packet is $1,2$ or $4$.
\item The map $L(\Pi)\mapsto E_\Pi$ is a bijection from the cuspidal  $L$-packets of size $4$   to the biquadratic separable extensions of $F$.
\end{itemize}

2) There is a bijection   from the cuspidal  $L$-packets of size $4$   to the biquadratic separable extensions of $F$, 
sending  the unique cuspidal  $L$-packet of size $4$ to the  unique biquadratic separable extension of $F$  when $\charf_R= 2$, and equal to the map  $L(\Pi)\mapsto E_\Pi$  when $\charf_R\neq 2$.\
   \end{proposition}  
  \begin{proof} a) Assume $\charf_R\neq 2$.  If $ \Pi$ is cuspidal and $X_\Pi\neq \{1\}$ then  $\eta_E\in X_\Pi$ for some quadratic separable extension $E/F$, $\sigma_\Pi =\sigma(E,\xi)$ for some $\xi$ and  $|X_{\sigma(E,\xi)}|=2$ or $4$ by Proposition \ref{le:Hs}. 
 When $p=2$ then the map  is a bijection by Proposition \ref{le:Hs} via the local Langlands correspondence.
 
  b) Assume   $p$ is odd (and $\charf_R\neq p$). There is a unique  biquadratic separable extension of $F$ and an unique cuspidal  $L$-packet of size $4$ (Corollary \ref{cor:2}). 
 
 c) As $p$ is odd when $\charf_R= 2$,   the proposition follows from a) and b).
 \end{proof}

When $R=\mathbb F_\ell ^{ac}$ and $\ell\neq p$, it is well known 
 that  an irreducible smooth  $\mathbb F_\ell ^{ac}$-representation $\sigma$  of $W_F$  of dimension $2$  lifts to    an integral    irreducible smooth $\mathbb Q_\ell^{ac}$-representation $\tilde \sigma$ of $W_F$ \footnote{$\sigma $ extends to a $\mathbb F_\ell ^{ac}$-representation of the Galois group $\Gal_F$. As $\Gal_F$ is solvable this representation 
 lifts to a $\mathbb Q_\ell ^{ac}$-representation of $\Gal_F$ that one restricts to $W_F$ to get $\tilde \sigma$}. 
 The 
 order of $X_{\tilde \sigma}$ is  at most  to the order of $X_{ \sigma}$. We give now all the cases where the orders are different.
 
 \begin{theorem}  \label{th:liftw}  Assume    $\ell\neq 2$. 
 
 1) Let  $\tilde \sigma$ be a lift   to $\mathbb Q_\ell ^{ac}$ of a two-dimensional  irreducible   smooth $\mathbb F_\ell ^{ac}$-representation $\sigma$ of $W_F$.   The cardinalities of $X_\sigma$ and of  $X_{\tilde \sigma}$ are different if and only if $ |X_\sigma|=4,  |X_{\tilde \sigma}|=2 $, and this happens if and only if 
 $$ 
  p=2, \ \ell \ \text{divides} \ q+1, \ \tilde \sigma=\ind_{W_E}^{W_F}(\tilde \xi \circ  \alpha_E)$$
    where  
$E/F$ is a quadratic  unramified extension,  $\tilde \xi$ a  smooth $\mathbb Q_\ell^{ac}$-character of $E^*$ such that 
\begin{enumerate}
\item the order of $\tilde \xi ^\tau/\tilde \xi $ on $1+P_E$ is $2$ where $\Gal(E/F)=\{1,\tau\}$.
  
  \item
 $  \tilde \xi (b)\neq 1, \ \tilde \xi (b)^{\ell^{s}}=1$ 
 for  a root of unity $b\in E^*$ of order $q+1$, and $s$ is a positive integer such that  $\ell^s$ divides  $q+1$.
 \end{enumerate}

2)  Each  irreducible smooth $\mathbb F_\ell ^{ac}$-representation $\sigma$ of $W_F$ of dimension $2$ admits   a lift $\tilde \sigma$  to $\mathbb Q_\ell ^{ac}$  such that $|X_{\tilde \sigma}|  = |X_\sigma|$.
     \end{theorem}

\begin{proof} 1) Let $\Pi$ be  the  supercuspidal smooth $\mathbb F_\ell^{ac}$-representation  of $G$  and $\tilde \Pi$  the  integral  supercuspidal smooth $\mathbb Q_\ell^{ac}$-representation of $G$ lifting $\Pi$ such that
 $\sigma= \sigma_\Pi , \tilde \sigma = \sigma_{\tilde \Pi}$ by the Langlands correspondence \eqref{eq:LL}. We have $|X_{ \Pi}|=|X_{ \sigma}|, |X_{\tilde  \Pi}|=|X_{ \tilde \sigma}| $ \eqref{eq:LLr}.  
 By Proposition \ref{prop:liftc}, $|X_{ \sigma}|= |X_{\tilde \sigma}|$ or $2\,|X_{\tilde \sigma}|$, except may be when $p=2$ and $\tilde \Pi$ has positive level.  In this exceptional case,  $\eta_E\in X_{\tilde \Pi}$.  By Proposition  \ref{Hs}, $|X_{\sigma}|$ and $|X_{\tilde \sigma}| $ are equal to $1,2$ or $4$.
 Therefore,  $|X_{\sigma}|\neq |X_{\tilde \sigma}| $ is equivalent to $|X_\sigma|=4$ and $ |X_{\tilde \sigma}|=2$.  

When  $|X_\sigma|=4$ and $ |X_{\tilde \sigma}|=2$, 
   $\sigma=\ind_{W_E}^{W_F}\xi$, $\tilde \sigma=\ind_{W_E}^{W_F}\tilde \xi$   for a quadratic  unramified extension $E/F$, an integral   smooth $\mathbb Q_\ell^{ac}$-character $\tilde \xi$ of $E^*$,  of  reduction $\xi$ modulo $\ell$, with
 $\xi/\xi^\tau \neq 1$  where $\tau$ is the generator $\tau$ of $\Gal(E/F)$, 
and  $(\xi/\xi^\tau )^2=1$. This implies  $(\tilde \xi/\tilde \xi^\tau )^2=1$ on  $p_F^{\mathbb Z}(1+P_E)$  because $\ell\neq p$.
We have $E^*=p_F^{\mathbb Z}(1+P_E)\mu_E$ where $\mu_E=\{x\in E^* \ | \ x^{q^2-1}=1\}$.  We have $\tau(x)=x^q$ if $x\in \mu_E $.  The group $\{x^{q-1} \ | \ x\in \mu_E\} $ is generated by an arbitrary root of unity $b \in E^*  $ of  order $q+1$. So     $$ (\tilde \xi/\tilde \xi^\tau )^2=1 \Leftrightarrow \tilde \xi (b)^{2 }=1 \Leftrightarrow |X_{\tilde \sigma}|=4, \ \  \ \  (\tilde \xi/\tilde \xi^\tau )^2\neq 1 \Leftrightarrow \tilde \xi (b)^{2 }\neq 1 \Leftrightarrow |X_{\tilde \sigma}|=2.$$  
In the exceptional case, $p=2$ hence $\ell$ is odd and $ \xi (b)^{2 }=1 $  implies $\xi (b)=1$ (and conversely), or equivalently, the order of $\tilde \xi (b) $  is a power of $\ell$ dividing $q+1$.
 There exists  a lift  $\tilde \xi$ of $\xi$ such that   $\tilde \xi (b)\neq 1$  if and only if $\ell$ divides $q+1$. 
 
 2) Given a positive integer  $s$, each element $x\in  (\mathbb F_\ell ^{ac})^*, x\neq 1,$   is the reduction modulo $\ell$ of an element $\tilde x\in  (\mathbb Z_\ell ^{ac})^*$ such that $\tilde x^{\ell^s}\neq 1$.
\end{proof}  
  

 \begin{corollary} \label{prop:redu1}  
 1)  The reduction modulo $\ell$ of a   supercuspidal  $\mathbb Q_\ell^{ac}$-representation $ \tilde \pi$ of $G'$ has length $\leq 2$. It has length $2$ if and only if 
$$p=2, \ \ell \ \text{divides} \ q+1, \ \sigma_{\tilde \Pi} =\ind_{W_E}^{W_F}(\tilde \xi \circ \alpha_E), $$     where  $\tilde \pi\in L(\tilde \Pi)$, $E/F$ is  unramified, and $\tilde \xi $ is a smooth $\mathbb Q_\ell^{ac}$-character of  $E^*$  such that: 
 \begin{enumerate}

\item the order of $\tilde \xi ^\tau/\tilde \xi $ on $1+P_E$ is $2$ where $\Gal(E/F)=\{1,\tau\}$.
  
  \item
 $  \tilde \xi (b)\neq 1, \ \tilde \xi (b)^{\ell^{s}}=1$ 
 for  a root of unity $b\in E^*$ of order $q+1$, and  $\ell^s$ divides  $q+1$.
 \end{enumerate}

 2)  Each cuspidal   $\mathbb F_\ell ^{ac}$-representation $\pi$ of $G'$  is the reduction modulo $\ell$ of  an integral  supercuspidal  $\mathbb Q_\ell^{ac}$-representation  of $G'$. 
\end{corollary}
\begin{proof}   1) follows from  
 
 -  if $\ell\neq 2$,  Theorem \ref{th:liftw} 1),   the formula \eqref{eq:red'}, and the local Langlands correspondence,

- if $\ell=2$, Proposition \ref{prop:liftc}  1).

2) follows from 

- if  $p=2$ and   $\pi$ is   in an $L$-packet $L(\Pi)$ with $\Pi$ minimal of positive level (hence $\pi$ is supercuspidal, see Corollary \ref{cor:v}) with $E_\Pi/F$  unramified, then $\pi$ lifts to $\mathbb Q_\ell^{ac}$ by
Theorem \ref{th:liftw} 2),   the formula \eqref{eq:red'}, and the local Langlands correspondence.

-  otherwise, Proposition \ref{prop:liftc} 2). 

 \end{proof}
  \begin{remark}

Assume $p\neq 2$.
A pair $(E,\theta)$ where  $E /F$ is a quadratic 
extension of $F$  and $\theta$ is a smooth $R$-character of $E^*$,  is called  {\sl admissible} (\cite{BH06} \S 18.2) if :  

1) $\theta$ does not factorize through $N_{E/F}$ (equivalently is regular with respect to $\Gal(E/F)$) 

2)  If $\theta|_{ 1+P_E }$ does factorize through $N_{E/F}$ (equivalently is invariant under $\Gal(E/F)$), 
then $E/F$ is unramified. 

To an admissible pair $(E,\theta)$ is associated  the two-dimensional  irreducible  $R$-representation $\sigma (E,\theta)=\ind_{W_E}^{W_F} (\theta \circ \alpha_E)$ of $W_F$, and when $R=\mathbb C$
 an explicitly constructed  supercuspidal representation $\pi (E,\theta)$   of $G$  (loc.cit. \S 19). 
 Isomorphism classes of   supercuspidal  complex representations of 
$G$, are parametrized by isomorphism classes of admissible pairs $(E,\theta)$  (loc.cit.\S 20.2).
  The Langlands local  correspondence  sends  $\pi  (E,\theta)$  to  $\sigma (E, \mu \theta) $ where the explicit  ``rectifyer''  $\mu$   is a tame character of $E^*$  depending only on  $\theta|_{1+P_E}$. As the  Langlands correspondence is compatible with automorphisms of $\mathbb C$ preserving $\sqrt  q$, 
 the previous classification in terms of admissible pairs transfers to $R$-representations 
where $R$ is an algebraically closed field of characteristic 0 (given a choice of square root of $q$ in $R$). 
The classification and correspondence for $ R=\mathbb Q_\ell^{ac}$ reduce modulo $\ell \neq p$ (the integrality property for a pair 
$(E,\theta)$ is that $\theta$ takes integral values) to get a similar classification of supercuspidal $\mathbb F_\ell^{ac}$-representations 
  in terms of admissible pairs. The  integral admissible pairs   over $\mathbb Q_\ell^{ac}$ 
that do not reduce to admissible pairs over $\mathbb F_\ell^{ac}$,  yield under reduction  cuspidal but not supercuspidal 
$\mathbb F_\ell^{ac}$-representations. 
\end{remark}

 \subsection{Principal series}\label{S:434}

Notations of \S \ref{SL2}.   We identify     a  smooth $R$-character $\eta$ of $T'$  with  a $R$-character  of $F^*$ and   of $T $ by:
\begin{equation}\label{eq:434}
\eta(\diag(a,d)) = \eta(\diag(a,a^{-1})) = \eta(a)   \ \ \ (a,d\in F^*).
\end{equation}
Proposition \ref{prop:spG} describes $i_{B}^ {G} (\eta)$.  The transfer of the properties (i) to (iv) to   $$i_{B'}^ {G'} (\eta)= (i_{B}^ {G} (\eta))|_{G'}$$is easy and gives:
 
\begin{enumerate}
\item  For  smooth $R$-characters $\eta, \eta'$ of $F^*$,  $[i_{B'}^{G'}(\eta)] $  and  $[i_{B'}^{G'}(\eta')] $  are disjoint   if $ \eta'\neq \eta^{\pm 1}$, and equal if $ \eta'=\eta^{\pm 1}$.
  \item  The smooth dual of   $i_{B'}^{G'}(\eta)$  is   $i_{B'}^{G'}(\eta^{-1})$.
\item    $(i_{B'}^{G'}(\eta ))_U$ has dimension $2$, contains $\eta^{-1} $ and $\eta$ is a quotient.
 \item    $\dim W_Y(i_{B'}^{G'}(\eta))=1$  for all   $ Y\neq 0$. 
\end{enumerate} 
The transfer  of   the  properties  (v)  and (vi) is harder. 

\begin{proposition} \label{v} \begin{enumerate}
  \item $ i_{B'}^{G'} (\eta)$ is reducible if and only if 
  $\eta=q^{\pm \val}$, or $\eta\neq 1$ and $ \eta^2=1$.
   \item  When  $\charf_R\neq 2$, $ i_{B'}^{G'} (\eta_E) $ is  semi-simple of length  $2$, when $E/F$ is a quadratic separable extension,   which is ramified if $q+1=0$ in $R$.
\item When  $\charf_R=2$,      the only reducible   principal series is  $ i_{B'}^{G'} (1)=\ind_{B'}^{G'}(1)$. 
\item  The length of $ i_{B'}^{G'}(q^{-\val})$ and of  $ i_{B'}^{G'}(q^{\val})=\ind_{B'}^{G'}( 1)$  is
$$\lg(\ind_{B'}^{G'} 1)=\begin{cases} 2 \ &\text{if $q+1\neq 0$ in $R$}   \cr
 4  \ &\text{if $q+1= 0$ in $R$ and $\charf_R\neq 2$} \cr
 6 \ &\text{if $\charf_R=2$}
 \end{cases}.
 $$  
\end{enumerate} 
 \end{proposition}
 Note that $\charf_R=2$ implies $q+1= 0$ in $R$.
 \begin{proof}  We show (i)  (ii) and (iii). 

 If $i_B^G( \eta)$ is reducible, then  its restriction 
$ i_{B'}^{G'} (\eta)$ to $G'$ is reducible.   By Proposition \ref{prop:spG},   $i_B^G( \eta)$ is reducible if and only if $\eta =q^{\pm \val}$.  
 
 Assume $i_B^G( \eta)$ irreducible, i.e.  $\eta \neq q^{\pm \val}$. If  $\charf_R\neq 2$,   we have  $X_{ i_{B}^{G }(\eta)}=2 $   if and only  
 $\eta\neq 1$  and $ \eta^2= 1$ by the Langlands correspondence and  Remark  \ref{Hs}\footnote{or directly  because for a  smooth $R$-character $\chi$ of $F^*$,  the property (i) in Proposition \ref{prop:spG} implies     $
(\chi \circ \det)\otimes i_B^G (\eta) \simeq i_B^G (\eta) \Leftrightarrow \chi \eta =\eta \text{ or }  \eta^{-1} \Leftrightarrow  \chi=1 \text{ or } \chi=\eta  \text{ and }  \eta^{2}=1.$}. We have $\eta\neq 1, \eta^2=1$ if and only if $\eta=\eta_E$ for a quadratic separable extension $E/F$, which is ramified if $q+1=0$ in $R$ (notation \ref{not:etaE}) as $\eta \neq q^{\pm \val}$.
 If  $\charf_R=2$, then $p$ is odd, $\eta\neq 1$, and   $i_{B'}^{G '}(\eta) $ is irreducible. Indeed, 
the irreducible components of $i_{B'}^{G '}(\eta) $ are $B$-conjugate  (\S \ref{mrho}). They  give a partition of the  set of irreducible components of $(i_{B'}^{G '}(\eta))|_{B'}$. The character $\eta$ appears with multiplicity $1$ as $\eta \neq \eta^{-1}$, but as it is fixed by $B$,  the partition is trivial, i.e. $i_{B'}^{G '}(\eta) $ is irreducible. 

 (iv) \cite[Example 3.11 Method 2]{Cui23}. We give a proof for the convenience of the reader. When $q+1\neq 0$ in $R$,  the restriction to $G'$ of the Steinberg representation $\St$ of $G$  is irreducible, otherwise it would contain a cuspidal representation as $\dim_R \St_U=1$ which is impossible by \eqref{eq:ss}. When $q+1=0$ in $R$,  the cuspidal $R$-representation $\Pi_0$ (see Proposition \ref{prop:spG})  is   induced from the inflation to $Z GL_2(O_F)$ of a cuspidal  $R$-representation $\sigma_0$ of $GL_2(k_F)$. By \eqref{eq:case1},  $\lg(\Pi_0|_{G'})=2 \, \lg (\sigma_0|_{SL_2 ( k_F)} )$.  The representation  $\sigma_0|_{SL_2 ( k_F)} $ is 
 irreducible if $ \charf_R\neq 2$, and  has length   $2$ if  $ \charf_R= 2$ (Appendix). 
\end{proof}

 \begin{corollary}\label{cor:v} The non-supercuspidal smooth $R$-representations of $G'$ are:

  The trivial character. 

If $q+1\neq 0$ in $R$, the Steinberg $R$-representation $\st=\St|_{G'}$.

The principal series $i_{B'}^{G'}(\eta)$  for the  smooth $R$-characters $\eta$ of $F^*$ with $\eta \neq q^{\pm \val}$ and $\eta\neq \eta_E$  for any quadratic  separable extension $E/F$.

If  $\charf_R\neq 2$,  the two irreducible   components  $\pi_E^{\pm}$ of $i_{B'}^{G'}(\eta_E)$ for a quadratic separable  extension $E/F$, which is ramified if $q+1=0$ in $R$.

If  $\charf_R\neq 2$ and $q+1=0$ in $R$,  the two irreducible components  of $\Pi_0|_{G'}$. 

If $\charf_R=2$,  the four irreducible components  of   $\Pi_0|_{G'}$.

\noindent The only isomorphisms between those representations are $i_{B'}^{G'}(\eta) \simeq i_{B'}^{G'}(\eta^{-1})$ for the irreducible principal series.
 \end{corollary}  
 
We get for non supercuspidal $L$-packets:

\begin{proposition}\label{pro:cnsc} 
  When  $q+1=0$ in $R$, there is 
a unique cuspidal non-supercuspidal  $L$-packet. Its  size is 
$\begin{cases} 2  \ & \text{if} \ \charf_R\neq 2 \\  4 \ & \text{if} \ \charf_R= 2 \end{cases} $.
 
 When  $\charf_R= 2$,  every non-cuspidal $L$-packet is a singleton.

  When  $\charf_R\neq 2$, the non-cuspidal $L$-packets are singletons or of size   $2$. Those  of size $2$ are in bijection with the isomorphism classes of  the quadratic  separable extensions of $F$.
   \end{proposition}
This proposition and Corollary \ref{cor:2} imply :

\begin{corollary}\label{cor:tout}  The $L$-packets of size $4$ are cuspidal. \end{corollary}
 
  We consider now the reduction modulo  a prime number $\ell  \neq p$. A 
 non-cuspidal irreducible  $\mathbb Q_\ell^{ac}$-representation   $\tilde \pi$  of $G'$     is integral  except when $\tilde \pi \simeq i_{B'}^{G'}(\tilde \eta)$  for a non-integral smooth $\mathbb Q_\ell^{ac}$-character $\tilde \eta$ of $F^*$.
When $\tilde \pi$ is integral, we deduce from Corollary \ref{cor:v} the length of  the reduction 
 $r_\ell(\tilde \pi)$ modulo $\ell$ of $\tilde \pi$.
\begin{proposition}\label{prop:nc1} 1) The reduction $r_\ell(\tilde \pi)$ modulo $\ell$ of $\tilde \pi$ irreducible non-cuspidal and   integral  is irreducible with the exceptions:
 
\medskip  If $\ell=2$,   $\lg (r_\ell (\tilde st)) =5$, $\lg (r_\ell (\tilde \pi_E^\pm)) =3$, $\lg (r_\ell (i_{B'}^{G'}(\tilde \eta))) =6$   for $\tilde \eta$ of order  a finite power of $\ell$.
 
If  $\ell\neq 2$ and $\ell$ divides $q+1$,    $\lg (r_\ell (\tilde st)) =3$,
$\lg(r_\ell (i_{B'}^{G'}(\tilde \eta))) =4$  for $\tilde \eta$ of order  a finite power of $\ell$,  $\lg(r_\ell (i_{B'}^{G'}(\tilde \eta))) =2$ if $\tilde \eta = \tilde  \eta_{E} \tilde \xi$, for a  ramified quadratic separable extension $E/F$ and a character $\tilde \xi$ of order a power of $\ell$.

\medskip  2)  Each non-cuspidal  irreducible  $\mathbb F_\ell ^{ac}$-representation  of $G'$  is the reduction modulo $\ell$ of  an integral  non-cuspidal irreducible  $\mathbb Q_\ell^{ac}$-representation  of $G'$. 
\end{proposition}
 
 \section{Local Langlands $R$-correspondence for $SL_2(F)$ } 
\subsubsection{} If $(\sigma, N)$  is  a two-dimensional  Deligne $R$-representation  of the Weil group $W_F$  (\S \ref{S:2d}),  a choice of  a basis of the space of $\sigma$ gives a Deligne morphism of $W_F $ into $GL_2(R)$ \footnote{We use the same notation $(\sigma, N)$ for the Deligne morphism of $W_F $ into $GL_2(R)$}. In this way equivalence classes of two-dimensional  Deligne $R$-representation  of $W_F$ identify  with Deligne morphism of $W_F $ into $GL_2(R)$, up to $GL_2(R)$-conjugacy.

 A Deligne morphism of $W_F $ into $PGL_2(R)$ is a  pair $(\sigma,N)$  where $\sigma:W_F\to PGL_2(R)$ is a smooth morphism, semisimple in the sense that if  $\sigma (W_F)$ is in a parabolic subgroup $ P$ then 
it is in a Levi of $P$, and $N$ is a nilpotent\footnote{$N$ is nilpotent in $\Lie(PGL_2(R))$ if  the Zariski closure of the $PGL_2(R)$-orbit of  $N$  contains   $0$} element in $\Lie(PGL_2(R))$ with the usual 
requirement \eqref{requi}. We say that $(\sigma,N)$  is irreducible if $\sigma (W_F)$    is not contained in a 
proper parabolic subgroup (that means that $N=0$  and the inverse image of $\sigma(W_F)$ 
in $GL_2(R)$  acts irreducibly on $R^2)$. 
The question arises whether  a  Deligne morphism $(\sigma ,N) $ of $W_F $ into $PGL_2(R)$  lifts to a two-dimensional Weil-Deligne $R$-representation. 

When $(\sigma ,N)$ is reducible, we may assume 
that  $\sigma$ takes value in the diagonal torus of $PGL_2(R)$, and that $N$ is upper triangular.  The map  $x\to \diag(x,1) $ modulo scalars is an isomorphism from $R^*$ to this torus, so $\sigma$
comes from an $R$-character $\chi$  of $W_F$, and $\sigma$ lifts to the two-dimensional $\chi \oplus 1$. That deals 
with the case where $N=0$. When $N \neq 0$, then  $(\sigma ,N)$ lifts 
to  $(q^{-\val}\oplus 1,N)$.  

The following lemma answers  the question, more generally for irreducible Deligne morphisms 
of $W_F$ into $PGL_n(R)$    for  integers $n\geq 2$ (the definitions in the first alinea  for $n=2$ generalize to $n>2$). 

\begin{lemma} Any irreducible smooth morphism $\rho:W_F \to PGL_n(R)$   has 
finite image and its natural extension to $\Gal_F$ lifts to an irreducible smooth $R$-representation 
of $\Gal_F$ of dimension $n$. 
\end{lemma}

\begin{proof}  Because the inertia group $I_F$ of $W_F$ is profinite and $\rho$  is smooth, $\rho(I_F)$ is finite.  
Let $\phi$ be a Frobenius element in $W_F$. If  the order of $\rho(\phi)$  is finite, then  $\rho(W_F)$ is finite, so $\rho$  extends by continuity to a smooth $R$-representation $\rho'$ of $\Gal_F$. 
  The proof of Tate's theorem (\cite{S77} \S 6.5) applies with $R$ instead of  $\mathbb C$ and that 
shows that $\rho'$ lifts to a smooth $R$-representation of $\Gal_F$.  
Let us show that  $\rho(\phi)$ has finite order. Since  $\rho(\phi)$   acts by conjugation on $\rho(I_F)$ which
  is finite, a power $\rho(\phi^d)$  for some positive $d$ acts trivially  on $\rho(I_F)$. 
But it also acts trivially on $\rho(\phi)$, hence on all of $\rho(W_F)$. Let $A\in GL_n(R) $ be 
a lift of $\rho(\phi^d)$. For $B\in GL_n(R)$, the commutator $(A,B)$ depends only 
on the image of $B$ in $PGL_n(R)$, and if $B$ has image $\rho(i)$ for   $i\in  I_F$, then 
$(A,B)$ is a scalar $\mu(i)$. If $B' \in GL_n(R)$ has image $\rho(i')$ for  $i'\in  I_F$, then 
$A(BB')A^{-1}= ABA^{-1}AB'A^{-1}$, giving $\mu(ii')=\mu(i)\mu(i')$, so conjugation by $A$ induces 
a morphism  $\mu: I_F\to R^*$. Since  $\rho(I_F)$ is finite,  a power $A^e$ for some positive $e$ 
commutes with the inverse image $J$ in $GL_n(R)$ of $\rho(W_F)$. Let $V$ be an eigenspace of $A^e$. 
It is stable under  $J$. If $V\neq R^n$, that yields a proper parabolic subgroup 
$P$ (the image in $ PGL_n(R)$ of the stabilizer of $V$) of $PGL_n(R)$ which contains $\rho(W_F)$, contrary 
to  the hypothesis. So $A^e$ is scalar, which implies that $\rho(\phi)$  has finite order dividing $de$. 
\end{proof}
 Two  $2$-dimensional Deligne $R$-representations  of $W_F$ in $GL_2(R)$   are twists  of each other by a smooth $R$-character of $W_F$ if and only if they give the same Deligne morphism of $W_F$ in  $PGL_2(R)$ 
if and only if the two corresponding irreducible smooth $R$-representations $\Pi, \Pi'$ of $G$   are twists  of each other by a smooth $R$-character of $G$  \eqref{eq:LLR} if and only if 
$\Pi$ and $\Pi'$ define the same $L$-packet $L(\Pi)=L(\Pi')$ of  irreducible smooth $R$-representations 
of $G'$ \eqref{re:scL}.
  
\subsubsection{}   From the above the local Langlands correspondence for $G$ induces  a bijection between $L$-packets of  irreducible smooth $R$-representations 
of $G'$ and   Deligne morphisms of  $W_F$ in $ PGL_2(R)$ up to $PGL_2(R)$-conjugacy. 
We would like to understand the internal structure of a given packet 
  in terms of an  associated Deligne morphism $W_F\to PGL_2(R)$ (called its   $L$-parameter).
  
   Let $\Pi$ be an  irreducible smooth $R$-representation of $G$.
The $L$-packet $L(\Pi)$ is a principal homogeneous space of $G/G_\Pi$.
 The packet containing the trivial representation of $G'$ is a singleton, so the parametrization is trivial. 
When $L(\Pi)$ is a packet of infinite-dimensional representations of $G'$ we take as a base point in $L(\Pi)$ the element   with non-zero Whittaker model
with respect to the character $\psi$ of $F$ (that is, $\theta_0$  of $U$) fixed in \S \ref{ss:Wh}. 
Let $C_\Pi$ denote  the centralizer  of the   image    in   $PGL_2(R)$ of a  Deligne morphism $(\sigma_\Pi , N_\Pi)$ of $W_F$  in   $GL_2(R)$ associated to $\Pi$,  and $S_\Pi $ the component group  of  $C_\Pi$. We shall compute $C_\Pi$ and $S_\Pi$, and when $\charf_R\neq 2$ we shall construct a canonical isomorphism from $G/G_\pi$ onto the  $R$-characters of $S_\Pi$. In this  way
 we get  an enhanced local 
Langlands correspondence  for $SL_2(F)$ in the sense of \cite{ABPS16}, \cite{AMPS17} if   $\charf_R\neq 2$ but not if  $\charf_R=2$.  
 J.-F. Dat tells us that  our results for  $\charf_R=2$ should still be compatible with the stacky approach
of Fargues and  Scholze to the semisimple Langlands correspondence. For example, for a supercuspidal $R$-representation $\Pi$ 
of $G$, the two components of  $\Pi|_{G'}$ should be indexed by the two irreducible $R$-representations
of the group scheme $\mu_2$.

 The group of  $R$-characters of $G/G_\Pi$ is $X_\Pi$, and $X_\Pi= \{ \chi \circ \det \ | \ \chi\in X_{(\sigma_\Pi , N_\Pi)}\}$ \eqref{eq:LLr}. We now construct a homomorphism $\varphi: X_{(\sigma_\Pi , N_\Pi)} \to S_\Pi$.
Let $\chi\in X_{(\sigma_\Pi , N_\Pi)} $. By definition, there exists $A\in GL_2(R)$ such that $AN_\Pi= N_\Pi$ and $A \,\sigma_\Pi (w) \,A^{-1}=\chi(w) \,\sigma_\Pi (w) $ for $w\in W_F$. The image $\overline A$ of $A$ in $PGL_2(R)$ belongs to $C_\Pi$ and we shall show that its image $\varphi (\chi)$ in $S_\Pi$ does not depend on the choice of $A$.

 \begin{theorem} \label{prop:LLSL2} The map  $\varphi: X_{(\sigma_\Pi , N_\Pi)} \to S_\Pi$ is a group isomorphism, and   $S_\Pi=\{1\}, \mathbb Z/ 2 \mathbb Z$ or $ \mathbb Z/ 2 \mathbb Z\times \mathbb Z/ 2 \mathbb Z$.
 
  When $\charf_R=2$, $S_\Pi=\{1\}$ for each $\Pi$,  but the length of   $\Pi|_{G'}$ is 

$1$ if $\Pi$ is not cuspidal,

$2$ if $\Pi$ is supercuspidal,

$4$ if $\Pi$ is   cuspidal not supercuspidal. \end{theorem}

   \begin{proof} A) Let $\Pi$ be a  supercuspidal  $R$-representation of $G$. Then $\sigma_\Pi  $ is irreducible and $N_\Pi=0$ (Proposition \ref{ex:1}).

When $\charf_R\neq 2$,   in \cite[Proposition 6.4]{CLL23},  an isomorphism $\phi:X_{\sigma_\Pi}\to C_\Pi$ is constructed when  $\charf_F\neq 2 $, but the proof does not use this hypothesis.  
 This implies $C_\Pi=S_\Pi$. One checks that $\varphi (\chi)=\phi(\chi)$ for $\chi \in X_{\sigma_\Pi }$, 
an isomorphism. 

When $\charf_R=2$, $p$ is odd,  the  cardinality of   $L(\Pi)$ is $2$   or $4$  (Propositions \ref{prop:ty1}, \ref{prop:3.12}),   $\sigma_\Pi =\ind_{W_E}^{W_F}(\theta) $ where $E/F$ is a quadratic separable extension and $\theta$
 a smooth $R$-character of $W_E$ (or equivalently of $E^*$) different from its conjugate $\theta^\tau$ by a generator $\tau$ of $\Gal(E/F)$.   The character  $\theta^\tau/\theta$ has finite odd order, say $m$, and
   $\sigma_\Pi (W_F)\subset GL_2(R) $  is a dihedral group of order $2m$, generated by a matrix
$ \begin{pmatrix} a &0\\0&a^{-1}
\end{pmatrix}$ of order $m$ and  $\begin{pmatrix} 0 &1\\1&0
\end{pmatrix}$ 
modulo conjugation in $GL_2(R)$. So $C_\Pi =\{1\}$ and there is no enhanced correspondence.

\medskip B)   Let $\Pi=i_B^G(\eta)$ be an irreducible normalized principal series  with the notation of \eqref{eq:434}, with  $\eta\neq q^{\pm \val} $. The cardinality of $L(\Pi)$  is $2$ if $\eta\neq 1, \eta^2=1$, and 
$L(\Pi)$ is a singleton otherwise. We have 
 $\sigma_{\Pi}=(\eta \oplus 1)\circ \alpha_F$, $N_{\Pi}=0$ (Proposition \ref{ex:1})  and we easily see that  $C_{\Pi}$  is:

 $PGL_2(R)$ when $\eta =1$, so $S_\Pi=\{1\}$.

 The diagonal torus when $\eta\neq 1, \eta^2\neq 1$, $S_\Pi=\{1\}$.

The normalizer of the trivial torus when  $\eta\neq 1, \eta^2=1$, so $\charf_R\neq 2$ and  $S_\Pi=  \mathbb Z/ 2 \mathbb Z$. We have $X_\Pi=\{1, \eta \circ \det\}$  (Remark \ref{Hs}) and $\varphi(\eta)$ is not trivial, so $\varphi: X_\Pi \to S_\Pi$ is an isomorphism.

\medskip  C) Let $\Pi$ be   an irreducible subquotient of $\ind_{B}^{G}1$.  The   length of $\Pi|_{G'}$ is   (\S \ref{S:434}):
  
 $1$   when $\Pi=1 , q^{val} \circ \det$ or $\St$, 
 
 $2$ when   $\Pi=\Pi_0$ if $\charf_R\neq 2$ and $q+1=0$ in $R$,
 
  $4$ when   $\Pi=\Pi_0$  if $\charf_R=2$.

\noindent  We have $\sigma_\Pi =((q^{1/2})^{\val } \oplus (q^{-1/2})^{\val} )\circ \alpha_F$   (formula \eqref{eq:LL}, Proposition \ref{ex:1}). 
The centralizer $C'_\Pi$ of the image of $\sigma_\Pi (W_F)$ in $PGL_2(R)$  is the image in $PGL_2(R)$ of 
  $$ \{A \in GL_2(R) \ | \ A \diag( q, 1)A^{-1} \in R^* \diag( q, 1)\}=$$ 
$$\{A =\begin{pmatrix} x &y\\z&t
\end{pmatrix} \in GL_2(R) \ |  \ \begin{pmatrix} xq &y\\zq&t
\end{pmatrix}  = u \begin{pmatrix} xq&yq\\z&t
\end{pmatrix} \ \text{for some } u\in R^*\}.$$

\noindent If $x\neq 0$ or $t\neq 0$  then $u=1$, and if $y\neq0$ then $qu=1$. If $z\neq 0$ then $u=q$. 
 So, $C'_\Pi$ is: 

  $PGL_2(R)$ if  $q-1=0$ in $R$.

 The  diagonal torus when $q-1 \neq 0$ in $R$ and $q+1\neq 0$ in $R$.
  
 The  centralizer of the diagonal torus  if  $q-1\neq 0$ in $R$ and $q+1= 0$ in $R$.
 
\noindent We have $N_\Pi= 0$, hence $C_\Pi= C'_\Pi$, when:

  $\Pi=1$   when $q+1\neq 0$ in $R$, hence  $C_1=PGL_2(R)$ if $q+1\neq 0, q-1=0$ in $R$  (so $\charf_R\neq 2$) 
 and $C_1$ is the diagonal torus  if $q+1\neq 0,q-1\neq 0$ in $R$.
In both cases $S_1=\{1\}$.

  $\Pi=\Pi_0$ cuspidal when $q+1=0$ in $R$.  Recalling \S \ref{S:434},
  when  $\charf_R\neq 2$, $\lg (\Pi_0|_{G'})=2$ and $C_{\Pi_0}$ is the normalizer of the diagonal torus and  $S_\Pi=  \mathbb Z/ 2 \mathbb Z$. We have $X_{\sigma_{\Pi_0}}=\{1, (-1)^{val}\}$ (Corollary \ref{cor:spG}). As in B), $\varphi((-1)^{val})$ is not trivial, so $\varphi: X_\Pi \to S_\Pi$ is an isomorphism.

    But
when  $\charf_R=2$,  then $q-1=0$ in $R$ and $C_{\Pi_0}=PGL_2(R)$. As $S_{\Pi_0}=\{1\}$ and $\lg (\Pi_0|_{G'})=4$,  there is no enhanced correspondence.

 We suppose now   $N_\Pi\neq 0$. Then (Proposition \ref{ex:1}) $\Pi=\St$   when $q+1\neq 0$ in $R$, $\Pi$ is a character when $q+1=0$ in $R$. In both cases $ \Pi|_{G'}$ is irreducible  (Corollary \ref{cor:v}).  We can suppose that $N_\Pi$ is  
  a non-trivial upper triangular matrix. A similar analysis gives that $C_\Pi$ is 

   the diagonal torus  if if  $q-1\neq 0$ in $R$,
   
     the upper triangular subgroup if  $q-1=0$  in $R$. 
  
 \noindent In both cases $S_\Pi=\{1\}$.
  \end{proof}
\begin{remark}  We computed  the centralizer  $C_\Pi \subset PGL_2(R)$:

 $C_\Pi  $ is finite if and only if $\Pi$ is supercuspidal.

When  $C_\Pi$ is connected, it is isomorphic to $PGL_2(R)$, the upper triangular subgroup, the diagonal subgroup,  or  $\{1\}$.
 
When $C_\Pi$ has two connected components it is isomorphic to   the normalizer of the diagonal subgroup or to $ \mathbb Z/ 2 \mathbb Z$.

 When $C_\Pi$ has four connected components, it is isomorphic to the  Klein group $ \mathbb Z/ 2 \mathbb Z\times \mathbb Z/ 2 \mathbb Z$.
\end{remark}

\subsubsection{}\label{S:464} Assume $\charf_R=2$. A kind of lifting has been introduced by   \cite{TV16} and  generalized in \cite{F23}.
They consider a (connected) split reductive $F$-group $\underline H$, equipped with an involution $\iota$ such that the group
of fixed points $ \underline H ^\iota$ is (connected) split reductive. They set up a correspondence, called linkage, 
between $\iota$-invariant irreducible smooth $R$-representations $\Pi$  of $H=\underline H(F)$ and irreducible smooth
$R$-representations of $H^\iota= \underline H ^\iota(F)$. More precisely they show that there is a unique isomorphism $\iota_\Pi$
from $\Pi$  to  its conjugate $\Pi^\iota$ by $\iota$,  which has trivial square.  They say that an irreducible smooth $R$-representation $\pi$ of $H^\iota$ is linked with $\Pi$  if
the Frobenius twist of $\pi$ occurs as a subquotient of the representation $T(\Pi) = \Ker(1+\iota_\Pi) )/ \im(1+\iota_\Pi) ) $ of 
$ H^\iota$.  They  ask for an interpretation of 
linkage in terms of dual groups.

Let us consider the special case  where $\underline H= GL_2$ and $\iota(g)=  g/\det(g)$ \footnote{$\iota(g)$ is conjugate to the transpose of the inverse of $g$}.
 Then  $\underline H^\iota=
 SL_2$, so $H=G, H^\iota=G'$.
Let $\Pi$ be an irreducible smooth $R$-representation of $G$ of central character $\omega_\Pi $. It is invariant under $\iota$ if and only if  $\Pi \simeq \Pi \otimes (\omega_\Pi  \circ \det)$.
 This implies that $\omega_\Pi $ has  trivial square, so is trivial because
$\charf_R=2$. In other words, $\Pi$ is $\iota$-invariant  if and only if $\Pi$  factors to a representation of
$PGL_2(F)$. It follows that then $\iota_\Pi $ is identity,  and $T(\Pi)$ is simply the 
restriction of $\Pi$ to $G'$, which we have throroughly investigated. In particular $T(\Pi)$ has finite length,
as expected.
The dual group of $\underline H$ over $R$ is $GL_2(R)$, that of $\underline H^\iota$  is $PGL_2(R).$
They ask for an interpretation of linkage in terms of a natural homomorphism from $ PGL_2(R)$ to $GL_2(R)$. 

 Let  $\sigma_\Pi :W_F\to GL_2(R)$ be  the semi-simple $L$-parameter of $\Pi$. The map $\varphi^{-1}(\sigma_\Pi) : W_F\to GL_2(R)$ followed by the quotient map $GL_2(R)\to PGL_2(R)$,  is the  semi-simple $L$-parameter  $\rho_\Pi: W_F\to PGL_2(R)$  of the Frobenius twist of any  constituent $\pi$ of $\Pi|_{G'}$. 

The map $ \Psi(g)= \varphi(g)/\det(g)$ for $g\in GL_2(R)$ where  $ \varphi :x\to x^2$ is the   Frobenius map of $R$, is trivial on scalar matrices, hence  factors through a homomorphism $\Psi:PGL_2(R) \to GL_2(R)$. The homomorphism $\Psi$  is injective of image   $SL_2(R)$.  Now if $\Pi$ is $\iota$-invariant, the determinant  of  $\sigma_\Pi  $   is trivial so $\sigma_\Pi=\Psi \circ \rho_\Pi$ and the conjectures of  \cite[\S 6.3]{TV16} are indeed true in our special case.

   \section{Representations of $SL_2(F)$ near the identity}\label{S5}
   
    \subsection{}  Assume  $
 \charf_F=0$ and $R=\mathbb C$. Let  $H$ be  the group of $ F$-points of a connected reductive group over $F$. We denote by $C_c^\infty(X;\mathbb C)$  the space of smooth complex functions with compact support on  a locally profinite space $X$.  The exponential map $\exp$ from $ \Lie(H)$  to $H$ induces an $H$-equivariant bijection between a neighbourhood of $0$ in $\Lie(H)$ and a neighbourhood of $1$ in $H$. So a  function $ f \in C_c^\infty(H;\mathbb C)$  with support small enough around $1$ gives a smooth function $f \circ \exp$ around $0$ in $C_c^\infty(\Lie(H);\mathbb C)$. Also there are only finitely many nilpotent orbits of $H$ in $\Lie(H)$, for the adjoint action. For each such orbit $\mathfrak O$, there is an $H$-invariant measure on $\mathfrak O$, and a  function $\varphi \in C_c^\infty(\Lie(H); \mathbb C) $ can be integrated along $\mathfrak O$ with respect to that measure, yielding an orbital integral $I_{\mathfrak O}(\varphi)$.  Choosing a non-degenerate invariant  bilinear form on $\Lie (H)$, a non-trivial character of $\Lie(H)$ and a Haar measure on
 $\Lie(H)$  yields a Fourier transform $\hat {\varphi}$ for a   function  $\varphi \in C_c^\infty(\Lie(H); \mathbb C) $.  Fix also a Haar measure $dh$ on $H$. 
 
 \begin{theorem} \label{th:1}Let $\Pi$ be a smooth complex representation of $H$ with finite length.
Then there is an open neighbourhood $ V(\Pi) $ of $1$ in $H$ and for each nilpotent orbit $\mathfrak O$ a unique complex number $c_{\mathfrak O}=c_{\mathfrak O}(\Pi)$  such that if $ f \in C_c^\infty(H;\mathbb C)$ has compact support in $V(\Pi) $ then the trace  $\tr _\Pi (f)$ of the linear endomorphism $\int_H f(h)\,  \Pi(h) \, dh$ is equal to \begin{equation}\label{gex}\tr _\Pi (f)=\sum_{\mathfrak O} c_{\mathfrak O}(\Pi) I_{\mathfrak O}(\hat \varphi) \ \ \text{ where } \ \varphi=f \circ  \exp.
\end{equation}
\end{theorem}

This was first proved by Roger Howe when $H=GL_n(F)$, and the general case is due to Harish-Chandra. 
 
 As is usual, we say that a nilpotent orbit $\mathfrak O' $ is smaller than  a nilpotent orbit  $\mathfrak  O $ if $\mathfrak  O'$ is contained in the closure of $\mathfrak O$. With the normalizations   as in  \cite{Va14} we have:
 
 \begin{theorem} \label{th:2}Let $\Pi$ be a smooth complex representation of $H$ with finite length. When $\mathfrak O$ is maximal among the orbits with $c_{\mathfrak O}(\Pi)\neq 0$, then $c_{\mathfrak O}(\Pi)$ is equal to the dimension of generalized Whittaker spaces for $\Pi$  attached to $\mathfrak O$.
\end{theorem}
 
 The result  when $p$ is odd  due to \cite{MW87} is extended to $p=2$  in  \cite{Va14} in general).  When $\mathfrak O$ is a regular nilpotent orbit, the generalized Whittaker model is the usual one, and the result then goes back to Rodier \cite{R74}. Varma actually proves that with that normalization all coefficients $c_{\mathfrak O}(\Pi)$ are rational \cite{Va14}.

\subsection{} \label{S:dg} Assume   $R=\mathbb C$.  For any $F$, when  $H$  is an open normal subgroup of $GL_r(D)$  where $D$ is a finite dimensional central division $F$-algebra, Theorem \ref{th:1}  still holds, with the exponential map replaced by the map $X\mapsto 1+X$  \cite{L04}. In the special case where $H=GL_r(D)$, Theorem \ref{th:2} also holds, at least for the natural generalized Whittaker space attached to each nilpotent orbit \cite{HV23}.

 \subsubsection{}\label{mrho} 
We use the notations and definitions introduced in \S \ref{ss:Wh}. 
 Let $H$ be  an open normal subgroup of $G=GL_2(F)$ containing  $ZG'$. The index of $H$ in $G$ is finite as   $H/ZG'$ is open in the compact group  $G/ZG'$.
 Put \begin{equation} \label{eq:d}V_H = F^*/\det (H), \ \  \dim_{\mathbb F_2}V_H=d, \ \   |G/H|=2^d.
 \end{equation}
 A nilpotent matrix    can be conjugated in a lower triangular nilpotent matrix $Y$  by an element of $G'$.  
 Two such matrices $Y$ and $Y'$ are  $H$-conjugate if and only if their bottom left coefficients differ by multiplication by an element of $\det (H)$.  
  \begin{equation}\label{eq:nild} \text{  The number of $H$-orbits in the nilpotent matrices in  $  M_2(F)$ is  $1+ 2^d$}.
\end{equation} The $0$-matrix forms the smallest nilpotent $H$-orbit (the ``trivial'' one). The non trivial nilpotent $H$-orbits are maximal, and parametrized by  $V_H $ via their bottom left coefficient.

 \medskip    With the same arguments  as those given for $ZG'$  in  \S \ref{ss:Wh}, any irreducible smooth $R$-representation $\pi$ of $H$  appears in the restriction to $H$ of an irreducible smooth representation $\Pi$ of $G$, unique 
  modulo torsion  by a smooth $R$-character of $G$. 
   The  irreducible components $\pi$ of $\Pi|_H$ are $G$-conjugate (even $B$-conjugate) and the $G$-stabilizer  of  $\pi$   does not depend on the choice of $\pi$ in $\Pi|_H$, and denoted by $ G_{\Pi|_H}$. 
      The   representation $\Pi|_H$ is semi-simple of multiplicity $1$ with  length  
      \begin{equation}\label{eq:lgen}\text{$\lg(\Pi|_H)= |G/G_{\Pi|_H}|$
 dividing  $\lg(\Pi|_{ZG'})= |G/G_\Pi|=|L(\Pi)|$, }
 \end{equation}
 hence  equal to $1,2$ or $4$ by 
   Theorem \ref{lpi}. The  representation $\pi|_{G'}$   is  semi-simple of multiplicity $1$ with  length  $\lg (\pi|_{G'})= \lg(\Pi|_{G'})/\lg(\Pi|_H)= |G_{\Pi|_H}/G_\Pi|$.

     For a lower triangular matrix $Y\neq 0$, we have:
 $$\sum_{\pi\subset  \Pi|_H}\dim_R W_Y(\pi)= \dim_R W_Y(\Pi) =1.$$  
 There is a single irreducible $\pi$ in $\Pi|_H$ with $W_Y(\pi)\neq 0$, and  $ \dim_R W_Y(\pi) \neq 0 \Leftrightarrow  \dim_R W_Y(\pi) =1$. If  $W_Y(\pi)\neq 0$ then $W_{Y'}(\pi)\neq 0$  when   $Y'$  and $Y$ are $H$-conjugate.
        We consider  $\dim_R W_Y(\pi)$   as a function $m_\pi$  on 
 $V_H$. 
  Because $\pi$ extends to $G_{\Pi|_H}$,   $m_\pi$ is invariant under translations by  $$W_{\Pi|_H} =\det (G_{\Pi|_H})/\det(H).$$  It follows that $m_\pi $ is the characteristic function of an affine subspace $A_\pi$   of $V_H$ with direction $W _{\Pi|_H}$, each such affine subspace being obtained  exactly for one   $\pi\subset \Pi|_H$.   For $g\in G$ we  denote $\pi^g(x)=\pi(g xg^{-1})$ for $g\in G, x\in H$, so $\pi^{gh}= (\pi^g)^h$  for $g,h\in G$. We  have $A_{\pi^g}=\det(g) A_\pi$.
We have a disjoint union (the Whittaker decomposition):
 \begin{equation}\label{eq:W} V_H  =\sqcup _{\pi\subset \Pi|_H} \,  A_\pi. \end{equation}
    If  $\lg(\Pi|_H)=1$, $m_\pi$ is  the constant function on $V_H$ with value $1$. 
If  $\lg(\Pi|_H)=2$,   the  two irreducible components of $\Pi|_H$ yield the characteristic functions of two affine hyperplanes of $V_H$ with the same direction.
Finally for $\lg(\Pi|_H)=4$, we get the characteristic functions of four affine subspaces  of codimension $2$ in $V_H$ with the same direction. 
In particular when $p$ is odd and $\lg(\Pi|_H)=4$, then  $H=ZG'$ and  $m_\pi$ is a non-zero delta function on $V_H=F^*/(F^*)^2$.

 \medskip Let    $C(V_H;\mathbb Z)$ denote the $\mathbb Z$-module of functions $f:V_H\mapsto \mathbb Z$. For an integer $0\leq r<d $,  let  $I_r$ denote  the $\mathbb Z$-submodule of $C(V_H;\mathbb Z)$ generated by  the characteristic functions of the $r$-dimensional affine subspaces of $V_H$. We have $I_0=C(V_H;\mathbb Z)$.

\begin{lemma}\label{le:d}  When $0< r<d$, $2I_{r-1}$ is included in $I_{r}$ and  the exponent of $I_0/I_r$ is $2^r$. \end{lemma} 
\begin{proof} Let  $W$ be a $r-1$-dimensional vector subspace of $V_H$ and $\{0,e,f,e+f\}$ a supplementary plane. For an affine subspace $A$ of $V_H$ of direction $W$,  the affine subspaces $A_e=A\cup A+e, A_f=A\cup A+f$ and $B= A+e\cup A+f$ of $V_H$ are $r$-dimensional, and taking their charactersitic functions $\chi$, we get $\chi_{A_e}+\chi_{A_f}-\chi_B=2\chi_A$. Thus  $2I_{r-1}\subset I_r$.   
By induction $ 2^r I_0 \subset I_r$.   
 The  map $s_r:C(V_H;\mathbb Z)\mapsto \mathbb Z/ 2^r \mathbb Z$ given by the sum of coordinates is surjective and vanishes on $I_r$ but not on $I_{r-1}$. So the exponent of $I_0/I_r$ is $2^r$.
  \end{proof}

 \subsubsection{}Let us precise Theorem \ref{th:1}  for  an open normal subgroup  $H$ of $G=GL_2(F) $ as in  \S \ref{mrho}.

\begin{notation}\label{not:54} On $G$ (hence on $H$) we put a Haar measure $dg$, and on $\Lie G=\Lie H=M_2(F)$ we put the Haar  measure $dX$ such that $X\mapsto 1+X$ preserves measures near $0$. The invariant bilinear map $(X,X')\mapsto \tr(XX')$ on $ \Lie(H)$  is non-degenerate.  The Fourier transform $\varphi \mapsto \hat {\varphi}$ on  $ C_c^\infty(\Lie(H); \mathbb C) $ is taken with respect to the non-trivial character $\psi \circ \tr$ on $\Lie(H)$.   For each nilpotent $H$-orbit $\mathfrak O$ in $\Lie(H)$, we normalize the nilpotent orbital integral $I_{\mathfrak O}(\hat \varphi) $ \cite[Proposition 1.5]{L05} in the same way as \cite[\S3]{Va14}; that normalization is valid even when $\charf_F>0$.  
By Remark 2 of loc.cit., for large enough $i$,  $K_i=1+M_2(P_F^i)$ and a lower triangular nilpotent matrix $Y$, the measure of $\Ad(K_i)(Y)$ is $0$ if $Y=0$
and $q^{-2i}$ otherwise. In particular  $I_{0}(\hat \varphi)=\varphi(0) $ 
for  the nilpotent trivial orbit $0\in \Lie H$.
 \end{notation}

 \begin{theorem} \label{th:1'} 
 Let $\pi$ be a smooth complex representation of $H$ with finite length.
There is an open neighbourhood $ V(\pi) $ of $1$ in $H$ and for each nilpotent $H$-orbit $\mathfrak O$ a unique complex number $c_{\mathfrak O}=c_{\mathfrak O}(\pi)$  such that if $ f \in C_c^\infty(H;\mathbb C)$ has compact support in $V(\pi) $ then   \begin{equation}\label{ge}\tr _\pi (f)=c_0(\pi) f(1) +\sum_{\mathfrak O\neq 0} c_{\mathfrak O}(\pi) I_{\mathfrak O}(\hat \varphi) \ \ \text{ where $ \varphi(X)=f(1+X)$ for $1+X\in V(\pi)$}.
\end{equation}
\end{theorem}

We call   \eqref{ge} the germ expansion  and  
$c_0(\pi)$  the  constant coefficient of the trace of $\pi$ around $1$. A character twist of $\pi$ does not change $c_0(\pi)$. For $\pi$ irreducible,  $c_{\mathfrak O}(\pi) =0$ for all  $\mathfrak O\neq 0$ if and only if  $\pi $ is degenerate (by Theorem 
\ref{th:2})  if and only if $\dim_{\mathbb C} \pi=1$. In this case   $c_{0}(\pi) =1$.

 We can determine that constant coefficient   $c_0(\pi)$ for any irreducible smooth  representation $\pi$ of $H$ from the case of $G$, because $\pi$ appears in the restriction to $H$ of an irreducible smooth complex representation $\Pi$ of $G$.  The irreducible components  of $\Pi|_H$ being $G$-conjugate to $\pi$ have the same constant coefficient\footnote{by the linear independence of nilpotent orbital integrals}, and  
\begin{equation}\label{Pipi}c_0(\Pi)= \lg (\Pi|_H) \, c_0(\pi).
\end{equation}
 We have \cite{HV23}   :
 
  \noindent $c_0(1_G)=1$.
  
 \noindent When $\Pi$ is parabolically induced, for example when $\Pi$ is  tempered and  not  a discrete series, $$c_{0}(\Pi)=0.$$ 
When $\Pi$ is a discrete series  representation   of  formal degree $d(\Pi) $, $$c_{0}(\Pi) = - d(\Pi)/d(\St).$$
When  $\Pi$ is a supercuspidal  complex smooth representation of $G$ of minimal level $f_\Pi$ (the minimal level\footnote{The  level is  the normalized level of \cite{BH06} \S 12.6 and  the depth in the sense of Moy-Prasad.} of the character twists of $\Pi$), 
\begin{equation}\label{eq:fo} c_0(\Pi)=\begin{cases} -2 q ^{f_\Pi}\ \text{if $f_\Pi$ is an integer}\cr
-(q+1) q ^{f_\Pi -1/2}\ \text{if $f_\Pi$ is a half-integer  (not an integer) }
\end{cases}.
\end{equation}
When  $f_\Pi$ is a half integer  (not an integer), $\Pi$ has positive level (\S \ref{S:331}), $\Pi=\ind_J^G\Lambda$  where
$J=E^*(1+Q^{f_\Pi+1/2})$, where $E/F$ is ramified, $Q$ is the Jacobson radical of an Iwahori order in $M_2(F)$, and $\Lambda$  is trivial on $1+Q^{2f_\Pi+1}$ \cite[\S 15]{BH06}. Let $\chi\in X_\Pi \setminus\{1\}$.  Then $\chi$  is ramified  \cite[20.3 Lemma]{BH06}.
The level $r_\chi$ of $\chi$ is   the  largest positive integer $r$ such that $\chi$  is non-trivial on $1+P_F^r$ when $\chi$ is ramified. We have
\begin{equation} \label{fr}1\leq  r_\chi <f_\Pi.
\end{equation}
 Indeed,    if  $r_\chi >f_\Pi$ then $\chi\circ \det$
is non-trivial on $1+Q^{2r_\chi}$ (as $\det(1+Q^{2r}_\chi)=1+P_F^{r_\chi}$), and   $(\chi\circ \det)\otimes \Lambda$ would  be non trivial on $1+Q^{2r_\chi}$ implying that the level of
$(\chi\circ \det)\otimes \Pi$  is at least $r_\chi$ by (loc.cit.(15.8.1)), contrary to the assumption that $\chi\in X_\Pi$.  So   $f_\Pi<r _\chi$
 as $r_\chi$  is an integer but not $ f_\Pi$. 

\begin{lemma} \label{le:Pipi}  If $f_\Pi=1/2$ then $X_\Pi =\{1\}$. If $q=2$ and $f_\Pi=3/2$ then $X_\Pi$ cannot have $4$ elements.
\end{lemma}
 \begin{proof}   
 If $f_\Pi=1/2$,  then $X_\Pi$ is trivial by the formula \eqref{fr}. 
  If $f_\Pi=3/2$, then  $r_\chi = 1$,
and if $q=2$ there are only $2$ quadratic characters of level $1$. That implies that $X_\Pi$ cannot have 4 elements.
  \end{proof}  

\begin{proposition}\label{prop:constant} Let $\Pi$ be an  irreducible complex smooth representation of $G$ and  $\pi $ an irreducible representation of $H$ contained in $\Pi|_H$.  Then
$$c_0(\pi)=-1/2 \ \text{if $p$ is odd, $\Pi$ is cuspidal of minimal level $ 0$ and $L(\Pi)$ has $4$ elements}.$$

$c_0(\pi)$ is an integer otherwise.

 $c_0(\pi) =0$ if $\pi$ is a principal series, and $c_0(\pi) <0$  if $\pi$ is infinite dimensional and not a principal series. 
 \end{proposition}
 
\begin{proof} 
 By formulas \eqref{eq:lgen}, \eqref{Pipi},  \eqref{eq:fo}, we  have

$c_0(1_G)=1$, so $c_0(1_{H})=1$.

$c_0(\St)=-1$ so $c_0(\st_H)=-1$, since the restriction $\st_H$ of $\St$ to $H$ is irreducible as $\st=\St|_{G'}$ is irreducible.
 
$c_0(\Pi)=0$ so $c_0(\pi)=0$, when $\Pi$ is an irreducible principal series.

 $c_0(\Pi)<0$ so $c_0(\pi)<0$,   when $\Pi$  supercuspidal of level  $f_\Pi$ (the minimal level). 
 If $p $ is odd, then $c_0(\Pi)$  is an even integer by  \eqref{eq:fo},
so that $c_0(\pi)$ is an integer if $L(\Pi)$ has $1$ or $2$ elements by  \eqref{Pipi}; if $L(\Pi)$ has $4$  elements,  then $f_\Pi=0$  by Proposition \ref{prop:3.12} and
$c_0(\Pi)=-2$, so $c_0(\pi)=-1/2$.
If $p=2$, then  $c_0(\Pi)$ is a multiple of $4$ (so $c_0(\pi)$ is an
integer) by   \eqref{eq:fo} except when: 

(i) $f_\Pi=0 $, where $c_0(\Pi)=-2$. But $L(\Pi)$ has size $2 $ by Proposition \ref{prop:ty1}, so $c_0(\pi)=-1$.

 (ii) $f_\Pi=1/2$, where $c_0(\Pi)=-(q+1)$.  But $L(\Pi)$ has size $1$ by Lemma \ref{le:Pipi}, so $c_0(\pi)=-(q+1)$.
 
 (iii) $f_\Pi=3/2$ and $q=2$, where $c_0(\Pi)=-6$.  But  $L(\Pi)$ has size $1$ or $2$ by Lemma \ref{le:Pipi}, so $c_0(\pi)=-6$ or $-3$.
\end{proof}

 \begin{theorem} \label{Ro}   Let $\pi$ be a finite length complex representation of $H$, $Y\neq 0$  a lower triangular matrix in $ M_2(F)$ and $\mathfrak O$ its $H$-orbit.   Then $c_{\mathfrak O}(\pi)=\dim_{\mathbb C}W_Y(\pi)$.
\end{theorem}
 \begin{proof} The proof uses the same idea as \cite{R74}. Remarking that the lower triangular group $ B^-$ of $G$ acts transitively on the lower triangular nilpotent matrices $Y$, and that  for $g\in  B^-$ we have  $c_{\mathfrak O}(\pi)=c_{\mathfrak O ^g}(\pi^g), \dim_{\mathbb C}(W_Y(\pi))=\dim_{\mathbb C}(W_{Y^g}(\pi^g))$,   it is enough to consider the case where $Y=\begin{pmatrix}0&0\cr1&0\end{pmatrix}$.
 We stick to that $Y$ (so $\theta_Y=\theta$ with the notation \ref{not:theta}).
 
 For each positive integer $i$, we define a character $\chi_i$ of  the pro-$p$ group $K_i= 1+ M_2(P_F^i)$, by the formula
\begin{equation}\chi_i(1+X)=\psi \circ \tr (p_F^{-2i}YX)=\psi(p_F^{-2i} X_{1,2}), \ \ X=\begin{pmatrix}X_{1,1}&X_{1,2}\cr X_{2,1}&X_{2,2}\end{pmatrix} \in M_2(P_F^i).\end{equation}
The character $\chi_i$ is trivial on $K_{2i}$. Conjugating by the diagonal matrix $d_i=\diag(p_F^{i} ,p_F^{-i} )$ we get a character $\theta_i$ on $H_i=d_i^{-1} K_i d_i=1 +  \begin{pmatrix} P_F^i & P_F^{-i}  \\  P_F^{3i}&  P_F^i\end{pmatrix} $ such that $\theta_i(1+X)=\psi(X_{1,2})$. The  limit of  
the groups    $H_i $ as $i\to\infty$ is the group $U$.
 We will prove that  the $\theta_i$ approximate  the  character $\theta_Y$ of $U$  
 in the sense that\begin{equation}\label{eq:Ro}
\lim_{i\to \infty} \dim_{\mathbb C} \Hom_{H_i}(\theta_i,\pi) = \dim_{\mathbb C} W_Y(\pi).
\end{equation}
On the other hand we will also prove  in  \S\ref{S:Ro}, following \cite{Va14}, that
\begin{equation} \label{eq:Ro'}\dim_{\mathbb C} \Hom_{K_i}(\chi_i,\pi) = c_{\mathfrak O}(\pi) \ \text{for large} \  i.
\end{equation} 
Since $ \dim_{\mathbb C} \Hom_{H_i}(\theta_i,\pi) =  \dim_{\mathbb C} \Hom_{K_i}(\chi_i,\pi) $, we shall get the result.
 \end{proof}
 
\subsubsection{}\label{S:Ro} Let us proceed to the proof of the formulas \eqref{eq:Ro} and \eqref{eq:Ro'}, through a sequence of lemmas, rather easy compared to the analogous statements in the more general cases treated by \cite{R74}, \cite{MW87} and \cite{Va14} when $\charf_F=0$, in \cite{HV23} for arbitrary $\charf_F$.

For $X\in M_2(F)$, put $\delta_i (X)=\chi_i^{-1}(1+X) $ if $X\in M_2(P_F^i)$ and $\delta_i(X)=0$ outside. With the notation \ref{not:54},  the Fourier transform $\hat \delta_i $ of $\delta_i$ is 

\begin{equation} \label{le:1} \hat \delta_i (X) = 
 \begin{cases} q^{-4i}\vol (M_2(O_F), dX)  \  & \text {if } \  X\in  p_F^{-2i}Y+M_2(P_F^{-i}) \\
0   \  &   \text {otherwise}
\end{cases}.
 \end{equation}
  
  \begin{lemma}\label{le:2} The $K_1$-normalizer of $\chi_i$   is $(ZU^-\cap K_1)K_i$. \end{lemma}
 \begin{proof} For a positive integer $j\leq i$, we  prove that the $K_1$-normalizer  of the restriction of  $\chi_i$ to $K_{2i-j}$ is $(ZU^-\cap K_1)K_j$  by induction on $j$. This is clear for $j=1$ and the case $j=i$ gives what we want. Assume that the claim is true for $j<i$ and let us prove it for $j+1$. Let $g\in K_1$ normalizing the  restriction of  $\chi_i$ to $K_{2i-j-1}$. By induction $g\in (ZU^-\cap K_1)K_j$ and we may assume $g\in K_j$. Write $g=1+X$ with $ X\in M_2(P_F^j)$. Then $g^{-1}Yg\equiv Y+YX-XY$ modulo $ M_2(P_F^{j+1})$
  and the hypothesis on $g$ means that $YX-XY \equiv 0$ modulo $ M_2(P_F^{j+1})$, which gives that $p_F^{-j}X$ commutes with $Y$ modulo $P_F$. But the commutant of $Y$ modulo $P_F$ in $M_2(k_F)$ is made out of   lower triangular matrices with the same diagonal elements.  Consequently $g\in (ZU^-\cap K_1)K_{j+1}$ as claimed.
  \end{proof} 
  
  \begin{lemma}\label{le:3} The $K_i$-orbit of $Y$ is the set of nilpotent matrices in $Y+M_2(P_F^i)$.
 \end{lemma} 
 \begin{proof} It is clear that $gYg^{-1}$ is a nilpotent element in $Y+M_2(P_F^i)$ for $g\in K_i$. Conversely let $Y+p_F^i Z$ nilpotent (hence of trace $0$) with $Z\in M_2(O_F)$. If $g=1+p_F^i X$ with $X\in M_2(O_F)$, then $g(Y+p_F^i Z)g^{-1}\equiv Y+p_F^i(YX-XY+Z)$ modulo $ M_2(P_F^{i+1})$. We choose $X$, as we can, so that 
 $ YX-XY+Z\equiv 0$ modulo $P_F$. So $g(Y+p_F^i Z)g^{-1}\in  Y +M_2(P_F^{i+1})$. The $K_i$-orbit of $Y$ is closed in $M_2(F)$. 
 We finish the proof by successive approximations.
  \end{proof}

Let  $\pi$ be a smooth representation of $H$ on a complex vector space $V$,  and $\phi:V\to V_\theta$ be the quotient map from $V$ to the $\theta$-coinvariants $V_\theta$ of $V$.  For   large enough $i$ such that $H_i \subset H$ let $V_i$
be the $\theta_i$-isotypic component of $V$. 

 \begin{lemma}\label{le:4}  For large enough $i$, $\phi(V_i)=V_\theta$.
 \end{lemma} \begin{proof} It is the same as that of Lemma 8.7 in \cite{HV23}.
    \end{proof}
 We have  $$H_{i+1}= (H_{i+1}\cap H_i)(H_{i+1}\cap U), \ [H_{i+1}:(H_{i+1}\cap H_i)]=[(H_{i+1}\cap U): (H_{i}\cap U)]=q^{-1},$$ 
and $\theta_{i+1}=\theta_i$ on  $H_{i+1}\cap H_i$.   Let $e_i= f_i  dg$ where $dg$ is the Haar measure on $H$ giving the volume $1$ to $H_i$ and $f_i$ is  the function on $G$ with support $H_i$ and value $\theta_i^{-1}$ on $H_i$.

     \begin{lemma} \label{le:5} We have $e_i e_{i+1} e_i= q^{-1} e_i$ when $i>1$ and $H_i\subset H$. In particular,   
the map $v\to \pi(e_{i+1})v:V_i\to V_{i+1}$ is injective.
 \end{lemma} \begin{proof}  
 The lemma  is equivalent to  $\pi (e_i e_{i+1} e_i )v= q^{-1}\pi (e_i)v$ for all  $v\in V$ and $(\pi,V)$  as above. The projector $V\to V_i$ is $\pi(e_i)$ and 
$$\pi(e_i e_{i+1} e_i)v= q^{-1}\sum_{u\in (H_{i+1}\cap U)/(H_{i}\cap U)}\pi(e_i \theta_{i+1}(u)^{-1} u e_i)v$$
If  $\pi(e_i  u e_i)v\neq 0$ for  $u\in H_{i+1}\cap U$, then $u$ intertwines $\theta_i$.
To interpret that condition  we conjugate $\theta_i$ back to $\chi_i$.  Then $H_i$ is sent to $K_i$ and $H_{i+1}$ is sent  to  $d_1^{-1}K_{i+1}d_1$ which, we remark, is contained in $K_{i-1}$.  By Lemma  \ref{le:2}, $u\in H_{i+1}\cap U$ conjugates to an element in $(Z U^- \cap K_1)K_i$, so that $u\in H_i\cap U$. We deduce that $\pi (e_i e_{i+1} e_i )v= q^{-1}\pi (e_i)v$ as claimed. 
  \end{proof}

 {\sl Proof  of the formula  \eqref{eq:Ro}}

  Fix a large integer $i$ such that the lemmas apply.
The  projector $\pi(e_i):V\to V_i$ can be obtained by first projecting onto $V^{H_i \cap  B^-}$, and then applying the projector $\pi (e_{i,U}) $ where $e_{i,U}=f_i|_{H_i\cap U} du$ for  
 the Haar measure on $H\cap U$ giving the volume $1$ to $H_i\cap U$. As $V_i\subset V^{H_{i +1}\cap   B^-}$, we have $\pi(e_{i+1}) = \pi(e_{i+1, U})$ on  $ V_i$. It follows that for $v\in V_i$ and  $v_1=\pi(e_{i+1}) v=\pi(e_{i+1,U}) v$  have  the same image $\phi(v_1)=\phi(v)$   in $V_\theta$. Iterating the process we get for positive integers  $k$,  vectors $v_k = \pi(e_{j+k})v_{k-1}= \pi(e_{j+k,U})v_{k-1}$ with  $\phi(v_k)=\phi(v)$. As $e_{i+1,U}e_{i,U}= e_{i+1,U}$ we have $v_k = \pi(e_{i+k,U})v$.
But $\phi(v)=0$ is equivalent to $ \pi(e_{i+k,U})v=0$ for large $k$. As $v_k=0$ implies $v_{k-1}=0$ by Lemma \ref{le:5}, we get that $\phi$ is injective on $V_i$. Since it is also surjective by Lemma \ref{le:4}, we deduce that  it gives an isomorphism $V_i\simeq V_\theta$.   This ends the proof of
 \eqref{eq:Ro}.

\medskip {\sl Proof  of the formula  \eqref{eq:Ro'}} 

Fix  an integer $i$ such that $K_i\subset H$. We have $\dim_{\mathbb C}(\Hom_{K_i} \chi_i, \pi) =\tr \pi(e'_i)$ where  $e'_i= f'_i  dg$ where $dg$ is the Haar measure on $H$ giving the volume $1$ to $K_i$ and $f'_i$ is  the function on $G$ with support $K_i$ and value $\chi_i^{-1}$ on $K_i$.  We have, $f'_i(1+X) =\delta_i(X)$. To prove \eqref{eq:Ro'}, it suffices to  apply the germ expansion \eqref{ge} to $\tr_\pi$ and to show that for large $i$, $I_{\mathfrak O}(\hat \delta_i)=1$ whereas  $I_{\mathfrak O'}(\hat \delta_i)=0$ for any nilpotent orbit $\mathfrak O'\neq \mathfrak O$.  
 From  the formula  \eqref{le:1}, $\hat \delta_i$ is  a multiple of  the characteristic function of $-p_F^{-2i}Y+M_2(P_F^{-i})$ and from   Lemma \ref{le:3} the nilpotent elements there form the $K_i$-orbit of $p_F^{-2i}Y$. It follows that $I_{\mathfrak O'}(\hat \delta_i)=0$ if $\mathfrak O'\neq \mathfrak O$.  That  $I_{\mathfrak O}(\hat \delta_i)=1$ is proved exactly  as in the proof of Lemma 7 in \cite{Va14}.
 
\subsubsection{}

For a locally profinite space  $X$, $x\in X$,  and a field $C$, two linear  forms $f,f'$ on  $C_c^\infty(V;C)$ for some  open neighbourhood  $ V$ of $x$ in $X$, are called equivalent     if  their  restrictions   to $C_c^\infty( W;C)$ for some  open neighbourhood  $ W$ of $x$ contained in $V$ are equal.  The equivalence class of $f$    is called its germ    $\tilde f$ at $x$. Denote  $\mathfrak G_x(X)$ the space of the germs  at $x$.

For a locally profinite space $X'$, an open   subset  $W$  in $X$ and an  open  subset $ W'$ in $X'$, a homeomorphism $j: W\to W'$ gives by functoriality an isomorphism $C_c^\infty( W';C)\to C_c^\infty( W;C)$ and an isomorphism $\mathfrak G_{j(x)}(X') \to \mathfrak G_x(X)$ from the space of  the germs of $X'$ at $j(x)$ to the space of the germs of $X$
 at $x\in W$.

\medskip 
 The nilpotent orbital integrals  $\mathcal F_{\mathfrak O}:\varphi\mapsto I_{\mathfrak O}(\hat \varphi)$  for $\varphi \in C_c^\infty(\Lie H;\mathbb C)$ and the nilpotent $H$-orbits $\mathfrak O$ in $\Lie(H)$,  are linearly independent $H$-equivariant  linear forms on  $C_c^\infty(\Lie H;\mathbb C)$ \cite[page 79]{L05}. They  form a basis of a $\mathbb Z$-module $I_H$ with rank $1+ 2^d$  \eqref{eq:nild}.  For each $H$-equivariant open neighborhood $V$ of $0$ in $\Lie H$,  the $\mathcal F_{\mathfrak O}$ remain independent as linear forms on $C_c^\infty( V;\mathbb C)$.  
The germs $\tilde{ \mathcal F}_{\mathfrak O}$  form a basis of the $\mathbb Z$-module $\tilde I_H$  of germs of  elements of $I_H$. Denote by  $ I_H^{Wh}$  the $\mathbb Z$-submodule of $I_H$ of basis  $\mathcal F_{\mathfrak O}$ for $\mathfrak O\neq 0$.

 Theorems \ref{th:1'} and \ref{Ro} say   that  the germ at $1$ of  the trace of  an irreducible   complex  smooth representation $\pi$  of $H$ identifies via the map $X\to 1+X$ with the germ at $0$ of a  unique element 
  $T_\pi= c_0(\pi) \mathcal F_{0}+ T_\pi^{Wh}$ where $c_0(\pi) \in \mathbb Q$, and $ T_\pi ^{Wh}
 \in  I_H^{Wh}$ is  determined by the non-degenerate Whittaker models  of $\pi$.  
 Note that $T_\pi^{Wh}=0$ if and only if  $\dim_{\mathbb C}\pi=1$.

 Denote by  $T_H^{Wh}$   the $\mathbb Z$-submodule  of $I_H^{Wh}$ generated by  respectively the  $T_\pi^{Wh}$, for  all  irreducible complex  smooth representations $\pi$  of $H$. 
Write $ \tilde I_H^{Wh},    \tilde T_H^{Wh}$  for the space of   germs at $0$ of $ I_H^{Wh}, T_H^{Wh}$.

\begin{theorem}\label{th:3}  We have $\tilde T_H=\tilde I_H$ when  $d=0,1$.
  
The   $\mathbb Z$-submodule $\tilde T_H^{Wh}$ is a submodule of   $\tilde I_H^{Wh}$
 of finite index.  The exponent of $\tilde I_H^{Wh}/\tilde T_H^{Wh}$ is $2^{d-2}$  when $d\geq 2$.
 \end{theorem}
 \begin{proof} 
    
 When $d=0$, then   $I_H$ has $\mathbb Z$-rank $2$, and  the germs  of the traces of the trivial representation $1$ and of the Steinberg representation $\st_H$ form a   $\mathbb Z$-basis  $\{\tilde\tr_1,  \tilde \tr_{\st_H}\}$ of  $\tilde I_H$.

When  $d=1$, then  $I_H$ has $\mathbb Z$-rank $3$, $\det H=N_{E/F}(E^*)$ for a quadratic separable extension $E/F$, the principal series 
$(i_B^G \eta_E)|_H$  is semi-simple of length $2$ and multiplicity free (Lemma \ref{le:2.4}  and footnote in the proof of Proposition   \ref{v}), and 
  the germs  of the  traces of the trivial representation $1$ and of the   two components $\pi_E^+, \pi_E^-$ of  
$(i_B^G \eta_E)|_H$ form  a $\mathbb Z$-basis 
$\{\tilde\tr_1,  \tilde \tr_{\pi_E^+},  \tilde \tr_{\pi_E^-}\}$  of  $\tilde I_H$.

 When $d\geq 2$, the theorem   follows from Lemma \ref{le:d}.
   \end{proof}
   
 Theorem \ref{th:3} can be equally well expressed in terms of the Grothendieck group  $ \Gr_{ R}(H)$.  This is  under this form that the theorem extends to $R$-representations.
    For  an open compact subgroup $K$  of $H$, and $\pi$ a finite length smooth  complex representation 
$\pi$ of $H$, $\pi|_K$ is semi-simple wifh finite multiplicities, and is determined by the restriction of the trace of $\pi$ to $C^\infty_c(K,\mathbb C)$. 
         
\begin{corollary} There  are  $2^d$ virtual finite length smooth complex representations   $\pi_1 , \ldots , \pi_{2^d}$ of $H$ with the following property:  
for any finite length smooth complex representation $\pi$ of $H$, there are unique integers $a_0(\pi),a_1(\pi), \ldots, a_{2^d}(\pi)$,   such that  on some   compact open subgroup $K=K(\pi) $ of $H$, $$\pi \simeq a_0(\pi)  1 + \sum _{i=1}^{2^d} a_i(\pi)  \pi_i.$$
  \end{corollary}
 
\begin{proof} By Theorem \ref{th:3}, the $\mathbb Z$-module $\tilde T_H^{Wh}$ has
 a basis  $\{\tilde T_{\pi_1}^{Wh} , \ldots , \tilde T_{\pi_{2^d}}^{Wh}\}$ where $\pi_1 , \ldots , \pi_{2^d}$  are
 virtual finite length smooth representations   of $H$. By Theorem \ref{th:1'}, for any finite length smooth representation $\pi$ of $H$   there  exist a unique rational number $a_0(\pi) $ and unique integers $a_1(\pi), \ldots, a_{2^d}(\pi)$,  such that     $$\tr_\pi= a_0(\pi) \tr_1 + \sum _{i=1}^{2^d} a_i(\pi) \tr_{\pi_i}$$ 
 on restriction to $C^\infty_c(K(\pi) ,\mathbb C)$ for   some  compact open subgroup $K(\pi) $ of $H$. As $a_0(\pi)=\dim _{\mathbb C}\pi^{K(\pi)}-   \sum _{i=1}^{2^d} a_i(\pi) \dim_{\mathbb C}\pi_i^{K(\pi)},$  we see that  $a_0(\pi)$  is an integer.  Equivalently,  on restriction to $K(\pi)$, $$\pi \simeq  a_0(\pi)  1 + \sum _{i=1}^{2^d} a_i(\pi)  \pi_i.$$
        \end{proof}


\subsubsection{}
This has consequences for  the representations of $G'$. 

An irreducible complex representation of $G'$ extends to $ZG'$, and  we can apply 
 Theorem \ref{th:1'} to $H=ZG'$ when  $\charf_F\neq 2$.
  When   $p$ is odd,  there is an unique $L$-packet $\tau_1,\tau_2,\tau_3,\tau_4$  of $G'$ with $4$ elements (Proposition \ref{cor:bq1}).  One can enumerate the $4$ non-trivial nilpotent $G'$-orbits $\mathfrak O_1,\ldots, \mathfrak O_4$ such that 
$c_{\mathfrak O_i}(\tau_j)=\begin{cases} 1 \ \text{if} \ i=j\cr 0 \ \text{if} \ i\neq j\end{cases}$. For $i=1, \ldots, 4$ we choose a lower triangular element $Y_i\in \mathfrak O_i$.
 \begin{theorem}\label{th:podd} ($p$ odd, $R=\mathbb C$) Let $\pi$ be a finite length  smooth complex representation of $G'$. On restriction to a small enough compact open subgroup $K(\pi)$ of $G'$, we have 
\begin{equation}\label{eq:a0}
\pi \simeq a_0(\pi)1+\sum_{i=1}^4 c_{\mathfrak O_i}(\pi)\tau_i,  \ \  c_{\mathfrak O_i}(\pi)=\dim _{\mathbb C}W_{Y_i}(\pi),
\end{equation}
where  $a_0(\pi) = \dim _{\mathbb C}\pi^{K(\pi)} - \sum_{i=1}^4 c_{\mathfrak O_i}(\pi)\dim _{\mathbb C}\tau_i^{K(\pi)}$.  The constant term in Theorem \eqref{th:1'} is  $$c_0(\pi)=a_0(\pi)- (\sum_{i=1}^4 c_{\mathfrak O_i}(\pi))/2.$$ 
\end{theorem} 
The constant term $c_0(\pi)$ can be computed using \eqref{Pipi} and \eqref{eq:fo}.
   
 \begin{remark} When $\charf_F=0, p$ odd and $ R=\mathbb C$, the theorem was  known (\cite{As94} and the last section of \cite{N23}). 
  \end{remark}

\subsubsection{}  For any $p$, let $\pi$ be an  irreducible smooth complex representation  of $G'$ and $r$ the cardinality of  the $L$-packet  of $\pi$.
 
For any $L$-packet $\{\tau_1, \tau_2, \tau_3, \tau_4\}$ of size   $4$, there exist  integers $a_0,b_0$ such that on a small enough compact open subgroup  of $G'$ we have 
 \begin{equation}\label{eq:1L}
 \ind_{B'}^{G'} 1 \simeq b_0T_1+ \sum_{i=1}^4 \tau_i ,\ \ \ \text{and if  $r=1$, \ }  \ \pi\simeq a_0T_1+ \sum_{i=1}^4 \tau_i .\end{equation}  
 
 If  $r=2$, then $\det (G_\pi)=N_{E/F}(E^*/F)$ for  a quadratic separable extension $E/F$. Choose  a bi-quadratic separable extension  of $F$ containing $E$.  There exist   $\tau_1$ and $\tau_2$ in the associated $L$-packet of size $4$  (Proposition \ref{cor:bq1}) and an integer $a_0$ such that  on a small enough compact open subgroup $K $ of $G'$ we have 
\begin{equation}\label{eq:2L}\
\pi\simeq a_0T_1+ \sum_{i=1}^2 \tau_i.
\end{equation}  

Therefore,  when $R=\mathbb C$ we have:

\begin{theorem} \label{th:anyp1} Let $\pi$ be an irreducible smooth  $R$-representation of $G'$.
There   are  an   integer $a_0 $ and  irreducible smooth  $R$-representations  $ \{\tau_1, \tau_2, \tau_3, \tau_4 \}$  of $G'$ forming an $L$-packet,  such that on a small enough compact open subgroup $K $ of $G'$ we have 
$$\pi \simeq a_0 1 + \sum_{ i=1} ^{4/r} \tau_i,$$
where $r$ is the  cardinality of the $L$-packet containing $\pi$.
\end{theorem}

\subsubsection{} \label{sss} Let us prove  Theorem \ref{th:anyp1}   for any $R$.
 
Let $R_c$ be the algebraic closure in $R$ of the prime field of $R$. Write    $R_c=\mathbb Q^{ac}$ when $\charf_R=0$ and $R_c=\mathbb F_\ell^{ac}$ when  $\charf_R=\ell >0$.

a) We show first that  Theorem \ref{th:anyp1} for  $R_c$ extends to  $R$. 
 A cuspidal   $R$-representation of $G'$ is  the scalar extension  $ \pi_R=R\otimes_{R_c}\pi$ to $R$ of  a  cuspidal  
  $R_c$-representation $\pi$ of $G'$ \cite{V96} and the $L$-packets of size $4$ are cuspidal. The scalar extension from $R_c$ to $R$ respects irreducibility, 
 identifies  the $L$-packets of  size $4$ over $R_c$  with those over $R$ and sends   the $L$-packets of  size $r$ over $R_c$ to $L$-packets of  size $r$ over $R$.
  Theorem   \ref{th:anyp1} for  $R_c$-representations imply Theorem   \ref{th:anyp1} extends  for $R$-representations  which are  scalar extensions of  $R_c$-representations: 
  $$ \pi\simeq  a_0 1 + \sum_{ i=1} ^{4/r}  \tau_i \ \  \text{implies by scalar extension} \ \   \pi_R\simeq  a_0 1 + \sum_{ i=1} ^{4/r} \tau_{i,R}.$$
  The only  irreducible  smooth $R$-representations ot $G'$ which are not scalar extensions of  $R_c$-representations,  
are  principal series $i_{B'}^{G'} (\eta) $.  But   
\begin{equation}\label{eq:ps}i_{B'}^{G'} (\eta) \simeq \ind_{B'}^{G'} (1) \ \ \text{on  some small open compact subgroup $K$ of $G'$},
\end{equation}
and  we have \eqref{eq:1L} for the $R_c$-representation $ \ind_{B'}^{G'} (1) $. 

Therefore, for any $L$-packet $\{ \tau_{1,R}, \tau_{2,R},   \tau_{3,R}, \tau_{4,R} \}$ of size $4$, there is an integer $a_0$ such that 
  $$ \ind_{B'}^{G'} (1) \simeq  a_0 1 + \sum_{ i=1} ^4 \tau_{i,R} \ \ \ \text{on  some small open compact subgroup $K$ of $G'$}.$$

b)  
 Theorem   \ref{th:anyp1}  for   $\mathbb C$ extends to   $\mathbb Q^{ac}$  because  the scalar extension from $\mathbb Q^{ac}$ to $\mathbb C$  respects irreducibility, representations in an $L$-packet  of size $4$  are cuspidal, and cuspidal  representations complex representations of $G'$ are defined over $\mathbb Q^{ac}$.
 
c)   
 Via an  isomorphism $\mathbb C\simeq \mathbb Q_\ell^{ac}$, Theorem \ref{th:anyp1} for $\mathbb C$ extends to  $\mathbb Q_\ell^{ac}$. 
  Theorem \ref{th:anyp1} for $\mathbb Q_\ell^{ac}$ extends to  $\mathbb F_\ell^{ac}$-representations. Indeed, from Proposition \ref{prop:nc1} 
an irreducible smooth  $\mathbb F_\ell^{ac}$-representation $\pi$ of $G'$ in an $L$-packet of size $r$, 
  lifts to an integral  irreducible smooth  $\mathbb Q_\ell^{ac}$-representation $\tilde \pi$ of $G'$ in an $L$-packet of  size $r$ (Proposition \ref{prop:nc}). From   Theorem \ref{th:anyp1}  for  $\mathbb Q_\ell^{ac}$,
there  is an $L$-packet  $ \{\tilde \tau_1, \tilde \tau_2, \tilde \tau_3, \tilde  \tau_4 \}$ of irreducible smooth $\mathbb Q_\ell^{ac}$-representations of $G'$ and an  integer $a_0, $  such that on a small enough compact open subgroup $K $ of $G'$ we have 
 $$\tilde \pi\simeq  a_0 1 + \sum_{ i=1} ^{4/r} \tilde \tau_i.  \ \  \text{By reduction modulo $\ell$} \ \  \pi\simeq  a_0 1 + \sum_{ i=1} ^{4/r}\tau_i,$$
 where  the reduction $\{\tau_1, \tau_2,   \tau_3, \tau_4 \}$  modulo $\ell$ of $ \{\tilde \tau_1, \tilde \tau_2, \tilde \tau_3, \tilde  \tau_4 \}$  forms an  $L$-packet of irreducible smooth $\mathbb F_\ell^{ac}$-representations of $G'$.
 This ends the proof of    Theorem \ref{th:anyp1}.

\begin{remark}\label{re6.18} The formula  \eqref{Pipi}, \eqref{eq:1L} and \eqref{eq:2L}
remain valid for $R$.
\end{remark}

\subsubsection{}\label{S:Nevins}  For an irreducible infinite dimensional complex representation $\Pi$ of $G$ with conductor $c$, Casselman had already described
the restriction of $\Pi$ to $K_0$ as the direct sum of the fixed points under $K_{c-1}$ and a complement depending only
on the central character of $\Pi$.  

Similarly, when $p$ is odd, and $\pi$ is an irreducible infinite dimensional complex
representation of $G'$, Nevins  \cite{N05},  \cite{N13} described explicitly the restriction of $\pi$ to $K_0'$  as a finite-dimensional part specific to
$\pi$, and a complement depending only on the central character of $\pi$. More recently,  Nevins  \cite{N23}  defined for any vertex $x$ of the Bruhat-Tits building of $G'$, admissible complex representations $\tau_{x,1},\ldots \tau_{x,5}$ of the maximal open compact subgroup $G'_x$ fixing $x$ such that  the following is true. Let  $\delta_\pi$ be the depth of $\pi$ in the sense of Moy-Prasad. Then, 
there are integers $a_{\pi,1},\ldots,  a_{\pi,5}$ such that on restriction to $G'_{x, \delta_{\pi +}}$,
$$\pi \simeq \sum_{i=1}^5 a_{\pi,i} \tau_{x,i}.$$

Now  allow any $R$ with $\charf_R\neq p$ (still assuming  $p$ odd). The representations $\tau_{x,i}$ of Nevins  transfered to $\mathbb Q_\ell^{ac}$ are integral, defined over $\mathbb Q^{ac}$ and can be transfered to $R$-representations  $\tau_{x,i,R}$.
The proof in \S \ref{sss} applies and shows that the above result is also valid  over $R$ with 
$\tau_{x,1,R},\ldots \tau_{x,5,R}$.

   \section{  Asymptotics  of invariant vectors by Moy-Prasad subgroups } 
 
 \subsection{}\label{7.1}  Notations as in \S \ref{S:2.3} \ref{SL2}. The Moy-Prasad subgroups  of $G' =SL_2(F)$ are the   intersections of the Moy-Prasad subgroups  of $G=GL_2(F)$  with $G'$ because  the Bruhat-Tits of  $G' $ and of $PGL_2(F)$  are the same. We write $K'=G'\cap K$ for a subgroup $K$ of $G$.

Let $\red: K_0=GL_2(O_F)\to GL_2(k_F) $ and  $\red': K'_0=SL_2(O_F)\to SL_2(k_F) $ 
denote  the usual quotient maps.
The parahoric subgroups of $G$ are  the $G$-conjugates of the maximal open compact subgroup $K_0 $ or of its  Iwahori subgroup  $I_0=\red^{-1}( B(k_F))$. Those of $G'$ are the $G'$-conjugates of the maximal open compact subgroup $K'_0$ or its  Iwahori subgroup $I'_0=\red'^{-1}( B'(( k_F))$, or of the maximal open subgroup $dK'_0d^{-1}= (d K_0 d^{-1})'$ where $d=\begin{pmatrix} 1 & 0\cr 0 & p_F \end{pmatrix}$ \cite[\S 3]{Abde}.

  The Moy-Prasad subgroups  of  $G$  are the $G$-conjugates of the  $j$-th congruence subgroups $K_j , I_j, I_{1/2+j} $  of   $K_0 ,I_0 $,  the pro-$p$ Iwahori subgroup $ I_{1/2}=\red^{-1}(U(k_F))$ of $I_0$, for any integer $j\geq 0$  \cite[\S 12]{HV23}.  The Moy-Prasad subgroups of $G'$ are the $G'$-conjugates of the $j$-th congruence subgroups
$K'_j , d K'_j d^{-1}, I'_j, I'_{1/2+j} $  for $j\geq 0$.   

Let $ \mathfrak j $ denote  the $O_F$-lattice of  matrices $(x_{i,j}) \in  M_2(O_F)$ with $x_{1,2}\in P_F$, and $ \mathfrak j_{1/2}$ the $O_F$-lattice of matrices $(x_{i,j})\in  \mathfrak  j$ with $x_{1,1}, x_{2,2}\in P_F$.
 We have: \begin{equation}K_0=M_2(O_F)^*, \ I_0=  \mathfrak j ^*, \   I_{1/2+j}= 1+p_F^j \, \mathfrak j_{1/2},   \  K_{1+j}= 1+p_F^{j}M_2(P_F ), \  I_{1+j}= 1+P_F^j \ \mathfrak j\end{equation} 
  for $j\geq 0$.
  Note that $I_0=K_0\cap d K_0 d^{-1}$ and the decreasing sequence   for  $H_j= K_j, d K_j d^{-1}$,
$$H_0\supset I_0 \supset I_{1/2} \supset  \ldots  \supset  H_j  \supset I_j \supset  I_{1/2+j} \supset H_{1+j} \supset I_{1+j} \supset \ldots$$
 The $G$-normalizer $ZK_0$ of the maximal compact subgroup $K_0$  normalizes all subgroups $K_j$ for $j\geq 0$.
The $G$-normalizer of the Iwahori group $I$ is generated by $I$ and $\begin{pmatrix} 0 & 1\cr p_F & 0 \end{pmatrix}$; it normalizes all subgroups $I_{1/2+j}, I_j$ for $j\geq 0$.
 \bigskip  Let $s=\begin{pmatrix} 0 & 1\cr 1 & 0 \end{pmatrix}$ and $\beta' =\begin{pmatrix} 0 & -p_F^{-1} \cr p_F & 0 \end{pmatrix}$. 

The Iwasawa decomposition of $G$ with respect to $(B,K_0)$ and the decomposition  of $G$ in double cosets  modulo  $(B,I_0)$ or  $(B,I_{1/2})$  are  \cite[\S 12]{HV23}:
 \begin{equation}\label{dcG}
 G=BK_0= BI_0 \sqcup Bs I_0 =  BI_{1/2} \sqcup Bs I _{1/2}.
 \end{equation}
 Note  that    $Bs I _{1/2}=B\beta' I _{1/2}$.
  The Iwasawa decomposition of $G'$ with respect to $(B',K'_0)$ or $(B',dK'_0d^{-1})$ and the decomposition of $G'$ in  double classes  modulo  $(B',I'_0)$ or  $(B',I'_{1/2})$  are \cite[Lemme 3.2.2, Lemme 3.2.8]{Abde}:
\begin{equation}\label{dcG'}G'=B'K'_0= B'dK'_0d^{-1}=B'I_0'\sqcup B'\beta' I_0'=  B'I'_{1/2} \sqcup B'\beta' I' _{1/2}.
\end{equation}

\subsection{}\label{7.2}   \begin{proposition} \label{pro:BK0}  The map  $B'\backslash G'/H_j'\to   B\backslash G/H_j$  induced by the inclusion  $G'\subset G$  is  bijective, 
 for any  $j$-th congruence subgroup  $H_j=K_j , d K_j d^{-1}, I_j, I_{1/2+j} $ and $j\geq 0$.   
  \end{proposition}
  \begin{proof} The  map $B'\backslash G'/H_j'\to   B\backslash G/H_j$ is surjective as $G=BG'$. 
When $j=0$, the map is  bijective  because  the two sets have the same cardinality   \eqref{dcG}, 
 \eqref{dcG'}.   
  
    Take $j>0$ and $g',g''$ in $G'$ such that $bg'h=g''$ with $b\in B, h\in H_j$. One wants to prove that 
   $b'g'h'=g''$ with $b'\in B', h'\in H'_j$. Multiplying  $g'$ on the left by an element of $B'$  one reduces to $g' \in H'_0$  if  $H_0=K_0, d H'_0 d^{-1}$,  and $g\in H'_0 \cup \beta' H'_0$ if  $H_0=I_0, I_{1/2}$  \eqref{dcG'}.  We have $\det(b) \det (h)=1$.   There exists  $c\in B\cap H_j$ such that $\det (c)=\det(h)$ by the Iwahori decomposition of the $j$-th congruence subgroup   $H_j=(B\cap H_j) (H_j\cap U^-)$  when  $j>0$. Three cases occur:
 
   1) $g'\in H'_0$.   Write  $(bc) g' (g'^{-1} c^{-1} g')h=g'' $ with 
   $b'=bc\in B'$,   $g'^{-1}c^{-1} g'\in H_j$  and 
    $h'= (g'^{-1}c^{-1} g' )h\in H'_j$.  
   
  2) $g'\in \beta' H'_0$ and $g''\in H'_0$.  Apply the same argument  to $g''$. 
 
3)  $g'$ and $g''$ are  in $\beta' H'_0$. Changing notations one wants to  prove that for $g'$ and $g''$   in $H'_0$  such that $b\beta' g'h=\beta g''$ with $b\in B, h\in H_j$, we have 
   $b'\beta g'h'=\beta g''$ with $b'\in B', h'\in H'_j$. Multiply on the left by $\beta^{-1}$. Noting that  $\beta^{-1}B \beta=B^- $, one is   reduced to prove that for  $g',g''\in H'_0$  such that  $ b g'h= g''$ with $ b\in B^-, h\in H_j$, we have  
   $b' g'h'= g''$ with $b' \in (B^-)', h'\in H'_j$. The argument used before with $B$ works also for $B^-$, because we have 
   the Iwahori decomposition  $H_j=(B^-\cap H_j) (H_j\cap U)$ when $j>0$. There exists     $c \in B^- \cap H_j$   
   such that $\det(c)=\det(h)$. Proceeding as in 1), one  writes $(bc) g' (g'^{-1} c^{-1} g')h=g'' $ with 
   $b'=bc\in (B^-)'$,   $g'^{-1}c^{-1} g'\in H_j$  and 
    $h'= (g'^{-1}c^{-1} g' )h\in H'_j$. 
        \end{proof}

 \bigskip Proposition \ref{pro:BK0} has important applications.  The cardinality of $B\backslash G/H_j$ is  computed in \cite[Proposition 11.2]{HV23} for $j\geq 0$. By Proposition \ref{pro:BK0}, $|B\backslash G/H_j|=|B'\backslash G'/H_j|$.    
   \begin{corollary} \label{cor8.4}
  The cardinality of  $B'\backslash G'/H'_j$ for $H'_j=K'_j , d K'_j d^{-1}, I'_j, I'_{1/2+j} $ and $j\geq 0$, is:  

   $|B'\backslash G'/K'_0|=|B'\backslash G'/dK'_0d^{-1}|= |B\backslash G/K_0|= 1$,  
   
\medskip
 $ |B'\backslash G'/K'_{1+j}|= |B'\backslash G'/dK'_{1+j}d^{-1}|= |B\backslash G/K_{1+j}|= (q+1)\, q^j,$

\medskip 
 $ |B'\backslash G'/I_{j}'|=|B'\backslash G'/I'_{1/2+j}|= |B\backslash G/I_{j}|=|B\backslash G/I_{1/2+j}|= 2 \, q^j.$
  \end{corollary}
 
Over any coefficient ring, the restriction to $G'$ of $\ind_B^G 1$ is $\ind_{B'}^{G'}1 $.
   The  vector spaces $(\ind_{B'}^{G'}1)^{H'_j}\supset ( \ind_B^G 1)^{H_j} $ have  the same dimension  by Proposition \ref{pro:BK0}, hence are equal.  
  \begin{corollary}  Over any coefficient ring, 
 any element in $\ind_B^G 1$  fixed by $H'_j$ is also fixed by $H_j$ for $j\geq 0$. 
\end{corollary}

  It is known that any infinite dimensional irreducible smooth $R$-representation $\Pi$ of $G$  near the identity,  is isomorphic to $\ind_B^G 1$ modulo a multiple of the trivial representation  \cite{HV23}. There exist  integers $a_\Pi$ and $j_\Pi\geq 0$ such that  for $j\geq j_\Pi$
  \begin{equation}\label{cPi} \text{$\Pi \simeq a_{\Pi} 1 + \ind_B^G 1$ on 
  $I_{j} $}.
  \end{equation}

 \begin{corollary}\label{corcPi} For $j\geq j_\Pi$, any element in $\Pi$  fixed by $H'_j$ is also fixed by $H_j$.
 \end{corollary}

 \begin{proposition}\label{pro7.5} $a_\Pi=0$ if $\Pi$ is a principal series, 
 $a_\Pi=-1$ when $q+1 \neq 0$ in $R$ and $\Pi$ is the twist of the Steinberg representation by a character, and when $\Pi$ is cuspidal with minimal depth $\delta_\Pi$ under torsion by characters,
  $$a_\Pi=\begin{cases}-2 q^{\delta_\Pi }& \text{if  $\delta_\Pi$ is an integer}\\  - (q+1) q^{\delta_\Pi -1/2}&  \text{otherwise}
\end{cases} .$$
If $|L(\Pi)|=4$, then $a_\Pi=-2$ for $p$ odd and $a_\Pi$ is a multiple of $4$ if $p=2$.
 \end{proposition} 
\begin{proof}  When $R=\mathbb C$, 
then $a_\Pi $ is 
 the constant term $c_0(\Pi)$ of the germ expansion  for $ \Pi$ because   the constant term $c_0( \ind_B^G 1)$ of the germ expansion of the trace of $ \ind_B^G 1$  around $1$ \eqref{ge} is $0$. 
 
 When $R=\mathbb F_\ell^{ac}$ and $\tilde \Pi$ 
 is a  $\mathbb Q_\ell^{ac}$-representation lifting $\Pi$,   $a_\Pi =a_{\tilde \Pi }$.
  When $\Pi$ is cuspidal, $\tilde \Pi $ is supercuspidal and the formula for  $a_\Pi $ follows from \eqref{eq:fo}. If $|L(\Pi)|=4$ the assertion on $a_\Pi$ follows from the proof of Proposition \ref{prop:constant}
 \end{proof}
 
 In the particular case where $\Pi|_{G'}=\pi$ is irreducible, we deduce that  for $j\geq j_\Pi$: $$\pi \simeq  a_\Pi  1 + \ind_{B'}^{G'} 1  \ \ \text{on  $ I'_{j}$}.$$ 
  For example,  an irreducible   principal series $\pi$ of $G' $  is the restriction to $G'$ of a   principal series $\Pi$ of $G$, and on $I'_{1/2+j} $  for $j \geq j_\Pi$ we have
   $\pi \simeq   \ind_{B'}^{G'} 1$.   

\bigskip   By \eqref{cPi} if $j\geq j_\Pi$,
  \begin{equation}\label{8.6} \dim _{\mathbb C} \Pi ^{H_j}=a_\Pi  +   |B \backslash G/H_0| \, q^j .
 \end{equation}  
  By Proposition \ref{pro:BK0}, $\Pi^{H_j}=\sum_{\pi\in L(\Pi) }\pi^{H'_j}$ for  $H_j= I_{1/2+j}, K_{1+j} , I_{1+j}$ and $j\geq 0$. 

 In particular, if  $\Pi|_{G'}=\pi$ is irreducible,  then  if $j\geq j_\Pi$,
$$\ \dim \pi^{H_j '}= a_\Pi  +   |B \backslash G/H_0| \, q^j .$$
In general, by Corollary \ref{cor8.4}  (loc.cit. \S 12.2),  for  $j$ large \footnote{$j\geq j_\Pi+1$ for $I_j, H_j$ and $j\geq j_\Pi$ for $I_{1/2+j}$}
\begin{equation}\label{7.5}
 \dim _{\mathbb C} \Pi ^{I_{j}} =\dim _{\mathbb C} \Pi ^{I_{1/2+j}} =  a_\Pi  + 2   \ q^j, \ \ 
   \dim _{\mathbb C} \Pi ^{K_{1+j}} =  a_\Pi  + (q+1)   \ q^j .  \end{equation}
    Let $\pi$ be  an infinite dimensional  irreducible smooth $R$-representation  of $G'$ contained in $ \Pi|_{G'}$. The Moy-Prasad filtration of the Iwahori subgroup  $I'$ of $G'$ is 
 $$I'=I'_0\supset I'_{1/2} \supset I'_1 \supset \ldots \supset I'_j \supset I'_{1/2+j} \supset I_{j+1}\supset \ldots .$$
   \begin{theorem}\label{th.pol}  With $a_\Pi$ as in  \eqref{cPi} and Proposition \ref{pro7.5},  we have for 
  $j$ large \footnote{$j\geq j_\Pi+1$ for $I_j$ and $j\geq j_\Pi$ for $I_{1/2+j}$},
 $$ \dim_R \pi^{I'_j}=  \dim_R \pi^{I'_{1/2+j}}= | L(\Pi)|^{-1} \, (a_\Pi + 2 q^j) .$$
   $ | L(\Pi)|^{-1} a_\Pi=-1/2$ if  $| L(\Pi)|=4$ and   $p$ is odd,  otherwise 
  $ | L(\Pi)|^{-1} a_\Pi$  is an integer. 
   \end{theorem} 
  \begin{proof} The determinant of the $G$-normalizer $N_G(I)$ of the Iwahori group $I$  is equal to $F^*$ (\S \ref{7.1}). Thus,  $N_G(I)$ acts  transtively on $L(\Pi)$ and as $N_G(I)$ normalizes  the Moy-Prasad filtration of $I$.
  The dimension of the invariants of $\pi$  by  $I'_{1/2+j}$ and $I'_j$  of $G'$ for $j\geq 0$, does not depend on the choice of $\pi$ in the $L$-packet $L(\Pi)$.  For these two groups $H'_j$ we have
$ \dim_R \pi^{H'_j}=  | L(\Pi)|^{-1} \, \dim_R \Pi^{H_j}$ for $j\geq j_\Pi$, by Proposition \ref{pro:BK0}. Apply now \eqref{7.5}. The assertion on $  | L(\Pi)|^{-1}  a_\Pi$ follows from Proposition \ref{pro7.5}.
\end{proof}
Let us now turn to the asymptotics for fixed points under congruence subgroups $K'_j$  of $K'_0=SL_2(O_F)$.  
 The $G$-normalizer $ZK_0$ of $K_0=GL_2(O_F)$ normalizes the $K'_j$. The subgroup $H=ZK_0G'$  of $G$ has index $2$ as $\det (H)=(F^*)^2 O_F^* $ has index $2$ in $F^*$.  The restriction of $\Pi$ to $H$ has length $1$ or $2$.  
All the elements $\pi$ of $ L(\Pi)$ in the same $H$-orbit share the same dimension $\dim_R \pi^{K'_j}$. 
With $a_\Pi, j_\Pi$ as in   \eqref{cPi}, we deduce from \eqref{7.5}:

  \begin{theorem}\label{th.pol1} When $ \Pi|_H$ is irreducible,  we have for 
  $j\geq j_\Pi$,
 $$ \dim_R \pi^{K'_{j+1}}=   | L(\Pi)|^{-1} \, (a_\Pi + (q+1) q^{j}) .$$
     \end{theorem}

\begin{proposition} \label{H2} The representation $\Pi|_H$ is reducible if and only if $\Pi$ is cuspidal induced from $ZK_0$ or $\charf_R\neq 2$  and $\Pi$ is  a principal series $\ind_B^G \chi$ where $\chi_1\chi_2^{-1}=(-1)^{val}$.
\end{proposition}
\begin{proof} When $\Pi|_{G'}$ is irreducible, then $\Pi|_H$ is irreducible. 
When $\Pi= i_B^G(\chi)$  is a principal series of reducible restriction to $G'$, then $\charf_R\neq 2$, and  $i_B^G(\chi)|_H$ is reducible if and only if $(-1)^{val} \circ \det \otimes  i_B^G(\chi) \simeq  i_B^G(\chi)$ if and only if $\chi_1\chi_2^{-1}= (-1)^{val}$ (notations of  \S \ref{S:iBG} and $\chi=\chi_1\otimes \chi_2$).

When $\Pi$ is cuspidal, if $\Pi= \ind_{ZK_0}^G \lambda$ is induced from $ZK_0$ then $\Pi|_H$ is reducible because $ZK_0 \subset H$ and $ (\ind_H^G( \ind_{ZK_0}^H \lambda))|_H $ contains but is different from $\ind_{ZK_0}^G \lambda $.
If $\Pi $ is not  induced from $ZK_0$, then with the notations of \S \ref{S:331}, $\Pi =\ind_J^G\lambda $ has positive level, $E/F$ is ramified, and $G=JH$. As $J^1\subset H$  and the intertwining of $\lambda_1=\lambda|_{J^1}$ in $G$ is $J$, then the    intertwining of $\lambda_1 $ in $H$ is $J\cap H$. The vectors $\lambda_1$-equivariant in $\Pi$ are the functions supported in $J$. Applying \cite[Proposition 6.5 and Corollary 6.6]{HV22},
 $\Pi|_H=\ind _{J\cap H}^H \lambda |_{J\cap H}$ is irreducible.
\end{proof}

Assume now that $\Pi|_{H}$ is reducible.  Let $\Pi^+$  be the component having a Whittaker model  with respect to  a character $\psi$ non-trivial on $O_F$ but trivial on  $P_F$,  and $\Pi^-$ the other one. 

\begin{theorem}\label{th.pol2} When $ \Pi|_H$ is reducible,  we have for large $j$,
  $$ \dim_R (\Pi^+)^{K'_{j}}=   \frac{a_\Pi}{ 2 }+  q^{2 m+1} , \ \ \  \text{when} \ j=2m+1, 2m+2,$$
    $$ \dim_R (\Pi^-)^{K'_{j}}=   \frac{a_\Pi}{ 2 }+   q^{2 m} , \ \ \  \text{when} \ j=2m, 2m+1.$$
    
     \end{theorem} 
     \begin{proof} When $R=\mathbb C$,  the constant term in the germ expansion of the trace of $\Pi^+$ around the identity is $a_\Pi/2$  by \eqref{Pipi} and Remark \ref{re6.18},  and $ \dim_R (\Pi^+)^{K'_{j}}- a_\Pi/2  $ for large $j$, depends only on the characters of $F$ for which $\Pi^+$ has a Whittaker model. This set does not depend on the choice of $\Pi$, as 
  $\Pi^+$ has a  Whittaker model  only with respect to the characters   $\psi_{t_1 t_2^{-1}} $ for  $\diag (t_1,t_2)\in T\cap H$, that is,  $\psi_a$ for $a\in \det (H)$  where  $\psi_a(x)=\psi(ax)$  for  $x\in F$.  
   By the usual arguments, the same is true for any  $R$.
     It suffices to prove the theorem for  $\Pi=\ind_{ZK_0}^G \lambda$ where $\lambda|_{K_0}$ is the inflation of a cuspidal representation $\lambda_0$ of $GL_2(k_F)$ (Proposition \ref{H2}). In this special case we will  show \begin{equation}\label{Pi+} \dim_R (\Pi^+)^{K'_{j}}=-1 +    q^{2 m+1}  \ \ \text{for   } \  j=2m+1, 2m+2, \ j\geq 1,\end{equation}
\begin{equation}\label{Pi-}\dim_R (\Pi^-)^{K'_{j}}=-1 +    q^{2 m}  \ \ \text{for   } \  j=2m, 2m+1, \ j\geq 1.\end{equation}
  Note that $a_\Pi=- 2$ (Proposition \ref{pro7.5}) and that  \eqref{Pi+}  implies \eqref{Pi-}  for $j\geq j_\Pi +1$,  
as
 $$\text{$ \dim_R (\Pi^+)^{K'_{j}}+  \dim_R (\Pi^-)^{K'_{j}}=a_\Pi + (q+1) q^{j-1}$ for $j\geq j_\Pi +1$}.  $$
The representation $\lambda_0$ is generic,  and it follows that $\Pi^+=\ind_{ZK_0}^H \lambda$ \cite[Proposition 1.6]{BH98}. Let $t=\begin{pmatrix}p_F & 0 \cr 0 & p_F^{-1}\end{pmatrix}$. The group $H=ZK_0G'$ is the disjoint union 
 $$H= \sqcup_{i\geq 0}\,  ZK_0 t^i K_0'.$$
   For $i\geq 0$, $j>0$  and $k\in K_0'$, consider the representation of $K'_j$ on the functions in  $\ind_{ZK_0}^H \lambda$  supported on the coset $ZK_0 t^i  kK_j'$. That it contains non-zero $K'_j$-fixed vectors does not depend on the choice of $k\in K'_0$, and it happens if and only if $t^{i} K_j'  t^{-i} \cap Z K_0 $  has non-zero fixed vectors in $\lambda$. For $j\leq 2i$, $t^{i} K_j'  t^{-i} \cap Z K_0 $  contains the lower unipotent sugbgroup of $K_0$ and fixes no non-zero vector   in $\lambda_0$ which  is cuspidal. For $j>2i$, $t^{i} K_j' t^{-i}\subset K_1$ and $K_1$ acts trivially on $\lambda_0$. So the   space of functions  in $\ind_{ZK_0}^H$ supported in $ZK_0 t^i  kK_j'$ and fixed by $K'_j$ has dimension  $0$   if $j\leq 2i$ and $q-1=\dim_R \lambda_0$ if $j>2i$.  The number of cosets $ZK_0 t^i  kK_j'$ in $ZK_0 t^i  K_0$ is the index in $K'_0/K'_j$ of the image of $t^{-i} ZK_0 t^i \cap K'_0$ in  $K'_0/K'_j$. As    $K'_{2i} \subset t^{-i} Z K_0 t^i \cap K'_0$, this index does not depend on $j$ when $j>2i$. It is the index in $K'_0$ of $t^{-i} ZK_0 t^i \cap K'_0= \{\begin{pmatrix}a&b\cr c & d\end{pmatrix} \in K'_0, c\in P_F^{2i}\}$.   One computes its value to be $1$ if $i=0$ and $(q+1) q^{2i-1}$ if $i>0$. Consequently for $j>0$,
 $$\dim_R( \Pi^+)^{K'_{j}} = (q-1) [1+ \sum_{0<i<j/2}   (q+1) q^{2i-1}.$$
This is equal to  $q-1$ for $j=1,2$, to $(q-1)(q^2+q+1)= -1+ q^3$ for $j=3,4$, and by induction to $-1+ q^{2m+1}$ for $j= 2m+1, 2m+2$, implying   \eqref{Pi+},  
hence the theorem.

To prove  \eqref{Pi-}  for $j\geq 1$, one  can work in the same manner as above  using that $\Pi^-$ is the conjugate of $\Pi^+$ by $\begin{pmatrix}p_F &0\cr 0& 1\end{pmatrix}$. We find that $\dim_R( \Pi^-)^{K'_{j}} $  is equal to $0$ for $j=1$, to $-1+q^2$ for $j=2,3$, and to $-1+q^{2m}$ for $j=2m, 2m+1$,  implying   \eqref{Pi-}.
  \end{proof}
  
  \begin{corollary}\label{cor.pol} When $\Pi|_H$ is reducible, 
  we have   for  large $j$,
$$\dim_R \pi^{K'_j}=
 \begin{cases} |L(\Pi)|^{-1} (a_\Pi + 2 q^j)  &  \text{ for $j$ odd and $\pi \subset \Pi^+|_{G'}$, or $j$ even and $\pi \subset \Pi^-|_{G'}$}\\
L(\Pi)|^{-1} (a_\Pi + 2 q^{j-1}) &  \text{ otherwise}
 \end{cases}.$$
  \end{corollary}
 For the maximal compact group  $dK_0d^{-1}$ of $G'$, the two asymptotics are interchanged.

\bigskip   We find remarkable that the regularity is obtained when increasing the index $j$ by $2$, and not by $1$ as was the case for the Iwahori or the pro-$p$ Iwahori subgroups. But that could have been anticipated, given the homogeneity properties of the nilpotent orbital integrals in $H$.
  
  \begin{remark} It is likely that the asymtotics (Theorems \ref{th.pol} and  \ref{th.pol1}, Corollary \ref{cor.pol}) are valid when $2j\geq c$ where $c$ is the conductor of $\Pi$.  When $R=\mathbb C$ and  $\Pi$ is cuspidal, this is actually true for $\dim_{\mathbb C} \Pi^{K_j}$  and can be derived from the formulas in \cite{MY22}. When $p$ is odd, Nevins has completely analysed 
  the restriction to $K'_0$ of the irreducible smooth complex representations of $G'$, and we presume that the asymptotics (and for which $j$ it is valid) can be derived from her results  \cite{N05}, \cite{N13}.  \end{remark}

   \section{Appendix  -  The finite group $SL_2(\mathbb F_q)$}
Let   $k$ be a finite field  of characteristic $p$ with $q$ elements.
   In this appendix we classify irreducible  representations of $G=GL_2(k)$ and of $G' =SL_2(k)$ over  an algebraically closed field $R$ of characteristic $0$ or $\ell>0, \ \ell\neq p$.
We could use \cite{Bon11} for  $\charf_R\neq 2$ and 
\cite{KT09} for any $R$, but we prefer using the same methods as in the main text.

Note that  the irreducible $R$-representations of the finite groups $G$ and $G'$ are defined over the algebraic closure of the prime field, and we can freely pass from  $R $ to any other algebraically closed field of the same characteristic.
Thus it is enough to consider the cases where $R=\mathbb C$ or $R=\mathbb F_\ell^{ac}$.
We aim also to prove the following theorem.

 \begin{theorem} \label{th:fini}  Any irreducible $\mathbb F_\ell^{ac}$ representation $\sigma$ of $GL_2(k)$ is the reduction modulo $\ell$ of a $\mathbb Q_\ell^{ac}$-representation $\tilde \sigma$ of $GL_2(k)$ 
 such that  $\tilde \sigma|_{SL_2(k)}$ and 
$ \sigma|_{SL_2(k)}$ have the same length.

 Any irreducible $\mathbb F_\ell^{ac}$-representation of $SL_2(k)$  is the reduction modulo $\ell$ of a $\mathbb Q_\ell^{ac}$-representation  of $SL_2(k)$.
 \end{theorem}
 
 Write $Z$ for the centre of $G$, $B$ for the upper triangular subgroup of  $G$, and $U$ for its unipotent radical. 
  Let us first recall the known classification of the $R$-representations of $G$ (\cite{BH02} for $R=\mathbb C$  and  \cite{V88} for $R=\mathbb F_\ell^{ac}$). 

The parabolically induced representation $\ind_B^G(1) $ realised by the space of constant functions on $B\backslash G$,  contains  the trivial character.  It also has the trivial character as a quotient, given by the functional $\lambda$ which sums the values of functions on $B\backslash G$.
The map from the trivial subrepresentation to the trivial quotient is multiplication by $q+1$, so is an isomorphism if $\ell$ does not divide $q+1$, and is $0$ otherwise. In the first case the quotient $\St =\ind_B^G(1)/1$ is irreducible, in the second case $Ker(\lambda )/1$  is a  cuspidal but not supercuspidal representation $\sigma_0$ of $G$.

The irreducible (classes of) $R$-representations $\sigma$  of $G$ are :

1) The characters $\chi \circ\det$ where $\chi$ is an $R$-character of $ k^*$.

2) When $q+1\neq 0$ in $R$, the twists $(\chi \circ\det)\otimes \St$ of $\St $ by  the $R$-characters  $ \chi \circ\det$ of $G$.

2') When  $q+1=0$ in $R$, the twists  $(\chi \circ\det)\otimes \sigma_0$ of $\sigma_0$  
by  the $R$-characters  $ \chi \circ\det$ of $G$.

3) The irreducible principal series $\ind_B^G(\chi_1\otimes  \chi_2)$, where $\chi_1 $ and $\chi_2 $  are two distinct $R$-characters of $k^*$.

4) The supercuspidal representations $\sigma(\theta)$, where $\theta $ is an $R$-character of  $k_2^*$,  $\theta\neq \theta^q$, where $k_2/k$ is a quadratic extension.  

The only isomorphisms between those representations are given by exchanging  $\chi_1 $ and $\chi_2 $  in 3),  $\theta $  and  $\theta^q$ in 4).

Twisting by an $R$-character $\chi \circ \det$ of $G$  has the obvious effect, for example sending $\theta$ to $(\chi \circ  N )\theta$ where $N(x)=x^{q+1} $ for $x\in k_2^*$  in 4).

 Any irreducible $R$-representation $\tau$ of $G'$ is contained in the restriction  $\sigma|_{G'}$ to $G'$ of an  irreducible $R$-representation $\sigma$ of $G$.  The representation $\sigma|_{G'}$  is semi-simple of  multiplicity $1$ and its irreducible components are $G$-conjugate. The stabilizer of  $\tau$  contains $ZG'$ and $G/ZG'$ is  isomorphic to $k^*/ ( k^*)^2$. We have  $|k^*/ ( k^*)^2|=1$ when $p=2$ and $|k^*/ ( k^*)^2|=2$ when $p$ is odd. Therefore  $\sigma|_{G'}$  is irreducible when $p=2$ and $\sigma|_{G'}$ has length $1$ or $2$ when $p$ is odd.

When $\charf_R\neq 2$,  
 the length  $\lg (\sigma|_{G'})$ of $\sigma|_{G'}$ is  the number of $R$-characters 
 $\chi$ of $k^*$ such that $(\chi \circ \det)\otimes \sigma \simeq \sigma$, so 
\begin{equation}\label{lgq} \lg (\sigma|_{G'})= \begin{cases} 2 \ & \text{ in case 3)    if  $(\chi_1/\chi_2)^2=1$  and  in case 4)  if $(\theta^{q-1})^2 =1$} \\
 1 \ & \text{ otherwise}
 \end{cases}.
 \end{equation}
The restrictions $\sigma_1|_{G'}, \sigma_2|_{G'}$  of two irreducible representations $\sigma_1 ,\sigma_2$ of $G$ are isomorphic if and only $\sigma_1,  \sigma_2$ are twists of each other by an $R$-character of $G$. Otherwise $\sigma|_{G'}, \sigma_2|_{G'}$ are disjoint.
So, we have a classification of the (isomorphism classes of) irreducible representations of $G'$ when $\charf_R\neq 2$.

 \begin{remark}\label{re:62}The restriction to $B$ of a cuspidal representation of $G$ is the Kirillov representation $\kappa$ of $B$ (the irreducible $R$-representation of B induced by any non-trivial $R$-character of $U$). The restriction of $\kappa$ to $U$ is the direct sum of all non-trivial $R$-characters of $U$. The group $B$ acts transitively on such characters, whereas $B'=B \cap G'$ acts with two orbits. It follows that the restriction of $\kappa$   to $B'$ has two inequivalent irreductible components. Consequently a cuspidal representation of $G$ restricts to $G'$ with length $1$ or $2$. \end{remark}

 Let $\ell $ be an odd prime number different from $p$. Let us consider the reduction modulo $\ell$ of the previous irreducibles $\sigma$ over $\mathbb Q_\ell^{ac}$ (since $G$ is finite they  are integral).
For an integral  $\mathbb Q_\ell^{ac}$-character $\chi$ (with values in $\mathbb Z_\ell^{ac}$) let $\overline \chi$ denote its reduction modulo $\ell$. Reduction modulo $\ell$  is compatible with twisting by a $\mathbb Q_\ell^{ac}$-character $\chi \circ \det$ in the sense that the reduction of $(\chi \circ \det) \otimes \sigma$ is the twist by $\overline \chi  \circ \det$  of the reduction of $\sigma$.

1) The trivial $\mathbb Q_\ell^{ac}$-character of $G$ reduces to the trivial $\mathbb F_\ell^{ac}$-character. 

2) When $\ell$  does not divide $q+1$, the   Steinberg $\mathbb Q_\ell^{ac}$-representation reduces to the   Steinberg $\mathbb F_\ell^{ac}$-representation. 

2') When  $\ell$  divides $q+1$, the  Steinberg $\mathbb Q_\ell^{ac}$-representation reduces to a length $2$ representation with subrepresentation $\sigma_0$ and trivial quotient (for the natural integral structure). 

3) The irreducible principal series $\ind_B^G(\chi_1\otimes  \chi_2)$
reduces to the irreducible principal series $\ind_B^G(\overline \chi_1\otimes \overline\chi_2)$  when  $\overline \chi_1\neq \overline \chi_2$, and to $( \overline \chi_1\circ \det) \otimes \ind_B^G(1)$ (of length $2$ when $\ell$  does not divide $q+1$, and length $3 $ otherwise) when $\overline \chi_1= \overline \chi_2$  (for the natural integral structure).

4) The   supercuspidal  $\mathbb Q_\ell^{ac}$-representation  $\sigma(\theta)$, reduces to the supercuspidal  $\mathbb F_\ell^{ac}$-representation $\sigma(\overline \theta)$ if $\overline \theta \neq (\overline \theta)^q=
\overline {\theta^q}$,  and  otherwise
(which can happen only if $\ell$ divides $q+1$) to $(\eta \circ \det)\otimes \sigma_0$ where $\eta$ is the
 $\mathbb F_\ell^{ac}$-character of   $\mathbb F_{q}^*$ such that $\eta  \circ N=\overline \theta$.

A given  $\mathbb F_\ell^{ac}$-character  of $k^*$ or $k_2^*$ has a unique lift  to  a  $\mathbb Z_\ell^{ac}$-character of the same order, and from the above it is clear that any irreducible $\mathbb F_\ell^{ac}$-representation $\sigma$ of $G$ lifts to a  $\mathbb Q_\ell^{ac}$-representation. Moreover,  one can choose a lift of $\sigma$ with the same length on restriction to $G'$, thus proving the theorem when $\ell$ is odd.

\bigskip Let us finally  assume $\charf_R=2$. 
Then $p$ is odd and $q+1=0$ in $R$. Write $q-1=2^s m$ with a positive integer $s$ and an odd integer $m$. 
The number of irreducible   $R$-representations of $G$ (resp. $ZG'$) is the number of conjugacy classes in $G$  (resp. $ZG'$) of elements of odd order. 
Let  $g\in G$ of odd order. Then $\det (g)\in k^*$ has odd order so $\det (g)\in  (k^*)^2$  and $g\in  ZG'$. The $G$-conjugacy class of $g$ is equal to  its $ZG'$-conjugacy class  unless the $G$-centralizer of $g $ is entirely in $ZG'$. 
In that exceptional case, the $G$-equivalence class of $g$ is the union of two $ZG'$-equivalence classes. 
This    happens only when $g=zu $ where $z\in Z$ (of odd order) and $u\neq 1$ is unipotent. That shows that $m$ is the number of $ZG'$-conjugacy classes of elements of odd order  minus the number of $G$-conjugacy of such elements. Consequently $m$ is the number of irreducible $R$-representations of $ZG' $  minus the number of irreducible $R$-representations of $G$.
  
Consider first  $\sigma(\theta)$ for  a $\mathbb Q_2^{ac}$-character $\theta$ of $k_2^*$ of order $2^{s+1}$. Certainly $\overline \theta$  is trivial so that the reduction of $\sigma(\theta)$ modulo $2$  is  $\sigma_0$.
But $\ell ( \sigma(\theta)|_{G'})=2$ by \eqref{lgq}, from which it follows that $\ell(\sigma_0|_{G'})\geq 2$. We have seen however that $\ell(\sigma_0|_{G'})\leq 2$ (Remark \ref{re:62}), so  $\ell(\sigma_0|_{G'})=2$, and each irreducible component of $\sigma_0|_{G'}$  lifts to an irreducible component of  $\sigma(\theta)|_{G'}$.  The  $\mathbb F_2^{ac}$-characters $\chi$ of $k^*$ have an odd order, their number is $m$, and the  representations $(\chi\circ \det) \otimes \sigma_0$  are not equivalent (the order of $\chi$ is odd). We deduce:

\begin{lemma} \label{le:53} All irreducible $\mathbb F_2^{ac}$-representations of $G$ restrict  irreducibly to $G'$  except the twists of $\sigma_0$ by characters.

The reduction modulo $2$ of any supercuspidal $\mathbb Q_2^{ac}$-representation of $G'$ is irreducible.\end{lemma}
 
We deduce the classification of  irreducible $R$-representations of $G'$ when $\charf_R=2$ and  
Theorem \ref{th:fini}  when $\ell=2$.

 \begin{remark} \label{re:redf} For use in the main text we summarize:

a) When $q+1=0$ in $R$, then  $\sigma_0|_{SL_2( k)}$ is irreducible if $\charf_R\neq 2$,  and has length $2$
 if $ \charf_R= 2$.

b) In 4) let $b\in k_2$ be an element of order $q+1$. 
We have  $\theta \neq \theta^q  \Leftrightarrow \theta(b) \neq 1$ and 
 $\sigma(\theta)|_{SL_2( k)}$ is irreducible if  $\theta^2(b)\neq 1$, and has length $2$ if  $\theta^2(b)=1$.
 
 When $\charf_R=2$, or when $p=2$ hence $(2, q+1)=1$, we have $\theta(b) \neq 1 \Leftrightarrow \theta(b^2) \neq 1 $ 
 hence  $\sigma(\theta)|_{SL_2( k)}$ is irreducible for all $\theta \neq \theta^q$.
 
 When $\charf_R\neq 2$  and $p$ is odd,   there exists $\theta $ such that $\theta(b) \neq 1 , 
 \theta(b)^2 =1$, unique modulo the twist  by a character $\chi$ such that $\chi(b)=1$.
  The corresponding representations $\sigma(\theta)$ of $G$ are twists of each other by a character of $G$. 
Their restrictions to $SL_2( k)$ are isomorphic and reducible of length $2$.
\end{remark}


\begin{thebibliography}{} 
      \bibitem[Abdellatif-11]{Abde} Ramla Abdellatif -- Autour des repr\'esentations modulo $p$ des
groupes r\'eductifs $p$-adiques de rang $1$.  Th\`ese, D\'epartement de Math\'ematiques d'Orsay, 2011.

 
 
  \bibitem[Assem94]{As94}  M. Assem -- The Fourier transform and some character formulae for   $p$-adic $SL_\ell$,   $\ell$ a prime. Amer. J. Math. 116 (1994), no. 6, pp. 1433-1467.
  
 \bibitem[Aubert-Baum-Plymen-Solleveld16]{ABPS16} A.-M. Aubert, P. Baum, R. J. Plymen and M. Solleveld  -- The local Langlands correspondence for inner forms of $SL_n$, Res. Math. Sci. 3 (2016), Paper no. 32, 34 pp.

\bibitem[Aubert-Mendes-Plymen-Solleveld17]{AMPS17} A.-M. Aubert, S. Mendes, R.Plymen, M.Solleveld --  On $L$-packets and depth for $SL_2(K)$ and its inner forms. International Journal of Number Theory Vol.13, No.10 (2017) 2545-2568.

 \bibitem[Aubert-Plymen24]{AP24} A.-M. Aubert, R.Plymen -- Endoscopic character identities for depth-zero supercuspidal
 representations of $\mathrm{SL}_2$.  ArXiv:2410.20183v2 (2024).
  
 
\bibitem[Bernstein-Zelevinski 77] {BZ77}J.Bernstein, A. Zelevinski -- Induced representations of reductive p-adic groups. I, Ann. Sci. \'{E}cole Norm. Sup. (4) 10 (1977), no. 4, 441-472.  

 
\bibitem[Bonnaf\'e11]{Bon11} C. Bonnaf\'e -- Representations of $SL_2(\mathbb F_q)$. Algebra and Applications, 13. Springer-Verlag London, (2011).

\bibitem[Borel91]{B91} Borel A. -- Linear algebraic groups, second enlarged ed., Grad. texts in Math., vol. 126, Springer-Verlag, New-York, 1991.
 
\bibitem[Borel-Tits72]{BT72} A. Borel, J. Tits --  Compl\'ements \`a l'article ``Groupes R\'eductifs'', Publ. Math. I.H.\'E.S 41 (1972),
253-276.

\bibitem[BourbakiA8] {BkiA8} N.~Bourbaki --  \'El\'ements de math\'ematiques.   Alg\`ebre, Chap.\ 8.  Modules et anneaux semi-simples. Springer, Berlin-Heidelberg  (2012).

 
\bibitem[Bushnell-Henniart98]{BH98} C. J. Bushnell, G. Henniart --  Supercuspidal representations of $GL_n$: explicit Whittaker functions, J. Algebra 209 (1), 1998, 270-287. 

  \bibitem[Bushnell-Henniart02]{BH02} C.J. Bushnell, G. Henniart  --   On the derived subgroups of certain
unipotent subgroups of reductive groups over infinite fields. Transform. Groups,
7(3):211-230, 2002.
 

\bibitem[Bushnell-Henniart06]{BH06} C. J. Bushnell, G. Henniart -- The Local Langlands Conjecture for $GL(2)$. Grundlehren der mathematischen Wissenschaften, volume 335. Springer 2006.


  \bibitem[Bushnell-Kutzko94]{BK94} C.J. Bushnell, P.C. Kutzko  --  The admissible dual of $\SL(N)$ $\mathrm{II}$. Proc. London Math. Soc. (3) 68 (1994), no. 2, 317-379.



 
\bibitem[Casselman73]{Cas73}  W.Casselman -- The Restriction of a Representation of $GL_2(k)$ to $GL_2(o)$.   Mathematische Annalen - 206 (1973) p.311 - 318. 

 
\bibitem[Cui23]{Cui23} Peiyi Cui -- Modulo $\ell$ representations of $p$-adic groups $SL_n(F)$: maximal simple k-types.  arXiv:2012.07492v2(2023).

\bibitem[Cui-Lanard-Lu24]{CLL23} Peiyi Cui, T. Lanard, Hengfei Lu  -- Modulo $\ell$ 
disntinction problems.  arXiv:2203.14788v4 (2024).

 
\bibitem[Dat09]{D09}   J.-Fr. Dat -- Finitude pour les repr\'esentations lisses de groupes $p$-adiques. J. Inst. Math. Jussieu, 8(2):261-333, (2009).

 
\bibitem[Dat-Helm-Kurinczuk-Moss23]{DHKM23}  J.-Fr. Dat, D.Helm,R.Kurinczuk,D.Moss -- Finiteness for Hecke algebras of $p$-adic groups. arXiv:2203.04929v2 (2023).
 
\bibitem[Feng23]{F23} T.Feng -- Modular functoriality in the local Langlands correspondence. ArXiv:2312.12542.

  \bibitem[Gelbart-Knapp82]{GeKn82} S.S. Gelbart, A.W. Knapp - $L$-indistinguishability and $R$-groups for the special linear
group. Adv. in Math. 43 (1982), 101-121.

 \bibitem[G\'erardin77]{G77} P.G\'erardin -- Weil representations associated to finite fields, Journal of Algebra 46 (1977), 54-101.
  
 


  


 \bibitem[Henniart01]{H01}  G. Henniart -- Repr\'esentations des groupes r\'eductifs $p$-adiques et de leurs sous-groupes distingu\'es cocompacts. Journal of Algebra 236, 236-245 (2001).
 
  \bibitem[Henniart-Vign\'eras19]{HV19} G. Henniart,  M.-F. Vign\'eras -- Representations of a $p$-adic group in characteristic $p$.  Proceedings of Symposia in Pure Mathematics Volume 101, 2019 in honor of J. Bernstein, 171-210. 
  
   \bibitem[Henniart-Vign\'eras22]{HV22} G. Henniart,  M.-F. Vign\'eras -- Representations of a reductive $p$-adic group in characteristic distinct from $p$.   
 Tunisian Journal of Mathematics Vol. 4 (2022), No. 2, 249-305.
 
  \bibitem[Henniart-Vign\'eras23]{HV23} G. Henniart,  M.-F. Vign\'eras -- Representations of $GL_n(D)$ near the identity. ArXiv 2305.06581 (2023).
  
  
 
 \bibitem[Hiraga-Saito12]{HS12} K. Hiraga, H. Saito  -- On $L$-packets for Inner Forms of $SL_n$. Memoirs of the A.M.S. No 1013, 2012.


\bibitem[Kleshchev-Tiep09]{KT09} A. S. Kleshchev, P. H.Tiep  --  Representations of finite special linear groups in non-defining characteristic. Adv. Math., Vol. 220 (2009), pp. 478-504.


\bibitem[Kutzko-Pantoja91]{KP91} P.C.Kutzko, J.Pantoja -- Compositio Mathematica, tome 79, no 2 (1991), p. 139-155

\bibitem[Labesse-Langlands79]{LL79} J.-P.Labesse, R.P.Langlands -- $L$-indistinguishability for $SL(2)$, Can. J. Math., Vol.XXX1, No.4,1979, pp.726-785.

\bibitem[Lemaire04]{L04} B.Lemaire -- Int\'egrabilit\'e locale des caract\`eres tordus de $GL_n(D)$. J.Reine Angew.Math. 566 (2004), 1-39.

\bibitem[Lemaire05]{L05} B.Lemaire -- Int\'egrabilit\'e locale des caract\`eres  de $SL_n(D)$. Pacific J. of Math. 222 No 1  (2005), 69-131.


\bibitem[Lemaire19]{L19} B.Lemaire -- Appendix A - Central morphisms. In  J.-P.Labesse, J.Schwermer --  Central morphisms and cuspidal automorphic representations. J. Number Theory 205 (2019), 170-193.


\bibitem[Luo-Ngo24]{Ngo24} Z. Luo, B.C. Ngo   -- Nonabelian Fourier  Kernels on $SL_2$ and $GL_2$. arXiv:2409.14696, 2024.

  \bibitem[Miyauchi-Yamauchi22]{MY22} M.Miyauchi, T.Yamauchi -- A remark on conductor, depth and principal congruence subgroups. J. of Algebra, 592 (2022) 424-434.
  
   
\bibitem[Moeglin-Waldspurger87]{MW87} C. Moeglin, J.-L. Waldspurger. --  Mod\`eles de Whittaker d\'eg\'en\'er\'es
pour des groupes p-adiques. Math. Z. 196, 427-452 (1987).


\bibitem[Neukirch99]{Neu99}J. Neukirch -- Algebraic number theory. Grundlehren der mathematischen Wissenschaften,  volume 322, Springer-Verlag, Berlin, 1999.  

\bibitem[Nevins05]{N05} M. Nevins --Branching Rules for Principal Series Representations of $SL(2)$ over a p-adic Field. J. Math. Vol. 57 (3), 2005 pp. 648-672.

\bibitem[Nevins13]{N13} M. Nevins -- Branching rules for supercuspidal representations of $SL_2(k)$, for $k$ a $p$-adic field.  Journal of Algebra 377 (2013) 204-231.

\bibitem[Nevins23]{N23} M. Nevins -- The local character expansion as  branching rules: nilpotent cones and the case of $SL(2)$. ArXiv:2309.17213v2, 2023.

\bibitem[Platonov-Rapinchuk91]{PR91} V. Platonov, A.Rapinchuk -- Algebraic groups and Number Theory. Pure and applied math. vol 139, 1991.


  \bibitem[Rodier74]{R74} F. Rodier -- Mod\`ele de Whittaker et caract\`eres de repr\'esentations, in non commutative harmonic analysis, J. Carmona et M. Vergne, Lect. Notes Math. 466, pp. 151-71. Berlin-Heidelberg-New-York: Springer 1974.
  


\bibitem[Serre77]{S77} J.-P.Serre -- Modular forms of weight one and Galois representations. Algebraic Number Fields, edited by A. Fr\"ohlich, Acad.Press (1977), 193-268.

\bibitem[Shelstad79]{Sh79} D. Shelstad --Notes on L-indistinguishability, Proc. Symp. Pure Math. 33(2) (1979)
193-203.


\bibitem[Silberger79]{Sil79} A.J. Silberger -- Isogeny restrictions of irreducible admissible representations are finite direct sums of irreducible admissible representations. Proceedings of the A.M.S. Vol. 73, 2, 1979.
 

\bibitem[Springer98] {Spr98}  T.A. Springer  --  Linear Algebraic Groups Second Edition. Progress in Mathematics 9, Birkh\"auser 1998.

\bibitem[Tadic92] {T92} M. Tadic -- Notes on representations of non-archimedean $SL(n)$. Pacific J. of Math. 152, no.2 (1992).


\bibitem[Treumann-Venkatesh16]{TV16} D.Treumann,  A.Venkatesh -- Functoriality, Smith theory, and the Brauer homomorphism, Ann. of Math. (2) 183 (2016), no. 1, 177-228.
\bibitem[Varma14]{Va14} S.Varma -- On a result of Moeglin and Waldspurger in residual characteristic 2. Math. Z. (2014) 277:1027-1048.


 \bibitem[Vign\'eras88] {V88} M.-F. Vign\'eras  -- Repr\'esentations  modulaires de $GL(2,F)$ en caract\'eristique $\ell$, $F$ corps fini de caract\'eristique $p\neq \ell$. C.R.Acad.Sci.Paris, t.306, S\'erie I, p.451-454,  1988.

 
 \bibitem[Vign\'eras89] {V89} M.-F. Vign\'eras  -- Repr\'esentations  modulaires de $GL(2,F)$ en caract\'eristique $\ell$, $F$ corps $p$-adique, $p\neq \ell$. Compositio Mathematica, t.72, $n^o1$, 1989, p.33-66.

\bibitem[Vign\'eras96]{V96} M.-F. Vign\'eras  -- Repr\'esentations $\ell$-modulaires d'un groupe r\'eductif fini $p$-adique avec $\ell \neq p$. Birkhauser Progress in math. 137 (1996).

\bibitem[Vign\'eras97]{V97} M.-F. Vign\'eras  -- A propos d'une conjecture de Langlands modulaire. In Finite reductive groups, Marc Cabanes Ed., Progress in Math. 141, Birkh\"auser 1997.

 
\bibitem[Vign\'eras01] {V01} M.-F. Vign\'eras  -- Correspondance de Langlands semi-simple pour $GL(n,F)$ modulo $\ell\neq p$. Invent. math. 144, 177-223 (2001).

 
\bibitem[Vign\'eras22]{V22} M.-F. Vign\'eras  -- Representations of $p$-adic groups over commutative rings. Proc. Int. Cong. Math. 2022, Vol. 1, pp. 332-374. EMS Press, Berlin, 2023.


   \end{thebibliography}
\end{document}